\documentclass[10pt,reqno,oneside]{amsproc}
\title[The Free-boundary Euler equations]{A regularity result for the free boundary compressible\\ Euler equations of a liquid}
\allowdisplaybreaks

\author[L.~Li]{Linfeng Li}
\address{Department of Mathematics\\
University of California Los Angeles\\
Los Angeles, CA 90095}
\email{lli265@math.ucla.edu}

\usepackage{amsmath, amsthm, amssymb, enumerate, enumitem, bbm,multicol}
\usepackage{color}

\chardef\forshowkeys=0
\chardef\refcheck=1
\chardef\showllabel=0
\chardef\sketches=1

\ifnum\forshowkeys=1

\usepackage[color]{showkeys}
\usepackage{marginnote}
\fi

\usepackage{times}
\usepackage{graphicx}
\usepackage[usenames,dvipsnames,svgnames,table]{xcolor}
\usepackage[colorlinks=true, pdfstartview=FitV, linkcolor=blue, citecolor=blue, urlcolor=blue]{hyperref}

\usepackage{tikz}

\usepackage{color}
\usepackage{xcolor}

\usepackage{datetime}
\usepackage{fancyhdr,lastpage}

\chardef\coloryes=0 
\chardef\isitdraft=0 

\ifnum\isitdraft=1 
\def\eqref#1{({\ref{#1}})}                

\definecolor{mygray}{rgb}{.6, .6, .6}

\definecolor{refkey}{rgb}{.3,0.3,0.3}

\def\startnewsection#1#2{\section{#1}\label{#2}\setcounter{equation}{0}}   

\def\nnewpage{} 
\else

\def\startnewsection#1#2{\section{#1}\label{#2}\setcounter{equation}{0}}   

\def\eqref#1{({\ref{#1}})}    

\def\nnewpage{} 

\fi

\begin{document}
	
	\def\colr{\color{red}}
	\def\ques{{\colr \underline{??????}\colb}}
	\def\nto#1{{\colC \footnote{\em \colC #1}}}
	\def\fractext#1#2{{#1}/{#2}}
	\def\fracsm#1#2{{\textstyle{\frac{#1}{#2}}}}   
	\def\nnonumber{}
	\def\ges{\gtrsim}
	\def\les{\lesssim}
	
	\def\colr{{}}
	\def\colg{{}}
	\def\colb{{}}
	\def\colu{{}}
	\def\cole{{}}
	\def\colA{{}}
	\def\colB{{}}
	\def\colC{{}}
	\def\colD{{}}
	\def\colE{{}}
	\def\colF{{}}

	\def\bp{\bar{\partial}}
	\def\pa{\partial}
	\def\pt{\partial_t}
	\def\dd{w}
	\def\DD{D}
	
	\ifnum\coloryes=1
	
	\definecolor{coloraaaa}{rgb}{0.1,0.2,0.8}
	\definecolor{colorbbbb}{rgb}{0.1,0.7,0.1}
	\definecolor{colorcccc}{rgb}{0.8,0.3,0.9}
	\definecolor{colordddd}{rgb}{0.0,.5,0.0}
	\definecolor{coloreeee}{rgb}{0.8,0.3,0.9}
	\definecolor{colorffff}{rgb}{0.8,0.9,0.9}
	\definecolor{colorgggg}{rgb}{0.5,0.0,0.4}
	\definecolor{coloroooo}{rgb}{0.45,0.0,0}

	\def\colb{\color{black}}
	\def\colr{\color{red}}
	\def\cole{\color{coloroooo}}
	
	\def\colu{\color{blue}}
	\def\colg{\color{colordddd}}
	\def\colgray{\color{colorffff}}

	\def\colA{\color{coloraaaa}}
	\def\colB{\color{colorbbbb}}
	\def\colC{\color{colorcccc}}
	\def\colD{\color{colordddd}}
	\def\colE{\color{coloreeee}}
	\def\colF{\color{colorffff}}
	\def\colG{\color{colorgggg}}

	\fi
	\ifnum\isitdraft=1
	\chardef\coloryes=1 
	\baselineskip=17.6pt
	\pagestyle{myheadings}
	\def\const{\mathop{\rm const}\nolimits}  
	\def\diam{\mathop{\rm diam}\nolimits}    
	\def\rref#1{{\ref{#1}{\rm \tiny \fbox{\tiny #1}}}}
	\def\theequation{\fbox{\bf \thesection.\arabic{equation}}}
	
	\chead{}
	\rhead{\thepage}
	\def\nnewpage{\newpage}
	\newcounter{startcurrpage}
	\newcounter{currpage}
	\def\llll#1{{\rm\tiny\fbox{#1}}}
	\def\blackdot{{\color{red}{\hskip-.0truecm\rule[-1mm]{4mm}{4mm}\hskip.2truecm}}\hskip-.3truecm}
	\def\bluedot{{\colC {\hskip-.0truecm\rule[-1mm]{4mm}{4mm}\hskip.2truecm}}\hskip-.3truecm}
	\def\purpledot{{\colA{\rule[0mm]{4mm}{4mm}}\colb}}
	\def\pdot{\purpledot}
	\else
	\baselineskip=14.8pt
	\def\blackdot{{\color{red}{\hskip-.0truecm\rule[-1mm]{4mm}{4mm}\hskip.2truecm}}\hskip-.3truecm}
	\def\purpledot{{\rule[-3mm]{8mm}{8mm}}}
	\def\pdot{}
	\fi

	\def\textand{\qquad \text{and}\qquad}
	\def\pp{p}
	\def\qq{{\tilde p}}
	\def\KK{\mathcal{K}}
	\def\MM{\mathcal{M}}
	\def\II{\mathcal{I}}
	\def\JJ{\mathcal{J}}
		\def\LL{\mathcal{L}}
	\def\GG{\mathcal{G}}
	\def\bet{\kappa}
	\def\ema#1{{#1}}
	\def\emb#1{#1}
	\def\rhog{\rho_d}
	\def\rhol{\rho_u}
	
	\ifnum\isitdraft=1
	\def\llabel#1{\nonumber}
	\else
	\def\llabel#1{\nonumber}
	\fi
	
	\def\tepsilon{\tilde\epsilon}
	\def\epsilonz{\epsilon_0}
	\def\restr{\bigm|}
	\def\into{\int_{\Omega}}
	\def\intu{\int_{\Gamma_1}}
	\def\intl{\int_{\Gamma_0}}
	\def\tpar{\tilde\partial}
	\def\bpar{\,|\nabla_2|}
	\def\barpar{\bar\partial}
	\def\RR{\mathbb{R}}
	\def\TT{\mathbb{T}}
	\def\FF{F}
	\def\gdot{{\color{green}{\hskip-.0truecm\rule[-1mm]{4mm}{4mm}\hskip.2truecm}}\hskip-.3truecm}
	\def\bdot{{\color{blue}{\hskip-.0truecm\rule[-1mm]{4mm}{4mm}\hskip.2truecm}}\hskip-.3truecm}
	\def\cydot{{\color{cyan} {\hskip-.0truecm\rule[-1mm]{4mm}{4mm}\hskip.2truecm}}\hskip-.3truecm}
	\def\rdot{{\color{red} {\hskip-.0truecm\rule[-1mm]{4mm}{4mm}\hskip.2truecm}}\hskip-.3truecm}
	
	\def\tdot{\fbox{\fbox{\bf\color{blue}\tiny I'm here; \today \ \currenttime}}}
	\def\nts#1{{\color{red}\hbox{\bf ~#1~}}} 
	
	\def\ntsr#1{\vskip.0truecm{\color{red}\hbox{\bf ~#1~}}\vskip0truecm} 
	
	\def\ntsf#1{\footnote{\hbox{\bf ~#1~}}} 
	\def\ntsf#1{\footnote{\color{red}\hbox{\bf ~#1~}}} 
	\def\bigline#1{~\\\hskip2truecm~~~~{#1}{#1}{#1}{#1}{#1}{#1}{#1}{#1}{#1}{#1}{#1}{#1}{#1}{#1}{#1}{#1}{#1}{#1}{#1}{#1}{#1}\\}
	\def\biglineb{\bigline{$\downarrow\,$ $\downarrow\,$}}
	\def\biglinem{\bigline{---}}
	\def\biglinee{\bigline{$\uparrow\,$ $\uparrow\,$}}
	\def\ceil#1{\lceil #1 \rceil}
	\def\gdot{{\color{green}{\hskip-.0truecm\rule[-1mm]{4mm}{4mm}\hskip.2truecm}}\hskip-.3truecm}
	\def\bluedot{{\color{blue} {\hskip-.0truecm\rule[-1mm]{4mm}{4mm}\hskip.2truecm}}\hskip-.3truecm}
	\def\rdot{{\color{red} {\hskip-.0truecm\rule[-1mm]{4mm}{4mm}\hskip.2truecm}}\hskip-.3truecm}
	\def\dbar{\bar{\partial}}
	\newtheorem{Theorem}{Theorem}[section]
	\newtheorem{Corollary}[Theorem]{Corollary}
	\newtheorem{Proposition}[Theorem]{Proposition}
	\newtheorem{Lemma}[Theorem]{Lemma}
	\newtheorem{Remark}[Theorem]{Remark}
	\newtheorem{definition}{Definition}[section]
	\def\theequation{\thesection.\arabic{equation}}
	\def\cmi#1{{\color{red}IK: #1}}
	\def\cmj#1{{\color{red}IK: #1}}
	\def\linfeng{\rm \colr Linfeng:~} 
	\def\old{\rm \colu Linfeng:~} 
	\def\XX{\mathbf{X}}
	\def\LOT{{\rm LOT}}

	\def\sqrtg{\sqrt{g}}
	\def\OO{\tilde\Omega}
	\def\EE{{\mathcal E}}
	\def\lot{{\rm l.o.t.}}                       
	\def\endproof{\hfill$\Box$\\}
	\def\square{\hfill$\Box$\\}
	\def\inon#1{\ \ \ \ \text{~~~~~~#1}}                
	\def\comma{ {\rm ,\qquad{}} }            
	\def\commaone{ {\rm ,\qquad{}} }         
	\def\dist{\mathop{\rm dist}\nolimits}    
	\def\ad{\mathop{\rm ad}\nolimits}    
	\def\sgn{\mathop{\rm sgn\,}\nolimits}    
	\def\Tr{\mathop{\rm Tr}\nolimits}    
	\def\dive{\mathop{\rm div}\nolimits}    
	\def\grad{\mathop{\rm grad}\nolimits}    
	\def\curl{\mathop{\rm curl}\nolimits}    
	\def\Curl{\mathop{\rm Curl}\nolimits}    
	\def\det{\mathop{\rm det}\nolimits}    
	\def\supp{\mathop{\rm supp}\nolimits}    
	\def\re{\mathop{\rm {\mathbb R}e}\nolimits}    
	\def\wb{\bar{\omega}}
	\def\Wb{\bar{W}}
	\def\indeq{\quad{}}                     
	\def\indeqtimes{\indeq\indeq\indeq\indeq\times} 
	\def\period{.}                           
	\def\semicolon{\,;}                      
	\def\bfx{\mathbf{x}}
	\newcommand{\cD}{\mathcal{D}}
	\newcommand{\cH}{\mathcal{H}}
	\newcommand{\imp}{\Rightarrow}
	\newcommand{\tr}{\operatorname{tr}}
	\newcommand{\vol}{\operatorname{vol}}
	\newcommand{\id}{\operatorname{id}}
	\newcommand{\p}{\parallel}
	\newcommand{\norm}[1]{\Vert#1\Vert}
	\newcommand{\abs}[1]{\vert#1\vert}
	\newcommand{\nnorm}[1]{\left\Vert#1\right\Vert}
	\newcommand{\aabs}[1]{\left\vert#1\right\vert}

	\ifnum\showllabel=1
	\def\llabel#1{\label{#1}}
	\else
	\def\llabel#1{\notag}
	\fi
	
	\def\bcb{\begin{color}{blue}}
	\def\bcr{\begin{color}{red}}
	\def\ec{\end{color}}
	\def\pa{\partial}

	\begin{abstract}
	We derive a~priori estimates for the compressible free boundary Euler equations in the case of a liquid without surface tension. 
	We provide a new weighted functional framework which leads to the improved regularity of the flow map by using the Hardy inequality.
	One of main ideas is to decompose the initial density function.
	It is worth mentioning that in our analysis we do not need the higher order wave equation for the density. 
	\end{abstract}

	\maketitle
	\tableofcontents
\startnewsection{Introduction}{sec01}
In this paper, we study the vacuum free boundary problem for compressible fluids in $\mathbb{R}^3$ described by the compressible Euler equations
\begin{align}
	&
	\rho ( \pt u +  u \cdot\nabla u ) 
	+
	\nabla p
	=0
	\inon{in $\Omega(t)$},
	\label{Euler1}
	\\&
	\pt \rho +u \cdot \nabla \rho + \rho \dive u=0
	\inon{in $\Omega(t)$}
	,
	\label{Euler2}
\end{align}
where $\Omega(t)$ represents an open 
subset of $\mathbb{R}^3$ which is occupied by the fluid at the time $t$, whose boundary is denoted by $\partial \Omega(t)$.
In \eqref{Euler1}--\eqref{Euler2}, $u= u(t,x)$, $\rho = \rho (t,x)$, and $p=p(t,x)$ denote the velocity field, density, and pressure, respectively.
To close the system, we consider the equation of state 
\begin{align}
	p(\rho)
	=
	K_1 \rho^\gamma 
	-
	K_2
	\label{Euler3}
\end{align}
for polytropic fluids,
where $K_1, K_2 >0$ are constants and $\gamma> 1$ is the adiabatic constant (see~\cite{CF}). 
The system \eqref{Euler1}--\eqref{Euler3} is supplemented with the initial, kinematic, and vacuum boundary conditions
\begin{align}
	&
	\Omega(0)
	=
	\Omega
	\\&
	(u(0, x), \rho(0, x))
	=
	(u_0 (x), \rho_0(x)),
	\inon{$x\in\Omega$}
		\label{Euler7}
	\\& \mathcal V (\Gamma(t))= u(t,x) \cdot n(t), \inon{$x\in\Gamma(t)$}\label{Euler-kinematic}	
	\\&
	p(t,x)=0,
	\inon{$x\in\Gamma(t)$}
	,
	\label{Euler4}
\end{align}
where $\Gamma (t) \subset \partial \Omega (t)$ is the moving free boundary,  $n(t)$ is the outward unit normal vector to $\Gamma (t)$, and $\mathcal V (\pa\Omega(t))$ is the normal velocity of $\Gamma (t)$. 
In order to model a compressible liquid, we assume that the initial density is strictly positive everywhere in the domain.
This problem is known to be ill-posed (see \cite{E}) unless the Rayleigh-Taylor sign condition
\begin{equation}
-\frac{\partial p}{\partial n}
 \geq \nu >0
 \inon{on~$\Gamma (t)$}
 ,
	\label{PVC}
\end{equation}
holds for some constant $\nu>0$.
The condition \eqref{PVC} is a natural physical condition which indicates that the pressure, and thus also the density, is greater in the interior than on the boundary.

In recent years, there have been a lot of effort made toward understanding the well-posedness of this free-boundary problem.
For the incompressible Euler equations, the divergence of the velocity field is zero and the density is a fixed constant.
Additionally, the pressure is not determined by the equation of state but rather a Lagrange multiplier enforcing the divergence free condition.
For the existence theory, we refer the reader to 
\cite{BHL, C, KN, N, RS1, RS2, S, W1, W2, Y1, Y2} in various settings including irrotationality (i.e., $\curl u = 0$) and \cite{CS1, CS2, CL, DE, L3, SZ1, SZ2} for rotational fluids.

Concerning the compressible fluids, we distinguish the cases between a liquid and a gas.
The gas problem refers to the equation of state \eqref{Euler3} with the constant $K_2=0$, which combined with the boundary condition \eqref{Euler4} leads to the vanishing of the density on the moving boundary.
As a consequence, the system becomes degenerate along the vacuum boundary (\cite{CLS, JM2}) and the standard method of symmetrizable hyperbolic equations cannot be applied. We refer to \cite{CS3, IT, JM1, JM2, LXZ} and reference therein for the well-posedness results.
In the case of a liquid, a commonly used equation of state is \eqref{Euler3} with a positive constant $K_2>0$, which is the pressure law we treat in this paper.
As opposed to the vanishing of the density on the boundary in the case of a gas, the density remains strictly positive everywhere in the domain and thus the Euler equations are uniformly hyperbolic.
The existence theory was established by Lindblad  in \cite{L2} using the Nash-Moser construction.
Trakhinin provided a different proof for the local well-posedness in \cite{T}, using the theory of symmetric hyperbolic systems.
Then, Coutand, Hole, and Shkoller in \cite{CHS} proved the well-posedness using the vanishing viscosity method and the time-differentiated a~priori estimates. 
A recent work \cite{LZ} by Luo and Zhang established the local well-posedness of the compressible gravity water wave problem using the Alinhac's good unknowns.
For other local well-posedness results in various settings, we refer to \cite{DK1, DK2, DL, GLL, L1, LL}.

In this paper, we derive a~priori estimates for the compressible free boundary Euler equations without surface tension in the case of a liquid.
Using a decomposition of the initial density (see~\eqref{EQ184}--\eqref{EQ185}), we provide a new weighted functional framework for the liquid problem.
An application of the Hardy inequality then leads to the improved regularity of the flow map and hence Jacobian, which in turn closes the energy estimates.
Note that in \cite{CHS}, the higher-order regularity of the Jacobian is obtained using the wave equation of the density, which requires a high time derivative regularity of the initial data.
Here we provide a direct method and a different functional framework with a low time-differentiated regularity.
In addition, we assume that the mixed tangential-time derivative of the initial data is in some weighted $L^2$ space as opposed to inhomogeneoues Sobolev space of the time derivative used in \cite{CHS}. As a consequence, our regularity result and the regularity in \cite{CHS} do not imply each other.

Under the Rayleigh-Taylor sign condition, \eqref{EQ113} below, the initial density grows as the distance function in the inward normal direction.
Also, the initial density is strictly bounded from below.
In light of this, we decompose the initial density into two parts (see~\eqref{EQ184}--\eqref{EQ185}): one is degenerate along the boundary with the growth rate of the distance function and the other is uniformly bounded from below.
A careful analysis of each part leads to weighted and homogeneous structure of the functional space.
For the degenerate part, we deploy the Hardy inequality, Lemma~\ref{hardy}, and obtain an improved $H^{0.5}$ regularity of the flow map which in turn leads to the improved regularity of the Jacobian (see Section~\ref{sec06}).
The  part of the initial density
which is uniformly bounded from below
produces a positive energy contribution on the boundary, allowing us to close the estimate for the improved regularity using the elliptic estimate.

In the pure tangential energy estimates, the Jacobian and cofactor matrix enter as a highest order term which requires improved regularity of the Jacobian and flow map.
To overcome this difficulty, we uncover the improved regularity by using the Euler momentum equation and a div-curl elliptic estimate.
Namely, we introduce a new fractional $H^{1.5}$ Sobolev space so that the highest order term involving the Jacobian and the cofactor matrix can be controlled at the expense of two time derivatives and the above-mentioned improved regularity of the flow map.
As a result, the scaling between space and time is necessarily one to two at the pure tangential level.
For the energy estimates with at least two time derivatives, the top order term involving the Jacobian and cofactor matrix can be handled by using integration by parts in space and time.
The fact that $DJ$ and $\pt v$ are of the same regularity is essential to close the estimates.
Consequently, the scaling between space and time becomes one to one when at least two time derivatives are involved.
The normal derivative estimate of the solution is obtained by the elliptic regularity, which requires divergence, curl, time, and tangential components of the solution.

Another major difficulty lies in the energy estimate with pure time derivatives.
The boundary integrals of the tangential component of the flow map cannot be controlled directly by the normal component or the trace lemma.
Instead, we resort to the momentum equation restricted to the boundary (see \eqref{EQLJ}) and show that the normal component dominates the tangential component on the free boundary, at least for a short time interval due to the closeness of the cofactor matrix to its initial state.

The paper is organized as follows.
In Section~\ref{sec02}, we introduce the Lagrangian formulation of the free boundary problem and the assumptions on the domain and initial data.
Also, we state the main result, Theorem~\ref{T01}.
The a~priori estimate needed for the boundedness of the solution, stated in Proposition~\ref{Lapriori}, is proven in Section~\ref{sec08}.
In Sections~\ref{sec03}--\ref{sec08},
we restrict ourselves to the case $\gamma=2$ in \eqref{Euler3} for the equation of state, while in Section~\ref{sec09} we explain the modifications needed for the general case $\gamma > 1$.
In Section~\ref{sec03}, we provide curl estimates which shall be needed when using the div-curl lemma.
In Section~\ref{sec04}, we derive energy estimates of the tangential, time, and divergence components of the solution, while
in Section~\ref{sec05}, we provide the normal derivative estimates of the solution using the elliptic regularity.
In order to close the estimate,
we establish the improved regularity of the Lagrangian flow map and the Jacobian in Section~\ref{sec06} and the improved regularity of the curl in Section~\ref{sec07}.
At last, in Appendix~\ref{secA}, we recall some auxiliary lemmas used in the proof.

\startnewsection{Lagrangian formulation and the main result}{sec02}
\subsection{Lagrangian formulation}
The flow 
map $\eta(t,\cdot) \colon x\mapsto \eta (t,x) \in \Omega(t)$ associated with the fluid velocity, defined as a solution to
\begin{align*}
	&
	\pt \eta (t,x)= u (t, \eta(t,x))
	\comma
	t>0
	\comma
	x\in \Omega,
	\\
	&
	\eta(0, x)
	=
	x
	\comma
	x\in \Omega
	,
\end{align*}
denotes the location of a particle at time $t$ that is initially placed at a Lagrangian label $x$. 
We define the following Lagrangian quantities:
\begin{align*}
	&
	v(t,x)= u(t, \eta(t,x))
	\indeq \text{(Lagrangian~velocity)}
	,
	\\
	&
	f(t,x)= \rho(t, \eta(t,x))
	\indeq \text{(Lagrangian~density)}
	,
	\\
	&
	A
	=
	[D\eta]^{-1}
	\indeq \text{(inverse~of~deformation~tensor)}
	,
	\\
	&
	J= \det D\eta
	\indeq \text{(Jacobian~determinant)}
	,
	\\
	&
	a= J A
	\indeq \text{(transpose~of~cofactor~matrix)}
	.
\end{align*}
The notation $F,_k$ is used to represent $\partial F / \partial x_k$, i.e., the partial derivative of $F$ with respect to the Lagrangian variable $x_k$.
We adopt the Einstein's summation convention for repeated indices.
The Latin indices $i,j,k,l,m, r,s$ are summed from $1$ to $3$, while the Greek indices $\alpha, \beta$ are summed  from $1$ to $2$.
Without loss of generality, we set the constants $K_1=K_2=1$ in \eqref{Euler3}. 
Namely, the equation of state reads
\begin{align}
	p(\rho)
	=
	\rho^\gamma 
	-
	1
	.
	\label{EQ101}
\end{align}
Using the Lagrangian quantities, the equations \eqref{Euler1}--\eqref{Euler4} can be rewritten in a fixed domain $\Omega$ as
\begin{align}
	&
	f_t + fA^j_i v^i,_j=0
	\inon{in $[0, T] \times \Omega$}
	,
	\label{Euler5}
	\\
	&
	fv_t^i + A^k_i f^\gamma,_k = 0
	\inon{in $[0, T] \times \Omega$}
	,
	\label{Euler6}
	\\
	&
	f^\gamma
	=
	1
	\inon{on $[0, T] \times \Gamma$},
	\label{Euler10}
	\\
	&
	(\eta, v, f) 
	= (Id, u_0, \rho_0)
	\inon{in $\{t=0\} \times \Omega$}
	,
	\label{Euler8}
\end{align}
where $\Gamma= \Gamma(0)$ denote the initial vacuum free boundary.
Since $J_t = JA^j_i v^i,_j$ and $J(0)=1$, it follows that
\begin{align}
	f= \rho_0 J^{-1},
	\label{EQ560}
\end{align} 
which indicates that the initial density can be viewed as a parameter in the Euler equations.
Thus, using \eqref{EQ560}, the system \eqref{Euler5}--\eqref{Euler8} becomes
\begin{align}
	&
	\rho_0 v^i_t 
	+
	a^k_i (\rho_0^\gamma J^{-\gamma}),_k 
	= 0
	\inon{in~$[0, T]\times \Omega$}
	,
	\label{EQ01}
	\\
	&
	\rho_0 J^{-1}
	=
	1
	\inon{on~$ [0, T]\times \Gamma $}
	,
	\label{EQ150}
	\\
	&
	(\eta, v) 
	= (Id, u_0)
	\inon{in~$\{t=0\} \times \Omega$}
	.
	\label{EQ02}
\end{align}
Since $\rho_0 = 1$ on $[0, T]\times \Gamma $, it follows that
\begin{align*}
	J=1
	\inon{on~$[0,T] \times \Gamma$}
	.
\end{align*}
From \eqref{EQ01} we write equivalently
\begin{align}
	v_t^i
	+
	\gamma (\rho_0 J^{-1})^{\gamma-2}
	A^k_i (\rho_0 J^{-1}),_k
	=
	0
		\inon{in $[0, T]\times \Omega$}
	\label{EQ04}
\end{align}
and
\begin{align}
	v^i_t 
	-	\gamma (\rho_0 J^{-1})^{\gamma-2} \rho_0
	J^{-3}
	 a^k_i J,_k
	+
		\gamma (\rho_0 J^{-1})^{\gamma-2}
	\rho_0,_k
	a^k_i J^{-2}
	= 0
		\inon{in $[0, T]\times \Omega$}
	.
	\label{EQ200}
\end{align}
The three equivalent equations \eqref{EQ01} and \eqref{EQ04}--\eqref{EQ200} are being used for different purposes: 
\eqref{EQ01} is used for the energy estimates, \eqref{EQ04} is used for estimates of the vorticity, while \eqref{EQ200}  is used for the Jacobian and the normal derivative estimates.

\subsection{The reference domain}
To avoid the use of local coordinates charts, we assume the initial domain $\Omega \subset \mathbb{R}^3$ at time $t=0$ is given by 
\begin{align*}
	\Omega
	=
	\{
	(x_1, x_2, x_3):
	(x_1, x_2) \in \mathbb{T}^2, x_3 \in (0,1)
	\}
	,
\end{align*}
where $\mathbb{T}^2$ denotes the $2$-dimensional torus with length $1$.
At time $t=0$, the reference vacuum boundary is the top boundary $\Gamma= \{x_3 = 1\}$, 
while $N=(0,0,1)$ denotes the outward unit normal vector to $\Gamma$.
The reference bottom boundary $\{x_3 = 0\}$ is fixed with the boundary condition 
\begin{align*}
	u^3 = 0
	\inon{on $[0,T]\times \{x_3=0\}$.}	
\end{align*}
The moving vacuum free boundary is given by 
$
	\Gamma (t) 
	= 
	\{ \eta(t, x_1, x_2, 1): (x_1, x_2)\in \mathbb{T}^2\}
$.

\subsection{Notations}
Given a vector field $F$ over $\Omega$, we use $DF$, $\dive F$, and $\curl F$ to denote its full gradient, divergence, and curl.
Namely, we denote by
\begin{align*}
	&
	[DF]_j^i
	=
	F^i,_j
	\\
	&
	\dive F
	=
	F^r,_r
	\\
	&
	[\curl F]^i
	=
	\epsilon_{ijk}
	F^k,_j
	,
\end{align*}
where $\epsilon_{ijk}$ represents the Levi-Civita symbol.
We denote the Lie derivatives along the trajectory map $\eta$ by
\begin{align*}
	&
	[D_\eta F]^i_r
	=
	A^s_r F^i,_s
	\\
	&
	\dive_\eta F
	=
	A^s_r F^r,_s
	\\
	&
	[\curl_\eta F]^i
	=
	\epsilon_{ijk}
	A^s_j F^k,_s
	,
\end{align*}
which correspond to Eulerian full gradient, divergence, and curl represented in Lagrangian coordinates.
The symbol $\bp$ is used for tangential derivative $\bp = (\partial_1, \partial_2)$, while $D = (\partial_1, \partial_2, \partial_3)$ stands for full spatial derivative.

For integers $k\geq 0$, we denote by $H^k (\Omega)$ the standard Sobolev spaces of order $k$ with corresponding norms $\Vert \cdot \Vert_k$.
For real number $s\geq 0$, the Sobolev spaces $H^s (\Omega)$ and the norms $\Vert \cdot \Vert_s$ are defined by interpolation.
The $L^p$ norms on $\Omega$ are denoted by $\Vert \cdot\Vert_{L^p (\Omega)}$ or simply $\Vert \cdot \Vert_{L^p}$ when no confusion can arise.
The negative-order Sobolev spaces $H^{-s} (\Omega)$ are defined via duality, namely,
\begin{align*}
	H^{-s} (\Omega)
	:=
	(H^s (\Omega))'
	.
\end{align*}
Note that in our configuration the boundary $\partial \Omega$ has two parts: the reference free boundary $\Gamma=\{x_3=1\}$ and the fixed bottom boundary $\{x_3=0\}$.
We shall work on the Sobolev spaces on the reference free boundary $\Gamma$ and use the notations $H^s (\Gamma)$ and $|\cdot|_s$ for the Sobolev spaces and norms, with $s\geq 0$.
The negative-order Sobolev spaces associated to the boundary are defined analogously as the domain $\Omega$, while
the $L^p$ norms on $\Gamma$ are denoted by $\Vert \cdot\Vert_{L^p (\Gamma)}$.

\subsection{The initial density}
Under \eqref{PVC}, we assume that the initial pressure satisfies the physical sign condition 
\begin{align*}
	\frac{\partial p_0}{\partial N} 
	\le - \nu
	<0
		\inon{on~$\Gamma$}
	.
\end{align*}
From \eqref{EQ101} the above condition is equivalent to
\begin{align}
	\gamma \rho_0^{\gamma-1} 
	\frac{\partial \rho_0}{\partial N}
	\le 
	-\nu
	<0
	\inon{on~$\Gamma$}
	.
	\label{EQ113}
\end{align}
Since $\gamma >1$ and $\rho_0=1$ on $\Gamma$, the condition \eqref{EQ113} implies that $\rho_0 - 1$ grows like the distance function to the reference boundary $\Gamma$.
Throughout this paper, we assume that
\begin{align}
	\rho_0(x) 
	= 
	1 + w(x)
	,
	\inon{$x\in\bar{\Omega}$}
	,
	\label{EQ68}
\end{align}
where $w(x)=1-x_3$ is the distance function to the reference free boundary.
It is clear that the Rayleigh-Taylor sign condition \eqref{EQ113} is satisfied for any $\gamma >1$.
We introduce the decomposition $
\rho_0^\gamma (x)
= 
\rhog^\gamma (x) + \rhol^\gamma (x)
$, where
\begin{align}
	\rhog (x)
	=
	((1+\dd (x))^\gamma 
	-
	(1
	+
	\frac{\gamma\dd (x)}{2}))^{\frac{1}{\gamma}}
	,
	\inon{$x\in\bar{\Omega}$}
	\label{EQ184}
\end{align}
and
\begin{align}
	\rhol (x)
	=
(	1
	+
	\frac{\gamma\dd (x)}{2})^{\frac{1}{\gamma}}
		,
	\inon{$x\in\bar{\Omega}$}
	.
	\label{EQ185}
\end{align}
It is readily checked that $\rhog(x) \colon \bar{\Omega} \to [0, \infty)$ is well-defined and $\rhog = 0$ on $\Gamma$.
Moreover, we have that $\rhol \geq 1$ for $x\in \bar{\Omega}$.
The $\rhog$ part of the density is degenerate along the boundary, while the $\rhol$ part is uniformly bounded from below.

\subsection{The Jacobian and cofactor matrix}
We recall the following differentiation identities:
\begin{align}
	&
	\partial_t J 
	=
	a^s_r v^r,_s
	,
	\label{EQ20}
	\\
	&
	\partial_t a^k_i
	=
	 J^{-1} (a^s_r a^k_i - a^s_i a^k_r) v^r,_s
	.
	\label{EQ15}
\end{align}
Using \eqref{EQ20}--\eqref{EQ15} and the identity $a= JA$, we obtain
\begin{align}
	\partial_t A^k_i
	=
	-A^s_i A^k_r v^r,_s
	.
	\label{EQ21}
\end{align}
Note that the identities \eqref{EQ20}--\eqref{EQ21} become spatial-differentiation formulas by replacing $v^r,_s$ with $D \eta^r,_s$ on the right hand side, where $D = \partial_l$ with $l=1,2,3$.
Namely, we have
\begin{align}
	&
	D J 
	=
	a^s_r D
	\eta^r,_s
	,
	\label{EQ60}
	\\
	&
	D a^k_i
	=
	J^{-1} (a^s_r a^k_i - a^s_i a^k_r) D 
	\eta^r,_s
	,
	\label{EQ61}
	\\
	&
	D A^k_i
	=
	-A^s_i A^k_r D
	\eta^r,_s
	.
	\label{EQ62}
\end{align}
The cofactor matrix satisfies the Piola identity which reads
\begin{align}
	{a^k_i},_k
	=
	0
	,
	\label{piola}
\end{align}
for $i=1,2,3$.
By the definition of the cofactor matrix, we have that
\begin{equation}
	\begin{bmatrix}
		a^3_1 & a^3_2 &a^3_3
	\end{bmatrix}
	=
	\begin{bmatrix}
		\eta^2,_1 \eta^3,_2 - \eta^3,_1 \eta^2,_2
		&
		\eta^3,_1 \eta^1,_2 -\eta^1,_1 \eta^3,_2
		&
		\eta^1,_1 \eta^2,_2 - \eta^1,_2 \eta^2,_1
	\end{bmatrix}
	\label{EQ70}
	.
\end{equation}
Note that only tangential derivatives appear in each entry on the right hand side, which shall play an essential role in our analysis.

\subsection{Main result}
The system \eqref{EQ01}--\eqref{EQ02} admits a conserved physical energy 
\begin{align*}
	\mathcal{E}(t)
	=
	\int_\Omega 
	\left(
	\frac{1}{2}
	\rho_0 |v(t)|^2
	+
	\frac{1}{\gamma-1}
	\rho_0^\gamma J^{1-\gamma} (t)
	+
	J(t)
	\right)
	dx
	.
	\llabel{E:ENERGY1}
\end{align*}
A direct computation shows that formally $\frac{d}{dt} \mathcal{E} (t) = 0$, assuming sufficient differentiability of the solution. 
However, it is too weak to control the regularity of the evolving free boundary.
Instead, we introduce the higher order energy functional 
\cole
\begin{align}
	\begin{split}
	E(t )
	&
	=
	\Vert v\Vert_4^2
		+
	\Vert \pt v\Vert_3^2
	+
	\Vert \pt J\Vert_4^2
	+
	\Vert \bp^3 J\Vert_{1.5}^2
	+
	\Vert \bp^3 \eta\Vert_{1.5}^2
	+
	\Vert \bp^3 \curl_\eta v\Vert_{0.5}^2
		\\&\indeq
	+
	\sum_{l=2}^4
	\Vert \bp^{4-l} \pt^l \eta\Vert_{1.5}^2
	+
	\sum_{l=2}^5
	(
		\Vert \rhog^{\frac{\gamma}{2}} \bp^{5-l} \pt^l D\eta\Vert_0^2
		+
	\Vert \pt^l v\Vert_{5-l}^2
	+
	\Vert \pt^l J\Vert_{5-l}^2
	+
	\vert  \pt^l \eta\vert_{5-l}^2
	)
	+1
	.
	\label{EQfun}
\end{split}
 \end{align}
\colb
Note that the time derivatives at $t=0$ are defined iteratively by differentiating \eqref{EQ04} and evaluating at $t=0$.

The following theorem is our main result.
\cole
\begin{Theorem}
\label{T01}
Let $\gamma> 1$.
Suppose that $(\eta(t), v(t))$ is a smooth solution of \eqref{EQ01}--\eqref{EQ02} on some time interval $[0,T_0]$ and the initial data satisfies $E (0)<\infty$.
Then there exists a sufficiently small constant $T \in (0,T_0)$ such that $\rho (t) \geq 1/2$ in $\bar{\Omega} (t)$ for $t\in [0,T]$ and
\begin{align*}
	\sup_{t \in [0,T]}
	E (t)
	\leq
	\MM
	,
	\llabel{EQ140}
\end{align*}
where $\MM>0$ is a constant depending on the initial data.
\end{Theorem}
\colb

In this paper we restrict ourselves to derive a~priori estimates and thus a smooth solution is assumed to be given.
As a result, there is no need to state the compatibility conditions for the initial data.
Nevertheless, we refer the reader to \cite{CHS} where the solution is constructed using the tangentially smoothed approximate system.

We denote by $M_0 = P(E (0))$ some constant depending on the initial data, which may vary from line to line.
Throughout this paper, the notation $A\les B$  means that $A \leq C B$ for some constant $C>0$ which depends on~$M_0$ and $\gamma$.
The notation $c_{l,m} \in \mathbb{R}$ stands for some constant that depends on $l,m\in \mathbb{N}$, which may vary from line to line. In fact, the true value of $c_{l,m}$ is not essential in our analysis.

The a~priori estimate needed to prove Theorem~\ref{T01} is the following. 
\cole
\begin{Proposition}
\label{Lapriori}
Let $\delta \in (0,1)$.
There exists a constant $T >0$ and
a nonnegative continuous function $P$ such that 
\begin{align}
	E(t)
	\les
	1+\delta \sup_{t \in [0,T]}E(t)
	+
	C_\delta tP(\sup_{t \in [0,T]}E(t))
	,
	\label{EQ498}
\end{align}
for all $t \in [0,T]$, where $C_\delta>0$ is a constant that depends on $\delta$.
We emphasize that the implicit constant used in \eqref{EQ498} is independent of $\delta$.
\end{Proposition}
\colb

For the proof of Theorem~\ref{T01} given Proposition~\ref{Lapriori}, we take $\delta>0$ sufficiently small and use the Gronwall inequality and the standard continuity argument.
Throughout Sections~\ref{sec03}--\ref{sec08}, we set $\gamma=2$ in \eqref{EQ01}, \eqref{EQ04}--\eqref{EQ200}, \eqref{EQ184}--\eqref{EQ185}, and \eqref{EQfun}.
We shall prove that the a~priori bound \eqref{EQ498} in Proposition~\ref{Lapriori}, thus completing the proof of Theorem~\ref{T01}.
In proving Proposition~\ref{Lapriori}, we assume the following:
\begin{enumerate}[label=(\roman*)]
	\item 	$\frac{1}{2} \leq  J\leq \frac{3}{2}$,
	\item $\Vert \eta \Vert_4 \leq C + T \Vert \pt \eta \Vert_4 \leq C$,
	\item $\Vert J \Vert_4 \leq C + T \Vert \pt J\Vert_4 \leq C$,
	\item $\Vert \pt J\Vert_3 \leq C + T \Vert \pt^2 J\Vert_3
	\leq C$,
	\item $\Vert \pt^2 v\Vert_2 \leq C +T \Vert \pt^3 v\Vert_2 \leq C$,
	\item  $	\Vert A- I_3 \Vert_{H^2}+
	\Vert a  - I_3\Vert_{H^2}
		\leq CT$,
\end{enumerate}
on $[0,T] \times \Omega$ for some constant $C>0$, where $I_3$ is the three-dimensional identity matrix. 
The above assumptions can be justified by using the fundamental theorem of calculus and taking $T>0$ sufficiently small, after we establish the a~priori bounds.

\startnewsection{Curl estimates}{sec03}
The following lemma provides the curl estimates of the solutions $\eta$ and $v$.
\begin{Lemma}
\label{Lcurl}
For $\delta\in (0,1)$, we have
\begin{align}
	\begin{split}
		&
		\sup_{t\in [0,T]}
		\sum_{l=2}^5
		\Vert \rhog
		\bp^{5-l} \pt^{l} \curl  \eta (t) \Vert_0^2
		+
		\sup_{t\in [0,T]}
		\sum_{l=2}^4
		\left(
		\Vert \pt^{l} D^{4-l}\curl v (t)\Vert_0^2
		+
		\Vert \bp^{4-l} \pt^l \curl \eta (t)\Vert_{0.5}^2
		\right)
		\\&\indeq
		+
		\sup_{t\in [0,T]}
		(\Vert \bp^3 \curl \eta (t)\Vert_{0.5}^2
		+
		\Vert D^3 \curl v (t)\Vert_0^2)
		\les
		C_\delta
		+
		\delta \sup_{t\in [0,T]} E(t)
		+
		C_\delta T P(\sup_{t\in [0,T]} E(t))
		.
		\label{EQ95}
	\end{split}
\end{align}
where $C_\delta>0$ is a constant depending on $\delta$.
\end{Lemma}

\begin{proof}[Proof of Lemma~\ref{Lcurl}]
We start with the estimates of $\Vert  \pt^{l} D^{4-l} \curl v\Vert_0^2$ for $l=2, 3, 4$.
Applying the Lagrangian curl to \eqref{EQ04}, we get
\begin{align}
		\curl_\eta v_t
	=
	0
	.
	\label{EQ52}
\end{align}
From \eqref{EQ21} and \eqref{EQ52} it follows that 
\begin{align}
	[\curl_\eta v]^k_t
	=
	\epsilon_{kji} (A^s_j)_t v^i,_s
	=
	-
	\epsilon_{kji}
	A^m_j A^s_r v^r,_m
	v^i,_s
	=: Q_0(A, Dv)
	,
	\label{EQ55}
\end{align}
where $Q_0$ is a quadratic function of $A$ and $Dv$. 
In \eqref{EQ55} and below,
we use $[F]^k$ to denote the $k$-th component of a vector field $F$.
Using the fundamental theorem of calculus, we obtain
\begin{align}
	[\curl_\eta v (t)]^k
	= 
	[\curl u_0]^k
	+
	\int_0^t 
	Q_0(A(t'), Dv(t'))
	dt'
	.
	\label{EQ53}
\end{align}
Fix $l \in \{2,3,4\}$.
Applying $\pt^{l} D^{4-l}$ to the above equation, we obtain
\begin{align*}
	\begin{split}
		[\curl_\eta \pt^{l} D^{4-l} v (t)]^k
	&
	= 
	\sum_{m=0}^{4-l}
	\sum_{\substack{n=0\\ m+n\geq 1}}^l
	c_{m,n}
	\epsilon_{kji} 
	\pt^n D^m  A^s_j 
	\pt^{l-n} 	D^{4-l-m}  v^i,_s
	+
	\pt^{l-1} D^{4-l}
	Q_0(A, Dv)
	.
	\llabel{EQ653}	
	\end{split}
\end{align*}
Inserting the identity $ [\curl_\eta \pt^{l} D^{4-l} v]^k =  [\pt^{l} D^{4-l} \curl v]^k + \epsilon_{k ji} \pt^{l} D^{4-l} v^i,_s \int_0^t (A^s_j)_t dt'$ to the above equation, we obtain
\begin{align*}
	\begin{split}
 	[\pt^{l} D^{4-l} \curl v]^k
	&
	= 
	\epsilon_{jki} \pt^{l} D^{4-l} 
	v^i,_s \int_0^t (A^s_j)_t dt'
	+
	\sum_{m=0}^{4-l}
	\sum_{\substack{n=0\\ m+n\geq 1}}^l
	c_{m,n}
	\epsilon_{kji} 
	\pt^n D^m  A^s_j 
	\pt^{l-n} 	D^{4-l-m}  v^i,_s
	\\&\indeq
	+
	\pt^{l-1} D^{4-l}
	Q_0(A, Dv)
	,
	\llabel{EQ655}
	\end{split}
\end{align*}
from where
\begin{align*}
	\begin{split}
	\Vert \pt^{l} D^{4-l} \curl v \Vert_0^2
	&
	\les
\underbrace{	\Vert
	\pt^{l} D^{5-l} v 
	(A-I_3)
	\Vert_0^2}_{\KK_1}
	+
\underbrace{	\Vert
	\pt^{l-1} D^{4-l}
	Q_0(A, Dv)
	\Vert_0^2}_{\KK_2}
		+
	\underbrace{	\Vert
		\pt^l D^{4-l}  A
		 D v
		\Vert_0^2}_{\KK_3}
		\\&\indeq
	+
	\underbrace{\sum_{m=0}^{4-l}
	\sum_{\substack{n=0\\ 3\geq m+n\geq 1}}^l
	\Vert
	\pt^n D^m  A
	\pt^{l-n} 	D^{4-l-m}  D v
	\Vert_0^2}_{\KK_4}
	.
	\llabel{EQ656}	
	\end{split}
\end{align*}
The term $\KK_1$ is estimated using the H\"older and Sobolev inequalities as
\begin{align*}
	\KK_1
	\les
	\Vert \pt^l D^{5-l} v\Vert_0^2
	\Vert A- I_3 \Vert_{L^\infty}^2
	\les
	TP(\sup_{t \in [0,T]} E(t))
	,
\end{align*}
since $\Vert \pt^l v\Vert_{5-l}^2 \leq E(t)$ for $l=2,3,4$ and $\Vert A- I_3\Vert_2 \les T$.
The highest order term (in terms of derivative counts on $\eta$) in $\KK_2$ can be written as
\begin{align*}
	- \epsilon_{kji}
	(A^m_j A^s_r v^r,_m 
	\pt^{l-1} D^{4-l}  v^i,_s 
	+
	A^m_j A^s_r  v^i,_s 
	\pt^{l-1} D^{4-l}  v^r,_m 
	)
	.
	\llabel{EQ642}
\end{align*}
Using the fundamental theorem of calculus, we obtain
\begin{align*}
	\Vert
	\pt^{l-1} D^{5-l} v 
	\Vert_0^2
	\les
	\Vert
	\pt^{l-1} D^{5-l} v (0)
	\Vert_0^2 
	+
	\int_0^T \Vert 
	\pt^{l} D^{5-l} v 
	\Vert_0^2
	\les
	1+TP(\sup_{t \in [0,T]} E(t))
	,
\end{align*}
where we used the Jensen's inequality.
From the H\"older and Sobolev inequalities it follows that  
\begin{align*}
	\Vert 	AADv
	\pt^{l-1} D^{5-l} v 
	\Vert_0^2
	\les
	1+
	TP(\sup_{t\in [0,T]} E(t))
	.
	\llabel{EQ672}
\end{align*}
where we drop the indices for simplicity.
The rest of the terms in $\KK_2$ are of lower order which can be treated in a similar fashion using the H\"older and Sobolev inequalities and the fundamental theorem of calculus.
Thus, we have
\begin{align*}
	\KK_2
	\les
	1+TP(\sup_{t \in [0,T]} E(t))
	.
\end{align*}
Similarly, the term $\KK_3$ is estimated as
\begin{align*}
	\KK_3
	\les
			\Vert
	\pt^l D^{4-l}  A
	\Vert_0^2 
	\Vert
	D v
	\Vert_{L^\infty}^2
	\les
	1+TP(\sup_{t \in [0,T]} E(t))
	.
\end{align*}
The term $\KK_4$ consists of essentially lower order terms which can be estimated analogously using the H\"older and Sobolev inequalities and the fundamental theorem of calculus, and we obtain
\begin{align*}
	\KK_4
	\les
	1+TP(\sup_{t \in [0,T]} E(t))
	.
\end{align*}
Consequently, we conclude
\begin{align}
	\sum_{l=2}^4
	\Vert \pt^{l} D^{4-l} \curl v (t)\Vert_0^2
	\les
	1+TP(\sup_{t \in [0,T]} E(t))
	.
	\label{EQ648}
\end{align}
The term $\Vert D^3 \curl v\Vert_0^2$ is treated in a similar fashion using the above arguments, and we arrive at
\begin{align}
	\Vert D^3 \curl v (t)\Vert_0^2
	\les
	1+TP(\sup_{t \in [0,T]} E(t))
	.
	\label{EQ650}
\end{align}

Now, we derive the estimates of $
\Vert \rhog
\bp^{5-l} \pt^{l} \curl  \eta \Vert_0^2$, where $l=2, 3, 4, 5$.
We rewrite \eqref{EQ53} as
\begin{align}
	[\curl \pt \eta (t)]^k
	=
	[\curl u_0]^k
	+
	\epsilon_{jki} v^i,_s
	\int_0^t
	(A^s_j)_t dt'
	+
	\int_0^t Q_0 (A, Dv) dt'
	.
	\label{EQ146}
\end{align}
Fix $l\in \{2,3,4,5\}$.
Applying $\rhog\bp^{5-l} \pt^{l-1}$ to the above equation, we arrive at
\begin{align*}
	\begin{split}
			&
	\Vert	\rhog \bp^{5-l} \pt^{l}
	 \curl \eta \Vert_0^2
	\les
\underbrace{	\Vert
	\rhog
	\bp^{5-l} \pt^{l-2}
	(Dv
	A_t)
	\Vert_{0}^2}_{\JJ_1}
		\\&\indeq
	+
	\underbrace{\Vert
	\rhog
	\bp^{5-l} \pt^{l-2}
	(\pt Dv
	(A- I_3)
	)
	\Vert_{0}^2}_{\JJ_2}
	+
\underbrace{	\Vert	\rhog 
	\bp^{5-l} \pt^{l-2}
	Q_0 (A, Dv)
	\Vert_{0}^2}_{\JJ_3}
	.
	\end{split}
\end{align*}
The highest order term in $\JJ_1$ scales like $\rhog \pt A \bp^{5-l} \pt^{l-1} D\eta$ which can be estimated using the fundamental theorem of calculus as
\begin{align*}
	\Vert \rhog \pt A \bp^{5-l} \pt^{l-1} D\eta\Vert_0^2
	\les
	\Vert \rhog \bp^{5-l} \pt^{l-1} D\eta (0)\Vert_{0}^2
	+
	\int_0^T
	\Vert \rhog  \bp^{5-l} \pt^{l} D\eta\Vert_0^2
	\les
	1+TP(\sup_{t \in [0,T]} E(t))
	,
\end{align*}
since $\Vert \rhog \bp^{5-l} \pt^l D \eta\Vert_0^2 \leq E(t)$, for $l=2,3,4,5$.
The rest of the terms in $\JJ_1$ are of lower order which can be treated in a similar fashion using the H\"older and Sobolev inequalities and the fundamental theorem of calculus.
Thus, we have
\begin{align*}
	\JJ_1
	\les
	1+TP(\sup_{t\in [0,T]} E(t))
	.
\end{align*}
The highest order term in $\JJ_2$ is estimated using the H\"older and Sobolev inequalities as
\begin{align*}
	\Vert \rhog \bp^{5-l} \pt^{l} D\eta
	(A-I_3)
	\Vert_0^2
	\les
		\Vert \rhog \bp^{5-l} \pt^{l} D\eta
		\Vert_0^2 
		\Vert A-I_3 \Vert_{L^\infty}^2
	\les
	TP(\sup_{t \in [0,T]} E(t))
	,
\end{align*}
since $\Vert A- I_3\Vert_2 \leq E(t)$.
The rest of the terms in $\JJ_2$, as well as the term $\JJ_3$, are treated analogously using the above arguments.
As a result, we get
\begin{align*}
	\JJ_2
	+
	\JJ_3
	\les
	1+TP(\sup_{t \in [0,T]} E(t))
	.
\end{align*}
Combining the above estimates, we conclude
\begin{align}
		\Vert	\rhog \bp^{5-l} \pt^{l}
	\curl \eta (t) \Vert_0^2
	\les
		1+TP(\sup_{t \in [0,T]} E(t))
		.
		\label{EQ649}
\end{align}

Next we derive the estimates of $\Vert  \bp^{4-l} \pt^l \curl \eta\Vert_{0.5}^2$ for $l=2,3,4$.
Fix $l\in \{2,3,4\}$.
Applying $\bp^{4-l} \pt^{l-1}$ to \eqref{EQ146}, we obtain
\begin{align*}
	\begin{split}
	\Vert	 \bp^{4-l} 
	\pt^l
	 \curl \eta	
	\Vert_{0.5}^2
	&\les
\underbrace{	\Vert
	\bp^{4-l} \pt^{l-2}
	(Dv
	A_t)
	\Vert_{0.5}^2}_{\II_1}
	+
	\underbrace{\Vert
	\bp^{4-l} \pt^{l-2}
	(\pt Dv
	(A- I_3)
	\Vert_{0.5}^2}_{\II_2}
		\\&\indeq
	+
	\underbrace{\Vert 
	\bp^{4-l} \pt^{l-2}
	Q_0 (A, Dv)
	\Vert_{0.5}^2}_{\II_3}
	.
	\llabel{EQ460}
	\end{split}
\end{align*}
For the term $\II_{1}$, from the Leibniz rule it follows that
\begin{align}
	\begin{split}
	\II_1
	&
	\les
	\Vert
	\bp^{4-l} \pt^{l-2}
	Dv
	A_t
	\Vert_{0.5}^2
	+
	\Vert
	Dv
	\bp^{4-l} \pt^{l-2}
	A_t
	\Vert_{0.5}^2
	\\&\indeq
	+
	\sum_{m=0}^{4-l}
	\sum_{\substack{n=0\\m+n=1}}^{l-2}
	\Vert
	\bp^{m} \pt^n
	Dv
	\bp^{4-l-m} \pt^{l-2-n}
	A_t
	\Vert_{0.5}^2
	=:\II_{11}
	+
	\II_{12}
	+
	\II_{13}
	.
	\label{EQ611}
	\end{split}
\end{align}
Recall the multiplicative Sobolev inequality
\begin{align}
	\Vert fg\Vert_r
	\les
	\Vert f\Vert_{1.5+\epsilon}
	\Vert g\Vert_r
	,
	\label{EQkap}
\end{align}
for $0\leq r \leq 1.5$ and $\epsilon>0$.
From \eqref{EQkap} it follows that
\begin{align*}
	\begin{split}
	\II_{11}	
	&
	\les
	\Vert \bp^{4-l} \pt^{l-1} D \eta\Vert_{0.5}^2
	\Vert A_t\Vert_{1.5+\epsilon}^2
		\les
	1+TP(\sup_{t \in [0,T]} E(t))
	\end{split}
\end{align*}
where we used the fundamental theorem of calculus and $\Vert \bp^{4-l} \pt^l \eta\Vert_{1.5}^2 + \Vert \pt^2 v\Vert_3^2 \leq E(t)$, for $l=2,3,4$.
Similarly, we have
\begin{align*}
	\begin{split}
		\II_{12}
		&
	\les
	\Vert \bp^{4-l} \pt^{l-1} A\Vert_{0.5}^2
	\Vert Dv\Vert_{1.5+\epsilon}^2
	\les
	1+TP(\sup_{t \in [0,T]} E(t))
	.
	\end{split}
\end{align*}
For the term $\II_{13}$, using \eqref{EQkap}, we get
\begin{align*}
\begin{split}
	\II_{13}
	&
	\les
	\Vert
	\bp 
	Dv
	\Vert_{1.5+\epsilon}^2 
	\Vert
	\bp^{3-l} \pt^{l-2}
	A_t
	\Vert_{0.5}^2
	\mathbbm{1}_{l=\{2,3\}}
	+
	\Vert
	\pt
	Dv
	\Vert_{1.5+\epsilon}^2
	\Vert
	\bp^{4-l} \pt^{l-3}
	A_t
	\Vert_{0.5}^2
	\mathbbm{1}_{l=\{3.4\}}
	\\&
	\les	
	(\delta \sup_{t \in [0,T]} E(t)
	+
	C_\delta
	+
	C_\delta T P(\sup_{t \in [0,T]} E(t)))
	(1+ TP(\sup_{t \in [0,T]} E(t)) )
	\\&
	\les
	C_\delta
	+
	\delta \sup_{t\in [0,T]} E(t)
	+
	C_\delta T P(\sup_{t\in [0,T]} E(t))
	,
\end{split}
\end{align*} 
where we used the Sobolev interpolation inequality and the fundamental theorem of calculus.
For the term $\II_2$, from the Leibniz rule it follows that
\begin{align*}
	\begin{split}
		\II_2
		&
		\les
		\Vert
		\bp^{4-l} \pt^{l}
		D\eta
		(A- I_3)
		\Vert_{0.5}^2
		+
		\Vert
		\pt Dv
		\bp^{4-l} \pt^{l-2}
		(A- I_3)
		\Vert_{0.5}^2
		\\&\indeq
		+
		\sum_{m=0}^{4-l}
		\sum_{\substack{n=0\\m+n=1}}^{l-2}
		\Vert
		\bp^{m} \pt^{n+1}
		Dv
		\bp^{4-l-m} \pt^{l-2-n}
		(A- I_3)
		\Vert_{0.5}^2
		=:\II_{21}
		+
		\II_{22}
		+
		\II_{23}
		.
		\llabel{EQ612}
	\end{split}
\end{align*}
We bound the term $\II_{21}$ using \eqref{EQkap} as
\begin{align}
	\begin{split}
		\II_{21}	
		&
		\les
		\Vert \bp^{4-l} \pt^{l} D \eta\Vert_{0.5}^2
		\Vert A- I_3 \Vert_{1.5+\epsilon}^2
		\les
		T P(\sup_{t \in [0,T]} E(t))
		,
		\label{EQ659}
	\end{split}
\end{align}
since $\Vert A- I_3 \Vert\les T$.
The terms $\II_{22}$, $\II_{23}$, and $\II_3$ are estimated analogously using the above arguments.
Thus, we have
\begin{align*}
	\II_{22}
	+
	\II_{23}
	+
	\II_3
	\les
	1+TP(\sup_{t \in [0,T]} E(t))
	.
\end{align*}
Consequently, we conclude
\begin{align}
	\begin{split}
	\sum_{l=2}^4
	\Vert	 \bp^{4-l} 
	\pt^l
	\curl \eta	
	\Vert_{0.5}^2
	\les
	1+
	T P( \sup_{t \in [0,T]} 	E(t))
	.
	\label{EQ647}
	\end{split}
\end{align}

Finally, we derive the estimate of $\Vert \bp^3 \curl \eta \Vert_{0.5}^2$.
An application of the fundamental theorem of calculus to \eqref{EQ146} leads to
\begin{align*}
	[\curl \eta (t)]^k
	=
	[	t \curl u_0]^k
	+
	\int_0^t
	\epsilon_{jki} v^i,_s
	\int_0^{t'}
	(A^s_j)_t dt'' dt'
	+
	\int_0^t
	\int_0^{t'} Q_0(A, Dv) dt'' dt'
	.
	\llabel{EQ455}
\end{align*}
Applying $\bp^3$ to the above equation, we arrive at
\begin{align*}
	\begin{split}
	\Vert	 \bp^3 \curl \eta 
	\Vert_{0.5}^2
	&
	\les
	\underbrace{T \Vert \bp^3 \curl u_0
	\Vert_{0.5}^2}_{\LL_1}
	+
\underbrace{	\Vert
	\bp^3
	\int_0^t 	
	\epsilon_{jki}
	v^i,_s
	(A(t') - I_3) dt'
	\Vert_{0.5}^2 }_{\LL_2}
	\\&\indeq
	+
	\underbrace{\Vert
	\bp^3
	\int_0^t 	\int_0^{t'}
	Q_0(A, Dv) dt'' dt'
	\Vert_{0.5}^2}_{\LL_3}
	.
	\llabel{EQ565}
	\end{split}
\end{align*}
The term $\LL_1$ is bounded by $TM_0$ since $\Vert \bp^3 \curl u_0 \Vert_{0.5} = \Vert \bp^3 \curl_\eta v(0)\Vert_{0.5}$.
For the term $\LL_2$, we use the Leibniz rule, obtaining
\begin{align*}
	\begin{split}
	\LL_2
	&
	\les
	\Vert
	\int_0^t 	
	\bp^3 Dv
	(A- I_3) dt'
	\Vert_{0.5}^2
	+
	\Vert
	\int_0^t 	
	Dv
	\bp^3 A dt'
	\Vert_{0.5}^2
	+
	\sum_{l=1}^2
	\Vert
	\int_0^t 	
	\bp^l	Dv
	\bp^{3-l} A dt'
	\Vert_{0.5}^2
		\\&
	=:
	\LL_{21}
	+
	\LL_{22}
	+
	\LL_{23}
	.
	\end{split}
\end{align*}
For the term $\LL_{21}$, we integrate by parts in time, leading to
\begin{align}
	\begin{split}
	\LL_{21}
	&\les
	\Vert	\int_0^t 	
	\bp^3 D\eta
	\pt A dt'
	\Vert_{0.5}^2
	+
	\Vert
	\bp^3 D \eta 
	(A- I_3)
	\big|^{t'= t}_{t'=0}
	\Vert_{0.5}^2
	=:
	\LL_{211}
	+
	\LL_{212}
	.
	\label{EQ662}
	\end{split}
\end{align}
For the term $\LL_{211}$, we appeal to \eqref{EQkap}, obtaining
\begin{align}
	\begin{split}
		\LL_{211}
		&\les
		\int_0^T 
		\Vert \bp^3 D\eta\Vert_{0.5}^2
		\Vert \pt A\Vert_{1.5+\epsilon}^2
		dt'
		\les
		TP(\sup_{t\in [0,T]} E(t))
		.
		\label{EQ661}
	\end{split}
\end{align}
The term $\LL_{212}$ is treated in a similar fashion as in \eqref{EQ659}, and we arrive at
\begin{align*}
	\LL_{212}
	\les
	TP(\sup_{t \in [0,T]} E(t))
	.
\end{align*}
Similarly, the term $\LL_{22}$ is estimated as
\begin{align}
	\begin{split}
	\LL_{22}
	\les
		\int_0^t 		
	\Vert
	Dv
	\bp^3 A 
	\Vert_{0.5}^2
	dt'
	\les
			\int_0^t 		
	\Vert
	Dv \Vert_2^2 
	\Vert
	\bp^3 A 
	\Vert_{0.5}^2
	dt'
	\les
	TP(\sup_{t\in [0,T]} E(t))
	.
	\label{EQ663}
	\end{split}
\end{align}
The terms $\LL_{23}$ and $\LL_3$ are treated using similar arguments as in \eqref{EQ662}--\eqref{EQ663}.
Therefore, we conclude the estimate
\begin{align}
	\Vert \bp^3 \curl \eta \Vert_{0.5}^2
	\les
	1+TP(\sup_{t \in [0,T]} E(t))
	.
	\label{EQ664}
\end{align}

Collecting the estimates in \eqref{EQ648}--\eqref{EQ650}, \eqref{EQ649}, \eqref{EQ647}, and \eqref{EQ664}, we complete the proof of the lemma.
\end{proof}

\startnewsection{Energy estimates}{sec04}
The following lemma provides $\bp^4$ energy estimates of the solutions $v$, $J$, and $\eta$.
\cole
\begin{Lemma}
\label{Lenergy1}
We have 
\begin{align}
\begin{split}
	&
	\sup_{t\in [0,T]}
	\int_\Omega 
	\left(	
	|\bp^4 J (t)|^2
	+
	 | \bp^4 v(t)|^2	
	 \right)
	+
	\sup_{t\in [0,T]}
	\int_\Gamma
	|\bp^4 \eta^3 (t)|^2
	\les
	1+ T P(\sup_{t\in [0,T]} E(t))
	.
	\label{EQ39}
\end{split}
\end{align}
\end{Lemma}
\colb

\begin{proof}[Proof of Lemma~\ref{Lenergy1}]
Applying $\bp^4$ to \eqref{EQ01} and taking the inner product with $\bp^4 v$, we obtain
\begin{align}
	\begin{split}
	&
	\frac{1}{2} \frac{d}{dt}
	\int_\Omega \rho_0 
	| \bp^4  v|^2 
	+
	\underbrace{
	\int_\Omega \bp^4
	a^k_i  
	(\rho_0^2 J^{-2}),_k 
	\bp^4 v^i
	}_{\mathcal{J}_1}
	+
	\underbrace{\int_\Omega 
	a^k_i 
	(\rho_0^2 \bp^4  J^{-2}),_k 
	\bp^4 v^i}_{\mathcal{J}_2}
	\\&\indeq
	=
	\underbrace{\sum_{l=1}^3
	c_{l}
	\int_\Omega
	\bp^{4-l} a_i^k 
	( \rho_0^2 
	\bp^l  J^{-2}),_k 
	\bp^4 v^i}_{\mathcal{J}_3}
	.
	\label{EQ91}
	\end{split}
\end{align}

\textit{Estimate of $\mathcal{J}_1$ in \eqref{EQ91}}:
From \eqref{EQ61} and the Leibniz rule it follows that
\begin{align*}
\begin{split}
	\JJ_1
	&
	=
	-
	\underbrace{
	\int_\Omega
	J^{-1}  a^s_i a^k_r 
	\bp^4 \eta^r,_s
	(\rho_0^2 J^{-2}),_k 
	\bp^4 v^i}_{\JJ_{11}}
	+
	\underbrace{
	\int_\Omega 
	J^{-1} a^s_r a^k_i 
	\bp^4  \eta^r,_s
	( \rho_0^2 J^{-2}),_k 
	\bp^4 v^i}_{\JJ_{12}}
	\\&\indeq
	+
	\underbrace{
		\sum_{l=1}^3
	c_l
		\int_\Omega
	\bar{\partial}^l
	( J^{-1} (a^s_r a^k_i - a^s_i a^k_r))
	\bp^{4-l} \eta^r,_s
	( \rho_0^2 J^{-2}),_k 
	\bp^4 v^i}_{\JJ_{13}}
	.
\end{split}
\end{align*}
For the term $\JJ_{11}$, we integrate by parts in $\partial_s$, obtaining
\begin{align}
\begin{split}
	-
	\int_0^T
	\JJ_{11}
	&
	=	
		\underbrace{\int_0^T
		\int_{ \{x_3 = 0\} }
		J^{-1}  a^3_i a^k_r
		\bp^4 \eta^r
		( \rho_0^2 J^{-2} ),_k
		\bp^4 v^i }_{=0}
	-
	\underbrace{\int_0^T
	\int_{ \Gamma }
	  J^{-1}  a^3_i a^3_r
	 \bp^4 \eta^r
	( \rho_0^2 J^{-2} ),_3
	\bp^4 v^i }_{\JJ_{111}}
	\\
	&\indeq
	+
\underbrace{	\int_0^T
	\int_\Omega
	 J^{-1}  a^s_i a^k_r
	 \bp^4 \eta^r
	(\rho_0^2 J^{-2} ),_k
	\bp^4 v^i,_s}_{\JJ_{112}}
	+
	\underbrace{\int_0^T
	\int_\Omega
	(J^{-1}  a^s_i a^k_r
	(\rho_0^2 J^{-2} ),_k
	),_s
	\bp^4 \eta^r
	\bp^4 v^i}_{\JJ_{113}}
	,
\label{EQ30}
\end{split}
\end{align}
where we note that $a^k_r(\rho_0^2 J^{-2}),_k =a^3_r(\rho_0^2 J^{-2}),_3$ since 
$\rho_0^2 J^{-2}=1$ on $\Gamma \times [0,T]$.
In \eqref{EQ30}, we have also used that $\eta^3 = 0 $ on $\{x_3= 0\} \times [0,T]$, which leads to $a^3_i \bp^4 v^i = 0$ on $\{x_3 = 0\} \times [0,T]$.
We rewrite the term $\JJ_{111}$ as
\begin{align*}
	\begin{split}
	-
	\JJ_{111}
	&
	=
	-
	\underbrace{\frac{1}{2}
	\int_0^T
	\frac{d}{dt}
		\int_\Gamma
		J^{-1} 
		(\rho_0^2 J^{-2}),_3
	 |\bp^4 \eta^r a^3_r|^2}_{\JJ_{1111}}
	 +
	 \underbrace{\frac{1}{2} 
	 \int_0^T
	 \int_{\Gamma}
	 ( J^{-1} 
	 (\rho_0^2 J^{-2}),_3)_t
	 |\bp^4 \eta^r a^3_r|^2}_{\JJ_{1112}}
	 \\
	 &\indeq
	 	 +
	 \underbrace{\int_0^T
	 \int_{\Gamma}
	  J^{-1} 
	 (\rho_0^2 J^{-2}),_3
	 \bp^4 \eta^r \pt a^3_r
	 \bp^4 \eta^i a^3_i}_{\JJ_{1113}}
	 .
	\end{split}
\end{align*}
Note that from \eqref{EQ68} we may assume that 
\begin{align}
	-J^{-1} 
	(\rho_0^2 J^{-2}),_3 
	=
	-J^{-1} 
	(\rho_0^2,_3 J^{-2} 
	-
	2 J^{-3} J,_3
	\rho_0^2 ) 
	\geq 
	\frac{1}{2}
	,
	\label{EQ84}
\end{align}
on $\Gamma \times [0,T]$,
which can be justified by using the fundamental theorem of calculus and taking $T>0$ sufficiently small after we establish the a~priori bounds.
From \eqref{EQ84} it follows that
\begin{align}
	\begin{split}
	\JJ_{1111}
	&
	=	
	 -\frac{1}{2}
	\int_\Gamma
	J^{-1} 
	(\rho_0^2 J^{-2} ),_3
	|\bp^4 \eta^r  a^3_r |^2
	(T)
	+
	M_0
	\geq
	\frac{1}{4}
	\int_\Gamma
	|\bp^4 \eta^r a^3_r  |^2 (T) 
	+
	M_0
	.
	\label{EQ249}
	\end{split}
\end{align}
For the term $\JJ_{1112}$,
using the H\"older and Sobolev inequalities and Lemma~\ref{trace}, we get
\begin{align*}
	\begin{split}
	\JJ_{1112}
	&
	\les
	\int_0^T
	(\Vert DJ_t\Vert_{L^\infty}
	+
	1)
	| \bp^4 \eta |_0^2
	\les
	\int_0^T
	(\Vert J_t\Vert_{H^3}
	+
	1)
	\Vert \bp^3 \eta\Vert_{1.5}^2
	\les
	T P(\sup_{t\in [0,T]} E(t))
	,
	\end{split}
\end{align*}
since $\Vert \pt^3 \eta\Vert_{1.5}^2 \leq E(t)$.
Similarly, the term $\JJ_{1113}$ is estimated as
\begin{align*}
	\JJ_{1113}
	\les
	T P(\sup_{t\in [0,T]} E(t))
	.
\end{align*}
For the term $\JJ_{112}$, we integrate by parts in time, obtaining
\begin{align*}
	\begin{split}
	\JJ_{112}
	&
	=	-
	\underbrace{	
	\int_0^T
	\int_\Omega
	J^{-1}  a^s_i a^k_r
	\bp^4 \partial_t \eta^r
	( \rho_0^2 J^{-2} ),_k
	\bp^4 \eta^i,_s
	}_{\JJ_{1121}}
	+
	\underbrace{
	\int_\Omega J^{-1}  a^s_i a^k_r
	\bp^4  \eta^r 
	( \rho_0^2 J^{-2} ),_k
	\bp^4 \eta^i,_s
	\big|_0^T
	}_{\JJ_{1122}}
	\\&\indeq
	-
	\underbrace{	
	\int_0^T
	\int_\Omega
	(	J^{-1}  a^s_i a^k_r
	( \rho_0^2 J^{-2} ),_k)_t
		\bp^4  \eta^r
	\bp^4 \eta^i,_s
	}_{\JJ_{1123}}
	.
	\end{split}
\end{align*}
It is clear that
$
	-
	\JJ_{1121}
	+
	\int_0^T
	\JJ_{12}
	=
	0
$.
Note that $\bp^4 J = a^s_i \bp^4 \eta^i,_s+ \sum_{l=1}^3 c_l \bp^l a^s_i \bp^{4-l} \eta^i,_s$.
For the pointwise in time term $\JJ_{1122}$, using the H\"older and Sobolev inequalities and the fundamental theorem of calculus, we arrive at
\begin{align*}
	\begin{split}
	\JJ_{1122}
	&
	\les
	\Vert  \bp^4 J (T)\Vert_0^2
	+
	\Vert   \bp^4 \eta (T)\Vert_{0}^2
	+
	\sum_{l=1}^3
	 \Vert \bp^l a (T) \bp^{4-l} D \eta (T)\Vert_0^2
	+
	1
	\les
	1+
	T P(\sup_{t\in [0,T]} E(t))
	,
	\end{split}
\end{align*}
since $\Vert \pt J\Vert_4^2 + \Vert v\Vert_4^2 \leq E(t)$.
For the term $\JJ_{1123}$, we appeal to Lemma~\ref{duality}, obtaining
\begin{align*}	
	\begin{split}
	-
	\JJ_{1123}
	&
	\les	
	\int_0^T
	(1+ \Vert DJ_t\Vert_{L^\infty})
	\Vert \bp^4 \eta\Vert_{0.5}
	\Vert \bp^4 D \eta\Vert_{-0.5}
	\les
		\int_0^T
		(1+ \Vert J_t\Vert_{3})
	\Vert \bp^3 \eta\Vert_{1.5}^2
	\les
	T P(\sup_{t\in [0,T]} E(t))
	,
	\end{split}
\end{align*}
where we used the H\"older and Sobolev inequalities.
The term $\JJ_{113}$ is estimated using the H\"older and Sobolev inequalities as
\begin{align*}
	\begin{split}
	\JJ_{113}
	\les
	\int_0^T
	(1+ \Vert D^2 J\Vert_{L^\infty})
	\Vert \bp^4 \eta\Vert_0
	\Vert \bp^4 v\Vert_0
	\les
	T P(\sup_{t \in [0,T]} E(t))
	. 	
	\end{split}
\end{align*}
The term $\JJ_{13}$ consists of essentially lower-order terms which can be estimated using the H\"older and Sobolev inequalities and the fundamental theorem of calculus as
\begin{align*}
	\int_0^T
	\JJ_{13}
	\les
	TP(\sup_{t\in [0,T]} E(t))
	.
\end{align*}
Collecting \eqref{EQ249} and the above estimates, we conclude that, after integrating in time of \eqref{EQ91},
\begin{align}
	\begin{split}
	&
	\frac{1}{2}
	\int_\Omega
	\rho_0
	|\bp^4 v(T)|^2
	+
	\frac{ 1}{4}
	\int_\Gamma
	|\bp^4 \eta^r  a^3_r |^2 (T)
	+
	\int_0^T
	\mathcal{J}_2
	\les
	1+T P(\sup_{t\in [0,T]} E(t))
		+
	\int_0^T
	\mathcal{J}_3
	.
	\label{EQ711}
	\end{split}
\end{align}
Using $|\bp^4 \eta^r a^3_r| \geq |\bp^4 \eta^3 a^3_3| - |\bp^4 \eta^\alpha a^3_\alpha|$ and $\Vert a  - I_3 \Vert_2 \les T$ on $[0,T]\times \Omega$, we obtain 
\begin{align}
	\begin{split}
		\int_\Gamma
		|\bp^4 \eta^3 a^3_3 |^2 (T)
		&
	\leq
	CT \int_\Gamma
	|\bp^4 \eta (T)|^2 
	+
		\int_\Gamma
	|\bp^4 \eta^r  a^3_r	|^2 (T)
	\leq
	CT \Vert \bp^3 \eta (T)\Vert_{1.5}^2 
	+
	\int_\Gamma
	|\bp^4 \eta^r  a^3_r |^2	 (T)
	,
	\label{EQ712}
	\end{split}
\end{align}
where we used the H\"older and Sobolev inequalities and Lemma~\ref{trace}.
Combining \eqref{EQ711}--\eqref{EQ712}, we arrive at
\begin{align}
	\begin{split}
		&
	\frac{1}{2}
	\int_\Omega
	\rho_0 |\bp^4 v(T)|^2
	+
	\frac{ 1}{8}
	\int_\Gamma
	|\bp^4 \eta^3 (T)|^2
	+
	\int_0^T
	\mathcal{J}_2
	\les
	1+ T P(\sup_{t\in [0,T]} E(t))
	+
	\int_0^T
	\mathcal{J}_3
	.
	\label{EQ92}	
	\end{split}
\end{align}

\textit{Estimate of $\mathcal{J}_2$ in \eqref{EQ91}}:
We integrate by parts in $\partial_k$ and use the Piola identity \eqref{piola} to get
\begin{align}
	\begin{split}
	\mathcal{J}_2
	&
	=
	\int_{\Gamma}
	a^3_i \rho_0^2 \bp^4 J^{-2}
	\bp^4 v^i
	-
	\int_{\{x_3=0\}}
	a^3_i \rho_0^2 \bp^4 J^{-2}
	\bp^4 v^i 
	-
	\int_\Omega
	a^k_i \rho_0^2 \bp^4 J^{-2}
	\bp^4 v^i,_k
	\\&
	=
	-
	\int_\Omega
	a^k_i \rho_0^2 \bp^4 J^{-2}
	\bp^4 v^i,_k
	,		
	\label{EQ63}
	\end{split}
\end{align}
since $\bp^4 J^{-2}=0$ on $\Gamma \times [0,T]$ and $a^3_i \bp^4 v^i = 0$ on $\{x_3=0\} \times [0,T]$.
Note that from \eqref{EQ20} and the Leibniz rule it follows
\begin{align}
	\begin{split}
		\bp^4 \pt J
		&
		=
		a^k_i \bp^4 v^i,_k
		+
		\bp^4 a^k_i v^i,_k
		+
		\sum_{l=1}^3
		c_l
		\bp^l a^k_i  \bp^{4-l} v^i,_k
		.
		\label{EQ680}
	\end{split}
\end{align}
Inserting \eqref{EQ680} to \eqref{EQ63}, we arrive at
\begin{align*}
	\begin{split}
	\mathcal{J}_2
	&
	=
	2\int_\Omega
	\rho_0^2
	a^k_i 
	J^{-3} \bp^4 J	\bp^4 v^i,_k
	+
	\underbrace{\sum_{l=1}^3
	c_l
	\int_\Omega
	\rho_0^2
	a^k_i 
	\bp^l J^{-3}
	\bp^{4-l} J	
	\bp^4 v^i,_k}_{\mathcal{J}_{24}}
	\\
	&
	=
		\underbrace{2\int_\Omega
	\rho_0^2
	J^{-3}
	\bp^4 J
	\bp^4 \pt J}_{\mathcal{J}_{21}}
	 -
	 \underbrace{
	 \int_\Omega
	 \rho_0^2
	 J^{-3}
	 \bp^4 J
	\bp^4 a^k_i v^i,_k
	}_{\mathcal{J}_{22}}
	 +
	 \underbrace{
	 	\sum_{l=1}^3
	 	c_l
	 	\int_\Omega
	 	\rho_0^2
	 	J^{-3}
	 	\bp^4 J
	 	 \bp^l a^k_i  \bp^{4-l} v^i,_k
	 }_{\mathcal{J}_{23}}
	 +
	 \JJ_{24}
	 .
	\end{split}
\end{align*}
The term $\mathcal{J}_{21}$ can be rewritten as
\begin{align}
	\begin{split}
	\mathcal{J}_{21}
	&
	=
	\frac{d}{dt}
	\int_\Omega
	\rho_0^2
	J^{-3}
	|\bp^4 J|^2
	-
	\underbrace{\int_\Omega
	(\rho_0^2 J^{-3})_t
	|\bp^4 J|^2}_{\JJ_{21}'}
	.
	\label{EQ410}
	\end{split}
\end{align}
Using the H\"older and Sobolev inequalities, we obtain
\begin{align*}
	\int_0^T
	\JJ'_{21}
	\les
	TP(\sup_{t \in [0,T]} E(t))
	.
\end{align*}
The highest order term in $\mathcal{J}_{22}$ is of the form $\int_\Omega \rho_0^2 J^{-3} \bp^4 J \bp^4 D\eta Dv$, which can be treated using the H\"older and Sobolev inequalities and Lemma~\ref{duality} as
\begin{align*}
	\begin{split}
	\int_0^T 
	\int_\Omega \rho_0^2 J^{-3} \bp^4 J \bp^4 D\eta Dv
	&
	\les
	\int_0^T 
	\Vert  \bp^4 J\Vert_{0.5}
	\Vert \bp^4 D \eta\Vert_{-0.5}
	\\&
	\les
	\int_0^T
	\Vert  \bp^3 J\Vert_{1.5}
	\Vert \bp^3  \eta\Vert_{1.5}
	\les
	T P(\sup_{t\in [0,T]} E(t))
	,
	\end{split}
\end{align*}
since $\Vert \bp^3 J\Vert_{1.5}^2 + \Vert \bp^3 \eta\Vert_{1.5}^2 \leq E(t)$.
The rest of the terms in $\JJ_{22}$ are of lower order which can be estimated using the H\"older and Sobolev inequalities and the fundamental theorem of calculus. 
Therefore, we obtain
\begin{align*}
	-
	\int_0^T
	\JJ_{22}
	\les
	1+TP(\sup_{t \in [0,T]} E(t))
	.
\end{align*}
The term $\JJ_{23}$ consists of lower-order terms which can be estimated using the H\"older and Sobolev inequalities and the fundamental theorem of calculus as
\begin{align*}
	\begin{split}
		\int_0^T 
		\JJ_{23}	
		\les
		T P(\sup_{t\in [0,T]} E(t))
		,
	\end{split}
\end{align*}
For the term $\JJ_{24}$, we integrate by parts in $\partial_k$ and use the Piola identity \eqref{piola}, obtaining
\begin{align}
	\begin{split}
	\mathcal{J}_{24}
	&
	=
		\sum_{l=1}^3
	c_l
	\int_{\Gamma}
	\rho_0^2 a^3_i \bp^l
	J^{-3}
	\bp^{4-l} J
	\bp^4 v^i
	+
	\sum_{l=1}^3
	c_l
	\int_{\{x_3 = 0\}}
	\rho_0^2 a^3_i \bp^l
	J^{-3}
	\bp^{4-l} J
	\bp^4 v^i
	\\&\indeq
	+
	\sum_{l=1}^3
	c_l
	\int_\Omega
	a^k_i
	(\rho_0^2  \bp^l
	J^{-3}
	\bp^{4-l} J),_k
	\bp^4 v^i
	=
		\sum_{l=1}^3
	c_l
	\int_\Omega
	a^k_i
	(	\rho_0^2 
	\bp^l
	J^{-3}
	\bp^{4-l} J),_k
	\bp^4 v^i
	,
	\label{EQ90}
	\end{split}
\end{align}
since $\bp^l J^{-3} = 0$ on $\Gamma \times [0,T]$ for $l \in \mathbb{N}$ and $a^3_i \bp^4 v^i = 0$ on $\{x_3 = 0\} \times [0,T]$.
Thus, using the H\"older and Sobolev inequalities, the term $\JJ_{24}$ is estimated as
\begin{align}
	\begin{split}
		\int_0^T
	\JJ_{24}
	\les	
	TP(\sup_{t \in [0,T]} E(t))
	.
	\label{EQ679}
	\end{split}
\end{align}
Combining \eqref{EQ92}, \eqref{EQ410}, and the above estimates, we conclude
\begin{align}
	\begin{split}
		&
		\frac{1}{2}
		\int_\Omega
		\left(
		\rho_0 |\bp^4 v (T)|^2
		+
		\rho_0^2
		J^{-3} (T)
		|\bp^4 J(T)|^2
		\right)
		+
		\frac{ 1}{8}
		\int_\Gamma
		|\bp^4 \eta^3 (T)|^2
		\\&\indeq
		\les
		1+T P(\sup_{t\in [0,T]} E(t))
		+
		\int_0^T
		\mathcal{J}_3
		.
		\label{EQ360}
	\end{split}
\end{align}

\textit{Estimate of $\mathcal{J}_3$ in \eqref{EQ91}}:
We proceed analogously as in \eqref{EQ90}--\eqref{EQ679}, obtaining
\begin{align}
	\int_0^T 
	\mathcal{J}_3
		\les
	T P(\sup_{t\in [0,T]} E(t))
	.
	\label{EQ94}
\end{align}

\textit{Concluding the proof}:
Combining \eqref{EQ360}--\eqref{EQ94}, we arrive at
\begin{align*}
\begin{split}
	&
	\frac{1}{2}
	\int_\Omega
	\left(
	\rho_0 |\bp^4 v (T)|^2
	+
	\rho_0^2
	J^{-3} 
	|\bp^4 J(T)|^2
	\right)
	+
	\frac{ 1}{8}
	\int_\Gamma
	|\bp^4 \eta^3 (T)|^2
	\les
	1+ T P(\sup_{t\in [0,T]} E(t))
	.
\end{split}
\end{align*}
The proof of the lemma is thus completed.
\end{proof}

The following lemma provides $\bp^3 \pt^2$ energy estimates of the solutions $v$, $J$, and $\eta$.
\cole
\begin{Lemma}
	\label{Lenergy2}
	For $\delta \in (0,1)$, we have 
\begin{align}
	\begin{split}
		&
		\sup_{t \in [0,T]}
		\int_\Omega
		\left(
		| \rhog D \bp^3 \pt^2 \eta (t)|^2
		+
		|\bp^3 \pt^2 v (t)|^2 
		+
		|\bp^3 \pt^2 J (t)|^2
		\right)
		+
		\sup_{t \in [0,T]}
		\int_\Gamma
		|\bp^3 \pt^2 \eta (t) |^2 
				\\&\indeq
		\les
		C_\delta
		+
		\delta \sup_{t\in [0,T]} E(t)
		+
		C_\delta T P(\sup_{t\in [0,T]} E(t))
		,
		\label{EQ600}
	\end{split}
\end{align}
where $C_\delta>0$ is a constant depending on $\delta$.
\end{Lemma}
\colb

\begin{proof}[Proof of Lemma~\ref{Lenergy2}]
Applying $\bp^3 \pt^2$ to \eqref{EQ01} and taking the inner product with $\bp^3 \pt^2 v^i$, we arrive at
\begin{align}
	\begin{split}
		&
		\frac{1}{2} \frac{d}{dt}
		\int_\Omega \rho_0 
		| \bp^3 \pt^2  v|^2 
		+
		\underbrace{
			\int_\Omega \bp^3 \pt^2
			a^k_i  
			(\rho_0^2 J^{-2}),_k 
			\bp^3 \pt^2 v^i
		}_{\mathcal{J}_1}
		+
		\underbrace{\int_\Omega 
			a^k_i (\rho_0^2 \bp^3 \pt^2  J^{-2}),_k 
			\bp^3 \pt^2 v^i}_{\mathcal{J}_2}
		\\&\indeq
		=
		\underbrace{\sum_{l=0}^3
			\sum_{\substack{m=0\\ 1\leq m+l \leq 4}}^2
			c_{l,m}
			\int_\Omega
			\bp^{3-l} \pt^{2-m} a_i^k 
			( \rho_0^2 
			\bp^l \pt^m J^{-2}),_k 
			\bp^3 \pt^2 v^i}_{\mathcal{J}_3}
		.
		\label{EQ691}
	\end{split}
\end{align}

\textit{Estimate of $\JJ_1$ in \eqref{EQ691}}:
Using the decomposition $\rho_0^2= \rhog^2+ \rhol^2$ (see \eqref{EQ184}--\eqref{EQ185}), we split the term $\mathcal{J}_1$ as
\begin{align}
	\begin{split}
		\mathcal{J}_1
		&
		=
		\underbrace{
			\int_\Omega 
			\bp^3 \pt^2 a^k_i 
			( \rhog^2 J^{-2}),_k 
			\bp^3 \pt^2  v^i
		}_{\mathcal{G}}
		+
		\underbrace{
			\int_\Omega 
			\bp^3 \pt^2  a^k_i 
			( \rhol^2 J^{-2} ),_k 
			\bp^3 \pt^2  v^i	
		}_{\mathcal{L}}
		.
		\label{EQ912}
	\end{split}
\end{align}
First we estimate the term $\mathcal{G}$ in \eqref{EQ912}.
Using integration by parts in $\partial_k$,
we arrive at
\begin{align}
	\begin{split}
		\mathcal{G}
		&
		=
		\int_\Gamma
		\bp^3 \pt^2  a^3_i 
		\rhog^2 J^{-2}
		\bp^3 \pt^2  v^i
		-
		\int_{\{x_3 = 0\}}
		\bp^3 \pt^2  a^3_i 
		\rhog^2 J^{-2}
		\bp^3 \pt^2  v^i
		-
		\int_\Omega 
		\bp^3 \pt^2  a^k_i 
		\rhog^2 J^{-2}
		\bp^3 \pt^2  v^i,_k
		,
		\label{EQ35}
	\end{split}
\end{align}
where we used the Piola identity \eqref{piola}.
The first term on the right hand side vanishes since $\rhog = 0$ on $\Gamma \times [0,T]$.
The second term on the right hand side of \eqref{EQ35} also vanishes since $\bp^3 \pt^2 a^3_i \bp^3 \pt^2 v^i = 0$ on $\{x_3 = 0\} \times [0,T]$.
Therefore,
from \eqref{EQ15} and the Leibniz rule it follows that
\begin{align}
	\begin{split}
		\mathcal{G}
		&
		=
		\underbrace{
			\int_\Omega 
			A^s_i A^k_r \bp^3 \pt^2 \eta^r,_s
			\rhog^2 J^{-1}
			\bp^3 \pt^2 v^i,_k
		}_{\mathcal{G}_{1}}
		-
		\underbrace{
			\int_\Omega 
			A^s_r A^k_i  \bp^3 \pt^2 \eta^r,_s
			\rhog^2 J^{-1}
			\bp^3 \pt^2 v^i,_k
		}_{\mathcal{G}_{2}}
		\\&\indeq
		+
		\underbrace{
			\sum_{l=0}^3
			\sum_{\substack{m=0\\m+l \geq 1}}^1
			c_{l,m}
			\int_\Omega 
			\bp^l \pt^m
			(
			J^{-1}
			(a^s_r a^k_i - a^s_i a^k_r) 
			)
			\bp^{3-l}\pt^{2-m} \eta^r,_s
			\rhog^2 J^{-2}
			\bp^3 \pt^2 v^i,_k
		}_{\mathcal{G}_3}
		.
		\label{EQ188}
	\end{split}
\end{align}
It is clear that
\begin{align}
	\begin{split}
		A^s_i A^k_r \bp^3 \pt^2 \eta^r,_s
		\bp^3 \pt^2 v^i,_k
		&
		=	
		A^s_r  \bp^3 \pt^2 \eta^i,_s
		A^k_r \bp^3 \pt^2 \partial_t \eta^i,_k
		\\&\indeq
		+
		(A^s_i \bp^3 \pt^2 \eta^r,_s 
		-
		A^s_r \bp^3 \pt^2 \eta^i,_s)
		A^k_r \bp^3 \pt^2 \partial_t \eta^i,_k
		.
		\label{EQ72}
	\end{split}
\end{align}
The first term on the right hand side can be rewritten as
\begin{align}
	\begin{split}
		A^s_r  \bp^3 \pt^2 \eta^i,_s
		A^k_r \bp^3 \pt^2 \partial_t \eta^i,_k
		&
		=
		[D_\eta \bp^3 \pt^2 \eta]^i_r
		[D_\eta  \bp^3 \pt^2 \partial_t \eta]^i_r
			\\&
		=
		\frac{1}{2}
		\frac{d}{dt}
		|D_\eta \bp^3 \pt^2 \eta|^2
		-
		\frac{1}{2}
		\bp^3 \pt^2 \eta^i,_s \bp^3 \pt^2 \eta^i,_k (A^s_r A^k_r)_t
		,
		\label{EQ74}
	\end{split}
\end{align}
while the second term on the right hand side of \eqref{EQ72} can be rewritten as
\begin{align}
	\begin{split}
		&
		(A^s_i \bp^3 \pt^2 \eta^r,_s 
		-
		A^s_r \bp^3 \pt^2 \eta^i,_s)
		A^k_r \bp^3 \pt^2 \partial_t \eta^i,_k
		=
		-\epsilon_{ijk}
		\bp^3 \pt^2 \eta^k,_r
		A^r_j \epsilon_{iml} 
		\bp^3 \pt^2 \partial_t \eta^l,_s A^s_m
		\\&
		=
		-\curl_\eta \bp^3 \pt^2 \eta \cdot \curl_\eta \bp^3 \pt^2 \partial_t \eta
		\\&
		=
		-
		\frac{1}{2} \frac{d}{dt}
		|\curl_\eta \bp^3 \pt^2 \eta|^2
		+
		\frac{1}{2}
		\bp^3 \pt^2 \eta^k,_r \bp^3 \pt^2 \eta^k,_s (A^r_j A^s_j)_t
		-
		\frac{1}{2}
		\bp^3 \pt^2 \eta^k,_r \bp^3 \pt^2 \eta^j,_s (A^r_j A^s_k)_t
		.
		\label{EQ75}
	\end{split}
\end{align}
Combining \eqref{EQ72}--\eqref{EQ75}, we arrive at
\begin{align*}
	\begin{split}
		\mathcal{G}_1
		&
		=
		\frac{1}{2}
		\int_\Omega
		\rhog^2 J^{-1}
		\frac{d}{dt}
		\left(
		|D_\eta \bp^3 \pt^2 \eta|^2
		-
		|\curl_\eta \bp^3 \pt^2 \eta|^2
		\right)
		+
		\int_\Omega
		Q( \rhog \bp^3 \pt^2 D\eta)
		\\
		&
		=
		\frac{1}{2}
		\frac{d}{dt}
		\int_\Omega
		\rhog^2 J^{-1}
		\left(
		|D_\eta \bp^3 \pt^2 \eta|^2
		-
		|\curl_\eta \bp^3 \pt^2 \eta|^2
		\right)
		+
		\int_\Omega
		Q(\rhog \bp^3 \pt^2 D\eta)
		,
	\end{split}
\end{align*}
where $Q$ is a quadratic function with $L^\infty ( [0,T] \times \Omega)$ coefficients which may vary from line to line.
Integrating $\mathcal{G}_1$ in time from $0$ to $T$, we get
\begin{align*}
	\begin{split}
		\int_0^T 
		\mathcal{G}_1
		=	
		\frac{1}{2}
		\int_\Omega
		\rhog^2 J^{-1} 
		\left(
		|D_\eta \bp^3 \pt^2 \eta (T)|^2
		-
		|\curl_\eta \bp^3 \pt^2 \eta (T)|^2
		\right)
		+
		\int_0^T 
		\int_\Omega
		Q(\rhog \bp^3 \pt^2 D\eta)
		+
		M_0
		.
	\end{split}
\end{align*}
For the term $\mathcal{G}_2$, we have
\begin{align*}
	\begin{split}
		-
		\mathcal{G}_2
		=	
		-\int_\Omega
		\rhog^2 
		J^{-1}
		\dive_\eta \bp^3 \pt^2 \eta
		\dive_\eta \bp^3 \pt^2 \pt \eta
		=
		-\frac{1}{2}
		\frac{d}{dt}
		\int_\Omega
		\rhog^2 
		J^{-1}
		|\dive_\eta \bp^3 \pt^2 \eta|^2
		+
		\int_\Omega
		Q(\rhog \bp^3 \pt^2 D\eta)
		,
	\end{split}
\end{align*}
which leads to
\begin{align}
	\begin{split}
		-\int_0^T
		\mathcal{G}_2	
		=
		-
		\frac{1}{2}
		\int_\Omega
		\rhog^2 J^{-1} |\dive_\eta \bp^3 \pt^2 \eta (T)|^2
		+
		\int_0^T 
		\int_\Omega
		Q(\rhog \bp^3 \pt^2 D\eta)
		+
		M_0
		.
		\label{EQ503}
	\end{split}
\end{align}
From \eqref{EQ20} and the Leibniz rule it follows that
\begin{align}
	\dive_\eta \bp^3 \pt^2 \eta
	=
	A^s_r \bp^3 \pt^2 \eta^r,_s
	=
	J^{-1} \bp^3 \pt^2 J
	+
	\sum_{l=0}^3
	\sum_{\substack{m=0\\l+m\geq 1}}^1
	c_{l,m} J^{-1}
	\bp^l \pt^m a^s _r 
	\bp^{3-l} \pt^{2-m} \eta^r,_s
	.
	\label{EQ207}
\end{align}
Inserting \eqref{EQ207} to the first term on the right hand side of \eqref{EQ503}, we arrive at
\begin{align*}
	\begin{split}
			&
		-
		\frac{1}{2}
		\int_\Omega
		\rhog^2 J^{-1} 
		|\dive_\eta \bp^3 \pt^2 \eta (T)|^2
		=
		-
		\frac{1}{2}
		\int_\Omega
		\rhog^2 J^{-3} | \bp^3 \pt^2 J (T)|^2	
		\\&\indeq
		+
		\sum_{l=0}^3
		\sum_{\substack{m=0\\l+m\geq 1}}^1
			c_{l,m}
			\int_\Omega
			\rhog^2 J^{-3}
			|\bp^l \pt^m a \bp^{3-l} \pt^{2-m} D\eta |^2 (T)
		\\&\indeq\indeq
		+
		\sum_{l=0}^3
		\sum_{\substack{m=0\\l+m\geq 1}}^1
		c_{l,m}
		\int_\Omega
		\rhog^2 J^{-3}
		\bp^3 \pt^2 J 
		\bp^l \pt^m a
		\bp^{3-l} \pt^{2-m} D\eta (T)
		\\&\indeq
		=:
			-
		\frac{1}{2}
		\int_\Omega
		\rhog^2 J^{-3} | \bp^3 \pt^2 J (T)|^2	
		+
		\GG_{21}+\GG_{22}
		,
	\end{split}
\end{align*}
where we drop the indices for simplicity.
The highest order terms in $\GG_{21}$ can be treated using the H\"older and Sobolev inequalities and the fundamental theorem of calculus as
\begin{align*}
	\begin{split}
	\int_\Omega \rhog^2 J^{-3} |\pt D\eta \bp^3 \pt D\eta|^2 (T)
	\les
	\int_\Omega  
	|\rhog \bp^3 \pt D\eta|^2 (T)
	\les
	1+
	TP(\sup_{t\in [0,T]} E(t))
	\end{split}
\end{align*}
and
\begin{align*}
	\begin{split}
		\int_\Omega \rhog^2 J^{-3} 
		|\bp D\eta \bp^2 \pt^2 D\eta|^2 (T)
		\les
		\int_\Omega  
		| \bp^2 \pt^2 D\eta|^2 (T)
		\les
		1+
		TP(\sup_{t\in [0,T]} E(t))
	,
	\end{split}
\end{align*}
since $\Vert \rhog \bp^3 \pt^2 D\eta\Vert_0^2 + \Vert v\Vert_4^2 + \Vert \pt^2 v\Vert_3^2 \leq E(t)$.
The rest of the terms in $\GG_{21}$ are of lower order which can be treated in a similar fashion using the H\"older and Sobolev inequalities and the fundamental theorem of calculus.
Thus, we arrive at
\begin{align}
	\GG_{21}
	\les
	1+TP(\sup_{t\in [0,T]} E(t))
	.
	\label{EQ119}
\end{align}
For the term $\GG_{22}$, using the Young inequality, we obtain
\begin{align*}
	\begin{split}
		\GG_{22}
		&
		\les
		\delta
		\Vert \rhog \bp^3 \pt^2 J (T)\Vert_0^2
		+
		C_\delta
			\sum_{l=0}^3
		\sum_{\substack{m=0\\l+m\geq 1}}^1
		\Vert \rhog J^{-3} 
		\bp^l \pt^m a 
		\bp^{3-l} \pt^{2-m} D\eta \Vert_0^2 (T)
		\\&
		\les
		C_\delta
		+
		\delta \sup_{t \in [0,T]} E(t)
		+
		C_\delta TP(\sup_{t \in [0,T]} E(t))
		,
	\end{split}
\end{align*}
where the last inequality follows from \eqref{EQ119}.
For the term $\GG_3$ in \eqref{EQ188}, we integrate by parts in time, leading to
\begin{align*}
	\begin{split}
	\int_0^T
	\GG_3
	&
	=
	\underbrace{
	\sum_{l=0}^3
	\sum_{\substack{m=0\\l+m\geq 1}}^1
	c_{l,m}
	\int_\Omega 
	\bp^l \pt^m
	(
	J^{-1}
	(a^s_r a^k_i - a^s_i a^k_r) 
	)
	\bp^{3-l} \pt^{2-m}
	\eta^r,_s
	\rhog^2 J^{-2}
	\bp^3 \pt^2 \eta^i,_k \big|^{t=T}_{t=0}}_{\GG_{31}}
	\\&\indeq
	+
	\underbrace{
	\sum_{l=0}^3
\sum_{\substack{m=0\\l+m\geq 1}}^1
c_{l,m}
	\int_0^T
	\int_\Omega 
	(\bp^l \pt^m
	(
	J^{-1}
	(a^s_r a^k_i - a^s_i a^k_r) 
	)
	\bp^{3-l} \pt^{2-m} \eta^r,_s
	\rhog^2 J^{-2})_t
	\bp^3 \pt^2 \eta^i,_k}_{\GG_{32}}
	.
	\end{split}
\end{align*}
We bound the term $\GG_{31}$ using the Young and Sobolev  inequalities and the fundamental theorem of calculus as
\begin{align*}
	\begin{split}
		\GG_{31}
		&
		\les
		\delta \Vert \rhog \bp^3 \pt^2 D\eta (T)\Vert_0^2
		+
		C_\delta
		\sum_{l=0}^3
		\sum_{\substack{m=0\\l+m\geq 1}}^1
		\sum_{i,k=1}^3
		\Vert 	\rhog
		\bp^l \pt^m
		(
		J^{-1}
		(a^s_r a^k_i - a^s_i a^k_r) 
		)
		\bp^{3-l} \pt^{2-m} \eta^r,_s
		\Vert_0^2 (T)
		+
		1
		\\&
		\les
		C_\delta
		+	
		\delta \sup_{t\in [0,T]} E(t)
		+
		TP(\sup_{t\in [0,T]} E(t))
		,
	\end{split}
\end{align*}
since $\Vert \rhog \bp^3 \pt^2 D\eta\Vert_0^2 + \Vert \pt^2 J\Vert_3^2 + \Vert \pt^2 v\Vert_3^2 \leq E(t)$.
For the term $\mathcal{G}_{32}$, we use the H\"older and Sobolev  inequalities and the fundamental theorem of calculus, obtaining
\begin{align*}
	\begin{split}
		\GG_{32}
	&
	\les
	\sum_{l=0}^3
	\sum_{\substack{m=0\\l+m\geq 1}}^1
	\sum_{i,k=1}^3
	\int_0^T
	\Vert
	\rhog 
	(\bp^l \pt^m
	(
	J^{-1}
	(a^s_r a^k_i - a^s_i a^k_r) 
	)
	\bp^{3-l} \pt^{2-m} \eta^r,_s J^{-2})_t \Vert_0
	\Vert \rhog \bp^3 \pt^2 D\eta\Vert_0
	\\&
	\les
	T P(\sup_{t\in [0,T]} E(t))
	.
	\end{split}
\end{align*}
Collecting the above estimates, we conclude
\begin{align}
	\begin{split}
		&
		\int_0^T
		\GG
		=
		\frac{1}{2}
		\int_\Omega
		\left(
		\rhog^2 
		J^{-1} 
		|D_\eta \bp^3 \pt^2 \eta (T)|^2
		-
		\rhog^2
		J^{-1} 
		| \curl_\eta \bp^3 \pt^2 \eta (T)|^2
		-
		\rhog^2
		J^{-3} 
		|\bp^3 \pt^2 J(T)|^2
		\right)
		+
		\GG'
		,
		\label{EQ992}	
	\end{split}
\end{align}
where $\GG'$ consists of the terms satisfying
\begin{align}
	\GG'
	\les
	C_\delta
	+
	\delta \sup_{t\in [0,T]} E(t)
	+
	C_\delta T P(\sup_{t\in [0,T]} E(t))
	.
	\label{EQ995}	
\end{align}
Next we estimate the term $\LL$ in \eqref{EQ912}.
From \eqref{EQ15} and the Leibniz rule it follows that
\begin{align}
	\begin{split}
		\LL
		&
		=
		-
		\underbrace{
			\int_\Omega
			J^{-1}  a^s_i a^k_r 
			\bp^3 \pt^2 \eta^r,_s
			(\rhol^2 J^{-2}),_k 
		\bp^3 \pt^2 v^i}_{\LL_{1}}
		+
		\underbrace{
			\int_\Omega 
			J^{-1} a^s_r a^k_i 
			\bp^3 \pt^2  \eta^r,_s
			( \rhol^2 J^{-2}),_k 
		\bp^3 \pt^2 v^i}_{\LL_{2}}
		\\&\indeq
		+
		\underbrace{
			\sum_{l=0}^3
			\sum_{\substack{m= 0\\l+m\geq 1}}^1
			c_{l,m}
			\int_\Omega
			\bp^l
			\pt^m
			( J^{-1} (a^s_r a^k_i - a^s_i a^k_r))
			\bp^{3-l}
			\pt^{2-m} 
			\eta^r,_s
			( \rhol^2 J^{-2}),_k 
			\bp^3 \pt^2 v^i}_{\LL_{3}}
		.
		\label{EQ940}
	\end{split}
\end{align}
For the term $\LL_{1}$, we integrate by parts in $\partial_s$, obtaining
\begin{align}
	\begin{split}
		&
	-
	\int_0^T
	\LL_{1}
	=	
	-
	\underbrace{\int_0^T
	\int_{ \Gamma }
	J^{-1}  a^3_i a^3_r
	\bp^3 \pt^2 \eta^r
	( \rhol^2 J^{-2} ),_3
	\bp^3 \pt^2 v^i }_{\LL_{11}}
	\\&\indeq
	+
	\underbrace{	\int_0^T
	\int_\Omega
	J^{-1}  a^s_i a^k_r
	\bp^3 \pt^2 \eta^r
	(\rhol^2 J^{-2} ),_k
	\bp^3 \pt^2 v^i,_s}_{\LL_{12}}
	+
	\underbrace{\int_0^T
	\int_\Omega
	(J^{-1}  a^s_i a^k_r
	(\rhol^2 J^{-2} ),_k
	),_s
	\bp^3 \pt^2 \eta^r
	\bp^3 \pt^2 v^i}_{\LL_{13}}
	,
	\label{EQ930}
	\end{split}
\end{align}
where we note that $a^k_r(\rhol^2 J^{-2}),_k =a^3_r(\rhol^2 J^{-2}),_3$ since 
$\rhol^2 J^{-2}=1$ on $\Gamma \times [0,T]$.
The term $\LL_{11}$ can be rewritten as
\begin{align*}
	\begin{split}
	-
	\LL_{11}
	&
	=
	-
	\underbrace{\frac{1}{2}
	\int_0^T
	\frac{d}{dt}
	\int_\Gamma
	J^{-1} 
	(\rhol^2 J^{-2}),_3
	|\bp^3 \pt^2 \eta^r a^3_r|^2}_{\LL_{111}}
	+
	\underbrace{\frac{1}{2} 
	\int_0^T
	\int_{\Gamma}
	( J^{-1} 
	(\rhol^2 J^{-2}),_3)_t
	|\bp^3 \pt^2 \eta^r a^3_r|^2
	}_{\LL_{112}}
	\\
	&\indeq
	+
	\underbrace{\int_0^T
	\int_{\Gamma}
	J^{-1} 
	(\rhol^2 J^{-2}),_3
	\bp^3 \pt^2 \eta^r \pt a^3_r
	\bp^3 \pt^2 \eta^i a^3_i
	}_{\LL_{113}}
	.
	\end{split}
\end{align*}
Note that from \eqref{EQ185} we may assume that 
\begin{align}
	-J^{-1} 
	(\rhol^2 J^{-2}),_3 
	=
	-J^{-1} 
	(\rhol^2,_3 J^{-2} 
	-
	2 J^{-3} J,_3
	\rhol^2 ) 
	\geq 
	\frac{1}{2}
	\inon{on~$[0,T]\times \Gamma$}
	,
	\label{EQ284}
\end{align}
from where
\begin{align}
	\begin{split}
	\LL_{111}
	&
	=	
	-\frac{1}{2}
	\int_\Gamma
	J^{-1} 
	(\rhol^2 J^{-2} ),_3
	|\bp^3 \pt^2 \eta^r  a^3_r |^2
	(T)
	+
	M_0
	\geq
	\frac{1}{4}
	\int_\Gamma
	|\bp^3 \pt^2 \eta^r a^3_r |^2 (T)
	+
	M_0
	.
	\label{EQ996}
	\end{split}
\end{align}
Using the H\"older and Sobolev inequalities, we get
\begin{align*}
	\begin{split}
	\LL_{112}
	+
	\LL_{113}
	&
	\les
	\int_0^T
	(\Vert DJ_t\Vert_{L^\infty}
	+
	1)
	| \bp^3 \pt^2 \eta|_0^2
	\les
	T P(\sup_{t\in [0,T]} E(t))
	,
	\end{split}
\end{align*}
since $\vert \bp^3 \pt^2 \eta\vert_0^2 \leq E(t)$.

Now we claim that
\begin{align}
	\begin{split}
		\sup_{t\in [0,T]}
		\vert \bp^3 \pt^2 \eta^\alpha (t)
		\vert_{0}^2
		\les
		C_\delta
		+
		\delta \sup_{t\in [0,T]} E(t)
		+
		C_\delta T P(\sup_{t\in [0,T]} E(t))
		,
		\label{EQ982}
	\end{split}
\end{align}
where $\alpha=1,2$.
To prove the claim, we resort to the boundary dynamics for the flow map $\eta$. In fact, from the momentum equation \eqref{EQ200}, we see that $\eta$ satisfies 
\begin{align*}
	\pt^2 \eta^i  -2  a^3_i J,_3 -2 a^3_i 
	=
	0
	\inon{on~$[0,T]\times \Gamma$}
	\comma
	i=1,2,3,
\end{align*}
since $\rho_0,_3=-1$, $\rho_0,_\alpha = 0$, and $J=\rho_0=1$ on $[0,T]\times \Gamma$. 
Then from 
\begin{align*}
	a^3_3\pt^2 \eta^\alpha 
	=
	a^3_3( 2 a^3_\alpha J,_3 
	+
	2 a^3_\alpha ) 
	= 
	a^3_\alpha (2 a^3_3 J,_3 + 2a^3_3)
	=  a^3_\alpha  \pt^2 \eta^3
	\comma
	\alpha=1,2,
\end{align*}
we deduce the following algebraic relation between $\pt^2\eta^\alpha$ and $\pt^2 \eta^3$:  
\begin{equation}
	\pt^2\eta^\alpha 
	= 
	(a^3_3)^{-1} a^3_\alpha \pt^2 \eta^3
	\inon{on~$[0,T]\times \Gamma$}
	\comma
	\alpha=1,2
	.
	\label{EQLJ}
\end{equation}
Applying $\bp^{3}$ to the above equation, we get $\pt^{2} \bp^3 \eta^\alpha 
=
\bp^3 \left[ (a^3_3)^{-1} a^3_\alpha \pt^2 \eta^3 \right] $, 
which leads to 
\begin{align}
	\begin{split}
	\vert \pt^{2} \bp^3
	\eta^\alpha \vert_0^2
	&
	\les
	\underbrace{
	\vert (a^3_3)^{-1} a^3_\alpha 
	\bp^3 \pt^{2} 
	\eta^3 \vert_0^2}_{\II_1}
	+
	\underbrace{ 
	\vert \bp^3 
	((a^3_3)^{-1} a^3_\alpha )
	\pa_t^2 \eta^3 \vert_0^2}_{\II_2} 
	+
	\underbrace{
	\vert  \bp  ((a^3_3)^{-1} a^3_\alpha )
	\bp^{2} \pt^2 \eta^3 \vert_0^2}_{\II_{3}}
	\\&\indeq
	+
	\underbrace{
	\vert  \bp^2  ((a^3_3)^{-1} a^3_\alpha )
	\bp  \pt^2 \eta^3 \vert_0^2}_{\II_4}
	.
	\label{EQ232}
	\end{split}
\end{align}
The term $\II_1$ is estimated using the H\"older and Sobolev inequalities as
\begin{equation} 
	\mathcal{I}_1 
	\les 
	\Vert a^3_\alpha\Vert_{L^\infty}
	\vert \bp^3 \pt^{2}\eta^3  
	\vert_0^2 
	\les 
	T P(\sup_{t\in [0,T]} E(t)), 
	\label{EQ984}
\end{equation}
where we used $\Vert a - I_3\Vert_{2} \les T$.
For the term $\mathcal{I}_2$, using the Leibniz rule, we get
\begin{align}
	\begin{split}
		\mathcal{I}_2
		&
		\les
		\underbrace{
			\vert \bp^3 (a^3_3)^{-1}
			a^3_\alpha
			\vert_0^2}_{\II_{21}}
		+	
		\underbrace{
			\vert (a^3_3)^{-1}
			\bp^3 a^3_\alpha
			\vert_0^2}_{\II_{22}}
		+	
		\underbrace{
			\sum_{l=1}^2
			\vert \bp^l (a^3_3)^{-1}
			\bp^{3-l} a^3_\alpha
			\vert_0^2}_{\II_{23}}
		.
		\label{EQ989}
	\end{split}
\end{align}
Using the H\"older and Sobolev inequalities and the fundamental theorem of calculus, we arrive at
\begin{align}
	\II_{21}
	\les
	\vert \bp^3 (a^3_3)^{-1}
	\vert_0^2
	\Vert
	a^3_\alpha
	\Vert_{L^\infty}^2
	\les
	TP(\sup_{t \in [0,T]} E(t))
	,
\end{align}
where we used
$
	\vert \bp^4 \eta\vert_0^2
	\les
	\Vert \bp^3 \eta\Vert_{1.5}^2
$ and \eqref{EQ70}
in the last inequality.
For the term $\II_{22}$, from \eqref{EQ70} and the Leibniz rule it follows that
\begin{align}
	\begin{split}
		\II_{22}
		\les
		\vert \bp^3 (\bp \eta \bp \eta^3)
		\vert_0^2
		\les
		\underbrace{
			\vert \bp \eta  \bp^4 \eta^3
			\vert_0^2}_{\II_{221}}
		+
		\underbrace{	
			\vert \bp^4 \eta  \bp \eta^3
			\vert_0^2}_{\II_{222}}
		+
		\underbrace{	
			\sum_{l=1}^2
			\vert \bp^{l+1} \eta  \bp^{4-l} \eta^3
			\vert_0^2}_{\II_{223}}
		.
	\end{split}
\end{align}
The term $\II_{221}$ is estimated using the H\"older and Sobolev inequalities as
\begin{align}
	\II_{221}
	\les
	\vert \bp^4 \eta^3 \vert_0^2
	\les
	C_\delta
	+
	\delta \sup_{t\in [0,T]} E(t)
	+
	C_\delta T P(\sup_{t\in [0,T]} E(t))
	,
\end{align}
where the last inequality follows from Lemma~\ref{Lenergy1}.
For the term $\II_{222}$, we use the H\"older inequality to get
\begin{align}
	\II_{222}
	\les
	 \vert \bp^4 \eta\vert_0^2
	\vert \bp \eta^3 \vert_{L^\infty (\Gamma)}^2
	\les
	T
	P(\sup_{t \in [0,T]} E(t))
	,
\end{align}
since $\Vert \bp \eta^3 \Vert_{L^\infty (\Gamma)} \les T$ which holds by the fundamental theorem of calculus.
The term $\II_{223}$ is estimated using the H\"older and Sobolev inequalities and the fundamental theorem of calculus as
\begin{align}
	\begin{split}
	\II_{223}
	\les
	\Vert \bp^2 \eta\Vert_{L^\infty (\Omega)}^2
	\vert \bp^3 \eta\vert_0^2	
	\les
	1+ TP(\sup_{t \in [0,T]} E(t))
	.
	\end{split}
\end{align}
Similarly, we bound the term $\II_{23}$ as
\begin{align}
	\II_{23}
	\les
	1+ TP(\sup_{t \in [0,T]} E(t))
	.
\end{align}
For the term $\II_3$, using the H\"older and Sobolev inequalities, we obtain
\begin{align}
	\II_3
	\les
	\vert \bp^2 \pt^2 \eta\vert_0^2
	\les
	1+ TP(\sup_{t \in [0,T]} E(t))
	,
\end{align}
since $\vert \bp^2 \pt^3 \eta\vert_0^2 \leq E(t)$,
where we used the fundamental theorem of calculus in the last inequality.
The highest order term in $\II_4$ is scales like $\vert \bp^3 \eta \bp \pt^2 \eta^3 \vert_0^2$
which is treated using the H\"older and Sobolev inequalities as
\begin{align*}
	\begin{split}
		\vert \bp^3 \eta \bp \pt^2 \eta^3 \vert_0^2
	\les
	\vert \bp^3 \eta \vert_{L^4 (\Gamma)}^2
	\vert
	\bp \pt^2 \eta^3 \vert_{L^4 (\Gamma)}^2
	\les
	\Vert \bp^3 \eta\Vert_1^2
	\Vert \pt^2 \eta\Vert_2^2
	\les
	1+ TP(\sup_{t \in [0,T]} E(t))
	,	
	\end{split}
\end{align*}
where we used the Sobolev embedding $H^{0.5}(\Gamma
) \hookrightarrow L^4(\Gamma)$, Lemma~\ref{trace}, and the fundamental theorem of calculus.
The rest of the terms in $\II_4$ are of lower order which can treated in a similar fashion as above.
Thus, we have
\begin{align}
	\II_4
	\les
	1+TP(\sup_{t\in [0,T]} E(t))
	.
	\label{EQ233}
\end{align}
Collecting the estimates \eqref{EQ232}--\eqref{EQ233}, we complete the proof of the claim \eqref{EQ982}.

Next we estimate the term $\LL_{12}$ in \eqref{EQ930}. 
Using integration by parts in time, we obtain
\begin{align*}
	\begin{split}
	\LL_{12}
		&
		=	-
		\underbrace{	
			\int_0^T
			\int_\Omega
			J^{-1}  a^s_i a^k_r
			\bp^3 \pt^2 v^r
			( \rhol^2 J^{-2} ),_k
			\bp^3 \pt^2 \eta^i,_s
		}_{\LL_{121}}
		+
		\underbrace{
			\int_\Omega J^{-1}  a^s_i a^k_r
			\bp^3 \pt^2 \eta^r 
			( \rhol^2 J^{-2} ),_k
			\bp^3 \pt^2 \eta^i,_s
			\big|_0^T
		}_{\LL_{122}}
		\\&\indeq
		-
		\underbrace{	
			\int_0^T
			\int_\Omega
			(	J^{-1}  a^s_i a^k_r
			( \rhol^2 J^{-2} ),_k)_t
			\bp^3 \pt^2  \eta^r
			\bp^3 \pt^2 \eta^i,_s
		}_{\LL_{123}}
		.
	\end{split}
\end{align*}
It is clear that
$
-
\LL_{121}
+
\int_0^T
\LL_{2}
=
0
$.
Note that from \eqref{EQ20} and the Leibniz rule it follows that
\begin{align}
	\begin{split}
	\bp^3 \pt^2 J 
	&
	=
	a^s_i \bp^3 \pt^2  \eta^i,_s
	+
	\bp^3 \pt a^s_i  \pt \eta^i,_s
	+
	\pt a^s_i \bp^{3} \pt \eta^i,_s
	\\&\indeq
	+
	\sum_{l=0}^3 \sum_{\substack{m=0\\ 1\leq l+m\leq 3\\(l,m) \neq (0,1)}}^1 
	c_{l,m} \bp^l \pt^m a^s_i \bp^{3-l} \pt^{2-m} \eta^i,_s
	.
	\label{EQ289}
	\end{split}
\end{align}
Inserting \eqref{EQ289} into the term $\LL_{122}$ and using the Young, H\"older, Sobolev inequalities and the fundamental theorem of calculus, we arrive at
\begin{align*}
	\begin{split}
		\LL_{122}
		&
		\les
		1+
		\delta
		\Vert  \bp^3 \pt^2 J (T)\Vert_0^2 
		+
		\delta
		\Vert
		\bp^3 \pt a \pt D\eta 
		\Vert_0^2 (T)
		+
		C_\delta
		\Vert  \bp^3 \pt^2 \eta (T)\Vert_{0}^2 
		\\&\indeq
		+
		\sum_{l=0}^3 \sum_{\substack{m=0\\1\leq l+m\leq 3, \\(l,m)\neq (0,1)}}^1 
		\Vert
		\bp^l \pt^m a \bp^{3-l} \pt^{2-m} D\eta 
		\Vert_0^2  (T)
		\les
		C_\delta
		+
		\delta \sup_{t\in [0,T]} E(t)
		+
		C_\delta T P(\sup_{t\in [0,T]} E(t))
		,
	\end{split}
\end{align*}
where we recall that $\Vert v\Vert_{4}^2 +\Vert \pt^2 v\Vert_{3}^2 + \Vert \pt^2 J\Vert_3^2 \leq E(t)$.
For the term $\LL_{123}$, we appeal to the H\"older and Sobolev inequalities and Lemma~\ref{duality}, obtaining
\begin{align*}
	\begin{split}
		\LL_{123}
		&
		\les
		\int_0^T
		\Vert \bp^3 \pt^2 \eta\Vert_{0.5}
		\Vert \bp^3 \pt^2 D \eta\Vert_{-0.5}
		\les
		\int_0^T
		\Vert \bp^2 \pt^2 \eta\Vert_{1.5}^2
		\les
		T P(\sup_{t\in [0,T]} E(t))
		,
	\end{split}
\end{align*}
since $\Vert \bp^2 \pt^2 \eta\Vert_{1.5}^2 \leq E(t)$.
The term $\LL_{13}$ in \eqref{EQ930} is estimated using the H\"older and Sobolev inequalities as
\begin{align*}
	\begin{split}
		\LL_{13}
		\les
		\int_0^T
		(1+ \Vert D^2 J\Vert_{L^\infty})
		\Vert \bp^3 \pt^2 \eta\Vert_0
		\Vert \bp^3 \pt^2 v\Vert_0
		\les
		T P(\sup_{t \in [0,T]} E(t))
		,
	\end{split}
\end{align*}
since $\Vert J\Vert_4^2 \leq E(t)$.
The term $\LL_{3}$ in \eqref{EQ940} consists of essentially lower order terms which can be estimated using the H\"older and Sobolev inequalities and the fundamental theorem of calculus as
\begin{align*}
	\int_0^T
	\LL_{3}
	\les
	TP(\sup_{t\in [0,T]} E(t))
	.
\end{align*}
Collecting \eqref{EQ691}, \eqref{EQ992}--\eqref{EQ995}, and \eqref{EQ996} and the above estimates, we conclude that
\begin{align*}
	\begin{split}
		&
		\frac{1}{2}
		\int_\Omega
		\rho_0
		|\bp^3 \pt^2 v|^2 (T)
		+
		\rhog^2 
		J^{-1} 
		|D_\eta \bp^3 \pt^2 \eta (T)|^2
		-
		\rhog^2
		J^{-1} 
		| \curl_\eta \bp^3 \pt^2 \eta (T)|^2
		-
		\rhog^2
		J^{-3} 
		|\bp^3 \pt^2 J(T)|^2
		\\&\indeq
		+
		\frac{ 1}{4}
		\int_\Gamma
		|\bp^3 \pt^2 \eta^r  a^3_r |^2 (T)
		+
		\int_0^T
		\mathcal{J}_2
		\les
		C_\delta
		+
		\delta \sup_{t\in [0,T]} E(t)
		+
		C_\delta T P(\sup_{t\in [0,T]} E(t))
		+
		\int_0^T
		\mathcal{J}_3
		.
	\end{split}
\end{align*}
Using $|\bp^3 \pt^2 \eta^r a^3_r| \geq |\bp^3 \pt^2 \eta^3 a^3_3| - |\bp^3 \pt^2 \eta^\alpha a^3_\alpha|$ and $\Vert a - I_3 \Vert_2 \les T$, we obtain 
\begin{align}
	\begin{split}
		&
		\frac{1}{2}
		\int_\Omega
		\left(
		\rho_0
		|\bp^3 \pt^2 v|^2 (T)
		+
		\rhog^2 
		J^{-1} 
		|D_\eta \bp^3 \pt^2 \eta (T)|^2
		-
		\rhog^2
		J^{-1} 
		| \curl_\eta \bp^3 \pt^2 \eta (T)|^2
		-
		\rhog^2
		J^{-3} 
		|\bp^3 \pt^2 J(T)|^2
		\right)
		\\&\indeq
		+
		\frac{ 1}{8}
		\int_\Gamma
		|\bp^3 \pt^2 \eta  |^2 (T)
		+
		\int_0^T
		\mathcal{J}_2
		\les
		C_\delta
		+
		\delta \sup_{t\in [0,T]} E(t)
		+
		C_\delta T P(\sup_{t\in [0,T]} E(t))
		+
		\int_0^T
		\mathcal{J}_3
		.
		\label{EQ994}
	\end{split}
\end{align}
where we used \eqref{EQ982} in the last inequality.

\textit{Estimate of $\JJ_2$ in \eqref{EQ691}}:
We integrate by parts in $\partial_k$ and use the Piola identity \eqref{piola} to get
\begin{align}
	\begin{split}
		\mathcal{J}_2
		&
		=
		\int_{\Gamma}
		a^3_i \rho_0^2 \bp^3 \pt^2J^{-2}
		\bp^3 \pt^2 v^i
		-
		\int_{\{x_3=0\}}
		a^3_i \rho_0^2 \bp^3 \pt^2 J^{-2}
		\bp^3 \pt^2 v^i 
		-
		\int_\Omega
		a^k_i \rho_0^2 \bp^3 \pt^2 J^{-2}
		\bp^3 \pt^2 v^i,_k
		\\&
		=
		-
		\int_\Omega
		a^k_i \rho_0^2 \bp^3 \pt^2 J^{-2}
		\bp^3 \pt^2 v^i,_k
		,		
		\label{EQ963}
	\end{split}
\end{align}
since $\bp^3 \pt^2 J^{-2}=0$ on $\Gamma \times [0,T]$ and $a^3_i \bp^3 \pt^2 v^i = 0$ on $\{x_3=0\} \times [0,T]$.
Note that from \eqref{EQ20} and the Leibniz rule it follows
\begin{align}
	\begin{split}
	\bp^3 \pt^3 J
	&
	=
	a^k_i \bp^3 \pt^2 v^i,_k
	+
	v^i,_k
	\bp^3 \pt^2 a^k_i 
	+
	2
	\pt a^k_i  \bp^{3} \pt v^i,_k
	\\&\indeq
	+
	\sum_{l=0}^3
	\sum_{\substack{m=0\\ 1\leq m+l\leq 4\\(l,m) \neq (0,1)}}^2
	c_{l,m}
	\bp^l \pt^m a^k_i  \bp^{3-l} \pt^{2-m} v^i,_k
	.
	\label{EQ980}
	\end{split}
\end{align}
Combing \eqref{EQ963}--\eqref{EQ980} , we arrive at
\begin{align*}
	\begin{split}
		\mathcal{J}_2
		&
		=
		2\int_\Omega
		\rho_0^2
		a^k_i 
		J^{-3} \bp^3 \pt^2 J	
		\bp^3 \pt^2 v^i,_k
		+
		\underbrace{
			\sum_{l=0}^3
		\sum_{\substack{m=0\\l+m\geq 1}}^1
		c_{l,m}
		\int_\Omega
		\rho_0^2
		a^k_i  
		\bp^l \pt^m J^{-3}
		\bp^{3-l} \pt^{2-m} J	\bp^3 \pt^2 v^i,_k}_{\JJ_{25}}
		\\
		&
		=
		\underbrace{2\int_\Omega
			\rho_0^2
			J^{-3}
			\bp^3 \pt^2 J
			\bp^3 \pt^3 J
		}_{\mathcal{J}_{21}}
			-
	\underbrace{
		\int_\Omega
		\rho_0^2
		J^{-3}
		\bp^3 \pt^2 J
		v^i,_k \bp^3 \pt^2 a^k_i 
	}_{\mathcal{J}_{22}}
-
\underbrace{
	2
	\int_\Omega
	\rho_0^2
	J^{-3}
	\bp^3 \pt^2 J
	\pt a^k_i \bp^3 \pt v^i,_k 
}_{\mathcal{J}_{23}}
\\&\indeq
	+
	\underbrace{
		\sum_{l=0}^3
		\sum_{\substack{m=0\\ 1\leq m+l\leq 4\\ (l,m) \neq (0,1)}}^2
		c_{l,m}
		\int_\Omega
		\rho_0^2 J^{-3}
		\bp^3 \pt^2 J
		\bp^l \pt^m a^k_i  \bp^{3-l} \pt^{2-m} v^i,_k
	}_{\mathcal{J}_{24}}
	+
	\JJ_{25}
		.
	\end{split}
\end{align*}
The term $\JJ_{21}$ can be rewritten as
\begin{align}
	\begin{split}
		\mathcal{J}_{21}
		&
		=
		\frac{d}{dt}
		\int_\Omega
		\rho_0^2
		J^{-3}
		|\bp^3 \pt^2 J|^2
		-
		\underbrace{\int_\Omega
		(\rho_0^2 J^{-3})_t
		|\bp^3 \pt^2 J|^2}_{\JJ'_{21}}
		,
		\label{EQ910}
	\end{split}
\end{align}
where the term $\JJ'_{21}$ satisfies
\begin{align*}
	\int_0^T
	\JJ'_{21}
	\les
	TP(\sup_{t\in [0,T]} E(t))
	.
\end{align*}
For the term $\mathcal{J}_{22}$, we integrate by parts in time, obtaining
\begin{align}
	\begin{split}
	-\int_0^T \JJ_{22}
	=	
	\underbrace{\int_0^T		
	\int_\Omega
	\bp^3 \pt J
	(	\rho_0^2
	J^{-3}
	v^i,_k \bp^3 \pt^2 a^k_i )_t}_{\JJ_{221}}
	+
\underbrace{	\int_\Omega
	\bp^3 \pt J
	\rho_0^2
	J^{-3}
	v^i,_k \bp^3 \pt^2 a^k_i 
	|^{t=0}_{t=T}}_{\JJ_{222}}
.
\label{EQ201}
	\end{split}
\end{align}
For the term $\JJ_{221}$, we integrate by parts in $\bp$ to get
\begin{align}
	\begin{split}
	\JJ_{221}
	&
	=	
	\int_0^T		
	\int_\Omega
	\bp^3 \pt J
	(\rho_0^2
	J^{-3} v^i,_k ),_t
	\bp^3 \pt^2 a^k_i 
	+
	\int_0^T		
	\int_\Omega
	\bp^3 \pt J
		\rho_0^2
	J^{-3}
	v^i,_k \bp^3 \pt^3 a^k_i 
	\\&
	=
	-
	\int_0^T		
	\int_\Omega
	\bp (\bp^3 \pt J
	(\rho_0^2
	J^{-3}
	v^i,_k),_t )
	\bp^2 \pt^2 a^k_i 
	-
	\int_0^T		
	\int_\Omega
	\bp (\bp^3 \pt J
	\rho_0^2
	J^{-3}
	v^i,_k )
	\bp^2 \pt^3 a^k_i 
	,
	\end{split}
\end{align}
since integration by parts in $\bp$ does not produce any boundary terms. Therefore,
using the H\"older and Sobolev inequalities and the fundamental theorem of calculus, we obtain
\begin{align}
	\JJ_{221}
	\les
	TP(\sup_{t \in [0,T]} E(t))
	,
\end{align}
since $\Vert \pt J\Vert_4^2 + \Vert \pt^2 v\Vert_3^2 \leq E(t)$.
For the term $\JJ_{222}$, we integrate by parts in $\bp$, obtaining
\begin{align}
	\begin{split}
	\JJ_{222}
	&
	=	
	\int_\Omega
	\bp (\bp^3 \pt J
	\rho_0^2
	J^{-3}
	v^i,_k )
	\bp^2 \pt^2 a^k_i 
	|^{t=T}_{t=0}
	\\&
	=
		\int_\Omega
	\bp^4 \pt J
	\rho_0^2
	J^{-3}
	v^i,_k 
	\bp^2 \pt^2 a^k_i 
	|^{t=T}_{t=0}
	+
		\int_\Omega
	\bp^3 \pt J
	\bp (\rho_0^2
	J^{-3}
	v^i,_k )
	\bp^2 \pt^2 a^k_i 
	|^{t=T}_{t=0}
	.
	\end{split}
\end{align}
From the Young, H\"older, and Sobolev inequalities and the fundamental theorem of calculus it follows that
\begin{align}
	\begin{split}
	\JJ_{222}
	\les
	C_\delta
	+
	\delta \sup_{t\in [0,T]} E(t)
	+
	C_\delta T P(\sup_{t\in [0,T]} E(t))
	,
	\label{EQ203}
	\end{split}
\end{align}
since $\Vert \pt^2 J\Vert_3^2 + \Vert \pt^2 v\Vert_3^2 \leq E(t)$.
For the term $\JJ_{23}$, we proceed analogously as in \eqref{EQ201}--\eqref{EQ203}, obtaining
\begin{align*}
	\begin{split}
		-
		\int_0^T
		\JJ_{23}
		\les
		C_\delta
		+
		\delta \sup_{t\in [0,T]} E(t)
		+
		C_\delta T P(\sup_{t\in [0,T]} E(t))
		.
	\end{split}
\end{align*}
The term $\JJ_{24}$ consists of essentially lower order terms which can be estimated using the H\"older and Sobolev inequalities and the fundamental theorem of calculus as
\begin{align*}
	\begin{split}
		\int_0^T
	\JJ_{24}
	\les
		C_\delta
	+
	\delta \sup_{t\in [0,T]} E(t)
	+
	C_\delta T P(\sup_{t\in [0,T]} E(t))
	.
	\end{split}
\end{align*}
For the term $\JJ_{25}$, we integrate by parts in $\partial_k$, obtaining
\begin{align}
	\begin{split}
		\JJ_{25}
	&	=
	\sum_{l=0}^3
	\sum_{\substack{m=0\\l+m\geq 1}}^1
	c_{l,m}
	\int_\Omega
	(\rho_0^2
	a^k_i  
	\bp^l \pt^m J^{-3}
	\bp^{3-l} \pt^{2-m} J	),_k
	\bp^3 \pt^2 v^i
	,
	\label{EQ220}
	\end{split}
\end{align}
from where
\begin{align*}
	\int_0^T
	\JJ_{25}
	\les
	1+TP(\sup_{t \in [0,T]} E(t))
	.
\end{align*}
Combining \eqref{EQ994}, \eqref{EQ910}, and the above estimates, we arrive at
\begin{align}
	\begin{split}
		&
		\frac{1}{2}
		\int_\Omega
		(
		\rho_0
		|\bp^3 \pt^2 v|^2 (T)
		+
		\rhog^2 
		J^{-1} 
		|D_\eta \bp^3 \pt^2 \eta (T)|^2
		-
		\rhog^2
		J^{-1} 
		| \curl_\eta \bp^3 \pt^2 \eta (T)|^2
								\\&\indeq
		+
		\rho_0^2
		J^{-3} 
		|\bp^3 \pt^2 J(T)|^2
		)
		+
		\frac{ 1}{8}
		\int_\Gamma
		|\bp^3 \pt^2 \eta  |^2 (T)
		\\&
		\les
		C_\delta
		+
		\delta \sup_{t\in [0,T]} E(t)
		+
		C_\delta T P(\sup_{t\in [0,T]} E(t))
		+
		\int_0^T
		\mathcal{J}_3
		,
		\label{EQ960}
	\end{split}
\end{align}
since $\rhog^2 \leq \rho_0^2$.

\textit{Estimate of $\JJ_3$ in \eqref{EQ691}}:
The term $\JJ_3$ consists of essentially lower order terms which can be treated in a similar fashion as in \eqref{EQ220}, and we obtain
\begin{align}
	\begin{split}
		\JJ_{3}
		\les
		1+T P(\sup_{t\in [0,T]} E(t))
		.
		\label{EQ961}
	\end{split}
\end{align}
\textit{Concluding the proof}:
Combining \eqref{EQ960}--\eqref{EQ961}, we get
\begin{align*}
	\begin{split}
		&
		\frac{1}{2}
		\int_\Omega
		\left(
		\rho_0
		|\bp^3 \pt^2 v|^2 (T)
		+
		\rhog^2 
		J^{-1} 
		|D_\eta \bp^3 \pt^2 \eta (T)|^2
		-
		\rhog^2
		J^{-1} 
		| \curl_\eta \bp^3 \pt^2 \eta (T)|^2
		+
		\rho_0^2
		J^{-3} 
		|\bp^3 \pt^2 J(T)|^2
		\right)
		\\&\indeq
		+
		\frac{ 1}{8}
		\int_\Gamma
		|\bp^3 \pt^2 \eta  |^2 (T)
		\les
		C_\delta
		+
		\delta \sup_{t\in [0,T]} E(t)
		+
		C_\delta T P(\sup_{t\in [0,T]} E(t))
		.
	\end{split}
\end{align*}
The proof of the lemma is thus completed by the curl estimate in Lemma~\ref{Lcurl}.
\end{proof}

The following lemma provides $\bp^2 \pt^3$, $\bp \pt^4$, and $\pt^5$ energy estimates of the solutions $\eta$, $J$, and $v$.
\cole
\begin{Lemma}
\label{Lenergy3}
For $\delta \in (0,1)$, we have 
\begin{align}
\begin{split}
	&
	\sup_{t\in [0,T]}
	\sum_{l=3}^5
	\int_\Omega 
	\left(
	|\rhog \bp^{5-l} \pt^l D\eta (t)|^2
	+
	|\bp^{5-l} \pt^{l} v(t)|^2	
	+
	|\bp^{5-l} \pt^{l} J (t)|^2
	\right)
		\\&\indeq
	+
	\sup_{t\in [0,T]}
	\sum_{l=3}^5
	\int_\Gamma
	|\bp^{5-l} \pt^{l} \eta (t)|^2
	\les
	C_\delta
	+
	\delta \sup_{t\in [0,T]} E(t)
	+
	C_\delta T P(\sup_{t\in [0,T]} E(t))
	,
	\label{EQ108}
\end{split}
\end{align}
where $C_\delta>0$ is a constant depending on $\delta$.
\end{Lemma}
\colb

\begin{proof}[Proof of Lemma~\ref{Lenergy3}]
Similar arguments as in Lemma~\ref{Lenergy2} lead to
\begin{align}
	\begin{split}
		&
		\sup_{t\in [0,T]}
		\sum_{l=3}^4
		\int_\Omega 
		\left(
		|\rhog \bp^{5-l} \pt^l D\eta (t)|^2
		+
		|\bp^{5-l} \pt^{l} v(t)|^2	
		+
		|\bp^{5-l} \pt^{l} J (t)|^2
		\right)
				\\&\indeq
		+
		\sup_{t\in [0,T]}
		\sum_{l=3}^4
		\int_\Gamma
		|\bp^{5-l} \pt^{l} \eta (t)|^2
		\les
		C_\delta
		+
		\delta \sup_{t\in [0,T]} E(t)
		+
		C_\delta T P(\sup_{t\in [0,T]} E(t))
		,
		\label{EQ908}
	\end{split}
\end{align}
It remains to establish the $\pt^5$ energy estimates.
Applying $\pt^5$ to \eqref{EQ01} and taking the inner product with $\pt^5 v$, we obtain
\begin{align}
	\begin{split}
		&
		\frac{1}{2} \frac{d}{dt}
		\int_\Omega \rho_0 
		|\pt^5 v|^2 
		+
		\underbrace{\int_\Omega 
			\pt^5 a^k_i  
			(\rho_0^2 J^{-2}),_k 
			\pt^5 v^i}_{\JJ_1}
		+
		\underbrace{\int_\Omega a^k_i 
			(\rho_0^2 \pt^5 J^{-2}),_k 
			\pt^5 v^i}_{\JJ_2}
		\\&\indeq
		=
		\underbrace{\sum_{l=1}^{4}
			c_l
			\int_\Omega
			\partial_t^l a_i^k 
			(\rho_0^2 \pt^{5-l}  J^{-2}),_k 
			\pt^5 v^i}_{\JJ_3}
		.
		\label{EQ414}
	\end{split}
\end{align}

\textit{Estimate of $\JJ_1$ in \eqref{EQ414}}: 
Using the decomposition $\rho_0^2= \rhog^2+ \rhol^2$ (see \eqref{EQ184}--\eqref{EQ185}), we split the term $\mathcal{J}_1$ as
\begin{align}
	\begin{split}
	\mathcal{J}_1
	&
	=
	\underbrace{
	\int_\Omega 
	\pt^5 a^k_i 
	( \rhog^2 J^{-2}),_k 
	\bp^3 \pt^2  v^i
	}_{\mathcal{G}}
	+
	\underbrace{
	\int_\Omega 
	\pt^5  a^k_i 
	( \rhol^2 J^{-2} ),_k 
	\bp^3 \pt^2  v^i	
	}_{\mathcal{L}}
	.
	\label{EQ911}
	\end{split}
\end{align}
For the term $\GG$, we proceed in a similar fashion as in Lemma~\ref{Lenergy2}, obtaining
\begin{align}
	\begin{split}
		&
		\int_0^T
		\GG
		=
		\frac{1}{2}
		\int_\Omega
		\left(
		\rhog^2 
		J^{-1} 
		|D_\eta \pt^5 \eta (T)|^2
		-
		\rhog^2
		J^{-1} 
		| \curl_\eta \pt^5 \eta (T)|^2
		-
		\rhog^2
		J^{-3} 
		|\pt^5 J(T)|^2
		\right)
		+
		\GG'
		,
		\label{EQ792}	
	\end{split}
\end{align}
where $\GG'$ consists of the terms satisfying
\begin{align}
	\int_0^T
	\GG'
	\les
	C_\delta
	+
	\delta \sup_{t\in [0,T]} E(t)
	+
	C_\delta T P(\sup_{t\in [0,T]} E(t))
	.
	\label{EQ795}	
\end{align}
Next we estimate the term $\LL$ in \eqref{EQ911}.
Using \eqref{EQ15} and the Leibniz rule, we get
\begin{align*}
	\begin{split}
		\LL
		&
		=
		-
		\underbrace{\int_\Omega 
		\pt^5 \eta^r,_s J^{-1}  a^s_i a^k_r
		( \rhol^2  J^{-2}),_k 
		\pt^5 v^i}_{\LL_{1}}
		+
		\underbrace{\int_\Omega 
			\pt^5 \eta^r,_s J^{-1} a^s_r a^k_i 
			(\rhol^2  J^{-2}),_k 
			\pt^5 v^i}_{\LL_{2}}
		\\&\indeq
		+
		\underbrace{	
			\sum_{l=1}^4
			c_l
			\int_\Omega
			\partial_t^l
			(J^{-1} (a^s_r a^k_i - a^s_i a^k_r) )
			\pt^{5-l}
			\eta^r,_s
			( \rhol^2  J^{-2}),_k 
			\pt^5 v^i}_{\LL_{3}}
		.
	\end{split}
\end{align*}
For the term $\LL_{1}$, we integrate by parts in $\partial_s$, leading to
\begin{align*}
	\begin{split}
	-
	\int_0^T 
	\LL_{1}
	&
	=
	\underbrace{	
	\int_0^T \int_{\{x_3 =0\}}
	\partial_t^{5} \eta^r J^{-1} a^3_i a^k_r 
	( \rhol^2  J^{-2}),_k \pt^5 v^i}_{=0}
	-
	\underbrace{
	\int_0^T \int_\Gamma
	\pt^5 \eta^r J^{-1} a^3_i a^3_r 
	( \rhol^2  J^{-2}),_3 \pt^5 v^i
}_{\LL_{11}}
	\\&\indeq
	+
	\underbrace{\int_0^T	
	\int_\Omega \pt^5 
	\eta^r J^{-1}  a^s_i a^k_r
	( \rhol^2 J^{-2}),_k 
	\pt^5 v^i,_s}_{\LL_{12}}
	+
	\underbrace{\int_0^T	
	\int_\Omega 
	\pt^5 \eta^r 
	(J^{-1}  a^s_i a^k_r
	( \rhol^2 J^{-2}),_k ),_s
	\pt^5 v^i}_{\LL_{13}}
	,
	\end{split}
\end{align*}
since $a^3_i \pt^5 v^i = 0$ on $\{x_3=0\} \times [0,T]$ and $\rhol^2 J^{-2}=1$ on $\Gamma \times [0,T]$.
The term $\LL_{11}$ can be rewritten as
\begin{align*}
	\begin{split}
		-\LL_{11}
		&
		=
		\underbrace{-
			\frac{1}{2}
			\int_0^T
			\frac{d}{dt}
			\int_\Gamma
			J^{-1} 	
			(\rhol^2  J^{-2} ),_3
			|\pt^5 \eta^r a^3_r |^2
		}_{\LL_{111}}
		+
		\underbrace{
			\frac{1}{2}
			\int_0^T
			\int_\Gamma
			(J^{-1} 	
			(\rhol^2  J^{-2} ),_3)_t
			|\pt^5 \eta^r a^3_r |^2
		}_{\LL_{112}}
		\\
		&\indeq
		+
		\underbrace{
			\int_0^T
			\int_\Gamma
			J^{-1} 	
			(\rhol^2  J^{-2} ),_3
			\pt^5 \eta^i
			a^3_i
			\pt^5 \eta^r 
			\pt a^3_r
		}_{\LL_{113}}
		.
	\end{split}
\end{align*}
From \eqref{EQ284} it follows that
\begin{align}
	\begin{split}
		\LL_{111}
		&
		=	
		-\frac{1}{2}
		\int_\Gamma
		J^{-1} 
		(\rhol^2 J^{-2} ),_3
		|\bp^5 \eta^r  a^3_r |^2
		(T)
		+
		M_0
		\geq
		\frac{1}{4}
		\int_\Gamma
		|\bp^5 \eta^r a^3_r |^2 (T)
		+
		M_0
		.
		\label{EQ223}
	\end{split}
\end{align}
Using the H\"older and Sobolev inequalities, we get
\begin{align*}
	\begin{split}
	\LL_{112}
	+
	\LL_{113}
	\les
	\int_0^T
	(1+ \Vert DJ_t\Vert_{L^\infty})
	|\pt^5 \eta|_0^2
	\les
	T P(\sup_{t \in [0,T]} E(t))
	,
	\end{split}
\end{align*}
since $\Vert \pt J\Vert_4^2 + \vert \pt^5 \eta \vert_0^2\leq E(t)$.
For the term $\LL_{12}$, we integrate by parts in time, obtaining
\begin{align*}
	\begin{split}	
		\LL_{12}
		&
		=
		\underbrace{
			\int_\Omega \pt^5 \eta^r J^{-1} 
			a^s_i a^k_r
			( \rhol^2  J^{-2}),_k 
			\pt^5 \eta^i,_s
			\big|_0^T}_{\LL_{121}}
		-
		\underbrace{\int_0^T	
			\int_\Omega 
			\pt^6 \eta^r J^{-1}  a^s_i a^k_r
			( \rhol^2  J^{-2}),_k 
			\pt^5 \eta^i,_s}_{\LL_{122}}
		\\&\indeq
		-
		\underbrace{\int_0^T	
			\int_\Omega 
			(	 J^{-1}  a^s_i a^k_r
			( \rhol^2  J^{-2}),_k )_t
			\pt^5 \eta^r
			\pt^5 \eta^i,_s}_{\LL_{123}}
		,
	\end{split}
\end{align*}
Note that from \eqref{EQ20} and the Leibniz rule it follows that
\begin{align}
	\pt^5 J 
	= 
	a^s_i \pt^5 \eta^i,_s
	 + 
	 \sum_{l=1}^4 c_l \pt^l a^s_i \pt^{5-l} \eta^i,_s
	 .
	 \label{EQ221}
\end{align}
Inserting \eqref{EQ221} to the term $\LL_{121}$ and using the Young, H\"older, and Sobolev inequalities and the fundamental theorem of calculus, we obtain
\begin{align*}
	\begin{split}
		\LL_{121}
		&
		\les
		\delta\Vert \pt^5 J(T)\Vert_0^2
		+
		C_\delta 
		\Vert \pt^5 \eta(T) \Vert_0^2
		+
		\sum_{l=1}^4
		\Vert \pt^l a \pt^{5-l} D \eta \Vert_0^2 (T)
		+
		1
		\\&
		\les
		C_\delta
		+
		\delta \sup_{t\in [0,T]} E(t)
		+
		C_\delta T P(\sup_{t\in [0,T]} E(t))
		,
	\end{split}
\end{align*}
where we recall that $\Vert \pt^5 J\Vert_0^2 + \Vert \pt^3 v\Vert_2^2 + \Vert \pt^4 v \Vert_1^2 + \Vert \pt^5 v\Vert_0^2 \leq E(t)$.
It is clear that $-\mathcal{L}_{122}+\int_0^T \mathcal{L}_{2} = 0$.
For the term $\mathcal{L}_{123}$, we use the H\"older and Sobolev inequalities, obtaining
\begin{align*}
	\begin{split}
	-\LL_{123}	
	\les
	\int_0^T
	(1+ \Vert DJ_t\Vert_{L^\infty})
	\Vert \pt^5 \eta\Vert_0
	\Vert \pt^5 D\eta\Vert_0
	\les
	T P(\sup_{t\in [0,T]} E(t))
	,
	\end{split}
\end{align*}
since $\Vert \pt^4 v\Vert_1^2 \leq E(t)$.
Similarly, the terms $\mathcal{L}_{13}$ and $\mathcal{L}_3$ are estimated as
\begin{align*}
	\LL_{13}
	+
	\int_0^T 
	\mathcal{L}_{3}
	\les
	 T P(\sup_{t\in [0,T]} E(t))	
	.
\end{align*}
Combining \eqref{EQ414}, \eqref{EQ792}--\eqref{EQ223}, and the above estimates, we conclude
\begin{align}
	\begin{split}
		&
		\frac{1}{2}
		\int_\Omega
		\left(
		\rho_0 |\pt^5 v(T)|^2
		+
		\rhog^2 
		J^{-1} 
		|D_\eta \pt^5 \eta (T)|^2
		-
		\rhog^2
		J^{-1} 
		| \curl_\eta \pt^5 \eta (T)|^2
		-
		\rhog^2
		J^{-3} 
		|\pt^5 J(T)|^2
		\right)
		\\
		&\indeq
		+
		\frac{ 1}{4}
		\int_\Gamma
		|\pt^5 \eta^r a^3_r  |^2 (T)
		+
		\int_0^T
		\JJ_2
		\les
		C_\delta
		+
		\delta \sup_{t\in [0,T]} E(t)
		+
		C_\delta T P(\sup_{t\in [0,T]} E(t))
		+
		\int_0^T
		\JJ_3
		.
		\label{EQ415}
	\end{split}
\end{align}

Now we claim that 
\begin{align}
	\begin{split}
	\sup_{t\in [0,T]}
	\vert \pt^5 \eta^\alpha (t)
	\vert_{0}^2
	\les
	1+T P(\sup_{t\in [0,T]} E(t))
	,
	\label{EQ382}
	\end{split}
\end{align}
where $\alpha=1,2$.
Applying $\pt^3$ to \eqref{EQLJ}, we get 
\begin{align}
	\begin{split}
	\vert  \pt^5  \eta^\alpha \vert_0^2
	&
	\les
	\underbrace{	
	\vert (a^3_3)^{-1} a^3_\alpha 
	\pt^5 \eta^3 \vert_0^2}_{\II_1}
	+
	\underbrace{ 
	\vert \pt^3 
	((a^3_3)^{-1} a^3_\alpha )
	\pa_t^2 \eta^3 \vert_0^2
	}_{\II_2} 
	+
	\underbrace{
	\vert  \pt ((a^3_3)^{-1} a^3_\alpha )
	\pt^4 \eta^3 \vert_0^2
	}_{\II_{3}}
	\\&\indeq
	+
	\underbrace{
	\vert  \pt^2  ((a^3_3)^{-1} a^3_\alpha )
	\pt^3 \eta^3 \vert_0^2
	}_{\II_{4}}
\label{EQ226}
	\end{split}
\end{align}
The term $\II_1$ is estimated using the H\"older and Sobolev inequalities as
\begin{equation} 
	\mathcal{I}_1 
	\les 
	\Vert a^3_\alpha \Vert_{L^\infty (\Omega)}
	\vert \pt^{5} \eta^3  
	\vert_0^2 
	\les 
	T P(\sup_{t\in [0,T]} E(t)),
	\label{EQ384}
\end{equation}
where we used $\Vert a- I_3 \Vert_{2} \les T$ and $\vert \pt^5 \eta \vert_0^2\leq E(t)$.
For the term $\mathcal{I}_2$, using the Leibniz rule, we get
\begin{align}
	\begin{split}
	\mathcal{I}_2
	&
	\les
	\underbrace{
	\vert \pt^3 (a^3_3)^{-1}
	a^3_\alpha
	\vert_0^2}_{\II_{21}}
	+	
\underbrace{
	\vert  (a^3_3)^{-1}
	\pt^{3} a^3_\alpha
	\vert_0^2}_{\II_{22}}
	+	
	\underbrace{
	\sum_{l=1}^2
	\vert \pt^l (a^3_3)^{-1}
	\pt^{3-l} a^3_\alpha
	\vert_0^2}_{\II_{23}}
	.
	\label{EQ389}
	\end{split}
\end{align}
We bound the term $\II_{21}$ using the H\"older and Sobolev inequalities and the fundamental theorem of calculus as
\begin{align}
	\II_{21}
	\les
	\vert \pt^3 
	(a^3_3)^{-1}
	\vert_0^2
	\Vert
	a^3_\alpha
	\Vert_{L^\infty }^2
	\les
	TP(\sup_{t \in [0,T]} E(t))
	.
	\label{EQ228}
\end{align}
Similarly, we estimate the term $\II_{22}$ as
\begin{align}
	\begin{split}
	\II_{22}
	\les
	\vert \pt^3 (\bp \eta \bp \eta) 
	\vert_0^2	
	\les
	1+TP(\sup_{t\in [0,T]} E(t))
	,
	\label{EQ224}
	\end{split}
\end{align}
since $\vert \pt^4 \bp \eta\vert_0^2 
\leq E(t)$.
The term $\II_{23}$ consists of lower-order terms which can be estimated in a similar fashion using the H\"older and Sobolev inequalities and the fundamental theorem of calculus as
\begin{align}
	\II_{23}
	\les
	1+ TP(\sup_{t \in [0,T]} E(t))
	.
	\label{EQ229}
\end{align}
For the term $\II_3$, we have
\begin{align}
	\begin{split}
	\II_3
	\les
	\vert \pt^4 \eta\vert_0^2
	\les
	1+TP(\sup_{t \in [0,T]} E(t))
	,	
	\label{EQ230}
	\end{split}
\end{align}
since $\vert \pt^5 \eta \vert_0^2 \leq E(t)$.
For the term $\II_4$, we proceed analogously as in \eqref{EQ224}, obtaining
\begin{align}
	\II_4
	\les
	1+TP(\sup_{t \in [0,T]} E(t))
	.
	\label{EQ231}
\end{align}
Collecting \eqref{EQ226}--\eqref{EQ231}, we complete the proof of the claim \eqref{EQ382}.
From \eqref{EQ415}--\eqref{EQ382} and using $|\pt^5 \eta^r a^3_r| \geq |\pt^5 \eta^3 a^3_3| - |\pt^5 \eta^\alpha a^3_\alpha|$ and $\Vert a - I_3 \Vert_2 \les T$, we conclude
\begin{align}
	\begin{split}
		&
		\frac{1}{2}
		\int_\Omega
		\left(
		\rho_0 |\pt^5 v(T)|^2
		+
		\rhog^2 
		J^{-1} 
		|D_\eta \pt^5 \eta (T)|^2
		-
		\rhog^2
		J^{-1} 
		| \curl_\eta \pt^5 \eta (T)|^2
		-
		\rhog^2
		J^{-3} 
		|\pt^5 J(T)|^2
		\right)
		\\
		&\indeq
		+
		\frac{ 1}{8}
		\int_\Gamma
		|\pt^5  \eta (T)|^2
		+
		\int_0^T
		\mathcal{J}_2
		\les
		C_\delta
		+
		\delta \sup_{t\in [0,T]} E(t)
		+
		C_\delta T P(\sup_{t\in [0,T]} E(t))
		+
		\int_0^T
		\mathcal{J}_3
		.
		\label{EQ693}
	\end{split}
\end{align}

\textit{Estimate of $\JJ_2$ in \eqref{EQ414}}:
We integrate by parts in $\partial_k$ and use the Piola identity \eqref{piola} to get
\begin{align*}
	\begin{split}
		\JJ_2
		&
		=
		\int_{\Gamma}
		a^3_i \rho_0^2 \pt^5 J^{-2}
		\pt^5 v^i
		-
		\int_{\{x_3=0\}}
		a^3_i \rho_0^2 \pt^5 J^{-2}
		\pt^5 v^i 
		-
		\int_\Omega
		a^k_i \rho_0^2 \pt^5 J^{-2}
		\pt^5 v^i,_k
		\\
		&
		=
		-
		\int_\Omega
		a^k_i \rho_0^2 \pt^5 J^{-2}
		\pt^5 v^i,_k
		,		
	\end{split}
\end{align*}
since $\pt^5 J^{-2}=0$ on $\Gamma \times [0,T]$ and $a^3_i \pt^5 v^i = 0$ on $\{x_3=0\} \times [0,T]$.
Using 
\begin{align*}
	\pt^6 J 
	= 
	a^k_i \pt^5 v^i,_k 
	+ 
	v^i,_k	\pt^5 a^k_i  
	+ 
	\sum_{l=1}^4 c_l \pt^l a^k_i \pt^{5-l} v^i,_k
	,
\end{align*}
we obtain
\begin{align*}
	\begin{split}
	\JJ_2
	&=
	2\int_\Omega
	a^k_i \rho_0^2
	J^{-3}
	\pt^5 J
	\pt^5 v^i,_k
	+
	\underbrace{\sum_{l=1}^{4}
	c_l
	\int_\Omega
	\rho_0^2 a^k_i 
	\pt^l J^{-3} \pt^{5-l} J
	\pt^5 v^i,_k}_{\JJ_{24}}
	\\&
	=
	\underbrace{2\int_\Omega
	\rho_0^2
	J^{-3}
	\pt^5 J
	\pt^6 J}_{\JJ_{21}}
	-
\underbrace{2
	\int_\Omega
	\rho_0^2
	J^{-3}
	\pt^5 J
	 v^i,_k
	 \pt^5 a^k_i
	}_{\JJ_{22}}
	+
		\underbrace{
	\sum_{l=1}^4
			c_l
	\int_\Omega
	\rho_0^2
	J^{-3}
	\pt^5 J
	\pt^l a^k_i
	\pt^{5-l} v^i,_k}_{\JJ_{23}}
	+
	\JJ_{24}
	.
	\end{split}
\end{align*}
The term $\JJ_{21}$ can be rewritten as
\begin{align}
	\begin{split}
		\int_0^T
	\JJ_{21}
	=
		\int_0^T
	\frac{d}{dt}
	\int_\Omega
	\rho_0^2 J^{-3}
	|\pt^5 J|^2	
	-
\underbrace{	\int_0^T
	\int_\Omega
	(\rho_0^2 J^{-3})_t
	|\pt^5 J|^2}_{\JJ'_{21}}
	,
	\label{EQ564}
	\end{split}
\end{align}
where the term $\JJ'_{21}$ satisfies
\begin{align*}
	\JJ'_{21}
	\les
	T P(\sup_{t\in [0,T]} E(t))
	.
\end{align*}
The highest order term in $\JJ_{22}$ is of the form
$\int_\Omega \rho_0^2 J^{-3} \pt^5 J Dv \pt^5 D\eta$, which can be treated using the H\"older and Sobolev inequalities as
\begin{align*}
	\int_0^T
	\int_\Omega \rho_0^2 J^{-3} \pt^5 J Dv \pt^5 D\eta
	\les
	\int_0^T
	\Vert \pt^5 J\Vert_0
	\Vert \pt^5 D\eta\Vert_0
	\les
	TP(\sup_{t \in [0,T]} E(t))
	 ,
\end{align*}
since $\Vert \pt^4 v\Vert_1^2 + \Vert \pt^5 J\Vert_0^2 \leq E(t)$.
The rest of the terms in $\JJ_{22}$ are of lower order which can be treated in a similar fashion using the H\"older and Sobolev inequalities and the fundamental theorem of calculus. Thus, we obtain
\begin{align*}
	\int_0^T
	\JJ_{22}
	\les
	TP(\sup_{t \in [0,T]} E(t))
	.
\end{align*}
Similarly, the term $\JJ_{23}$ is estimated as
\begin{align*}
	\int_0^T
	\JJ_{23}
	\les
	1+TP(\sup_{t \in [0,T]} E(t))
	.
\end{align*}
For the term $\JJ_{24}$, we integrate by parts in $\partial_k$, obtaining
\begin{align}
	\begin{split}
	\JJ_{24}
	=
	\sum_{l=1}^{4}
	c_l
	\int_\Omega
	(\rho_0^2 a^k_i 
	\pt^l J^{-3} \pt^{5-l} J),_k
	\pt^5 v^i
	.
	\label{EQ240}
	\end{split}
\end{align}
From the H\"older and Sobolev inequalities and the fundamental theorem of calculus it follows that
\begin{align}
	\begin{split}
	\int_0^T
	\JJ_{24}
	\les
	T P(\sup_{t\in [0,T]} E(t))
	,	
	\label{EQ241}
	\end{split}
\end{align}
since $\Vert \pt^4 J\Vert_1^2 + \Vert \pt^3 J\Vert_2^2+ \Vert \pt^5 v\Vert_0^2 \leq E(t)$.
Combining \eqref{EQ693}--\eqref{EQ564} and the above estimates, we conclude
\begin{align}
	\begin{split}
		&
		\frac{1}{2}
		\int_\Omega
		\left(
		\rho_0 |\pt^5 v(T)|^2
		+
		\rhog^2 
		J^{-1} 
		|D_\eta \pt^5 \eta (T)|^2
		-
		\rhog^2
		J^{-1} 
		| \curl_\eta \pt^5 \eta (T)|^2
		+
		\rho_0^2
		J^{-3} 
		|\pt^5 J(T)|^2
		\right)
		\\
		&\indeq
		+
		\frac{1}{8}
		\int_\Gamma
		|\pt^5 \eta (T) |^2
		\les
		C_\delta
		+
		\delta \sup_{t\in [0,T]} E(t)
		+
		C_\delta T P(\sup_{t\in [0,T]} E(t))
		+
		\int_0^T
		\JJ_3
		,
		\label{EQ418}
	\end{split}
\end{align}
since $\rho_0^2 \geq \rhog^2$.

\textit{Estimate of $\JJ_3$ in \eqref{EQ414}}: 
For the term $\JJ_3$, we proceed analogously as in \eqref{EQ240}--\eqref{EQ241}, obtaining
\begin{align}
	\int_0^T 
	\JJ_3
	\les
	T P(\sup_{t\in [0,T]} E(t))
	.
	\label{EQ122}
\end{align}

\textit{Concluding the proof}: 
Combining \eqref{EQ418}--\eqref{EQ122}, we arrive at
\begin{align}
	\begin{split}
		&
		\frac{1}{2}
		\int_\Omega
		\left(
		\rho_0 |\pt^5 v(T)|^2
		+
		\rhog^2 
		J^{-1} 
		|D_\eta \pt^5 \eta (T)|^2
		-
		\rhog^2
		J^{-1} 
		| \curl_\eta \pt^5 \eta (T)|^2
		+
		\rho_0^2
		J^{-3} 
		|\pt^5 J(T)|^2
		\right)
		\\
		&\indeq
		+
		\frac{ 1}{8}
		\int_\Gamma
		|\pt^5 \eta (T) |^2
		\les
		C_\delta
		+
		\delta \sup_{t\in [0,T]} E(t)
		+
		C_\delta T P(\sup_{t\in [0,T]} E(t))
		.
		\label{EQ419}
	\end{split}
\end{align}
From the curl estimates in Lemma~\ref{Lcurl} we complete the $\pt^5$ energy estimates.
Combining \eqref{EQ908} and \eqref{EQ419}, we conclude the proof of the lemma.
\end{proof}

\startnewsection{Normal derivative estimates}{sec05}
In the following three lemmas, we derive the normal derivative estimates of $J$ and $v$.

\cole
\begin{Lemma}
\label{LnormalJ1}
For $\delta \in (0,1)$, we have 
\begin{align}
\begin{split}
	&
	\sup_{t\in [0,T]}
	\left(
	\Vert \pt^4 v (t)\Vert_{1}^2
	+
	\Vert \pt^4 J(t) \Vert_{1}^2
	\right)
	\les
	C_\delta
	+
	\delta \sup_{t\in [0,T]} E(t)
	+
	C_\delta T P(\sup_{t\in [0,T]} E(t))
	,
	\label{EQ42}
\end{split}
\end{align}
where $C_\delta>0$ is a constant depending on $\delta$. 
\end{Lemma}
\colb

\begin{proof}[Proof of Lemma~\ref{LnormalJ1}]
We start with the term $\Vert \pt^4 v\Vert_1^2$.
Using Lemma~\ref{elliptic}, we have
\begin{align*}
	\begin{split}
	\Vert \pt^4 v\Vert_1^2
	\les
	\underbrace{\Vert \pt^5 \dive \eta\Vert_0^2}_{\KK_1}
	+
	\underbrace{\Vert \pt^5 \curl \eta\Vert_0^2}_{\KK_2}
	+
\underbrace{	\Vert \pt^5 \bp \eta \cdot N
	\Vert_{H^{-0.5}(\partial \Omega)}^2}_{\KK_3}
	+
\underbrace{	\Vert \pt^5 \eta\Vert_0^2}_{\KK_4}
	.
	\end{split}
\end{align*}
From the Leibniz rule it follows that
$
\dive \pt^5 \eta
=
\pt^5 J
+
(	\delta^s_r 
-
a^s_r)
\pt^5 \eta^r,_s
+
\sum_{l=1}^{4}
c_l
\pt^l a^s_r \pt^{5-l} \eta^r,_s
$.
Using Lemma~\ref{Lenergy3}, the H\"older and Sobolev inequalities, and the fundamental theorem of calculus, we obtain
\begin{align}
	\begin{split}
	\KK_1
	&
	\les
	\Vert \pt^5 J\Vert_0^2
	+
	\Vert (a-I_3) \pt^5 D\eta\Vert_0^2
	+
	\sum_{l=1}^4
	\Vert \pt^l a \pt^{5-l} D\eta\Vert_0^2
	\\&
	\les
	C_\delta
	+
	\delta \sup_{t\in [0,T]} E(t)
	+
	C_\delta T P(\sup_{t\in [0,T]} E(t))
	,
	\label{EQ999}
	\end{split}
\end{align}
since $\Vert a-I_3\Vert_2 \les T$ and $\Vert \pt^4 v\Vert_1^2 \leq E(t)$.
The term $\KK_2$ is already estimated in Lemma~\ref{Lcurl}.
For the term $\KK_3$, we appeal to Lemma~\ref{normaltrace}, obtaining
\begin{align*}
	\begin{split}
	\KK_3
	\les
	\Vert \pt^5 \bp \eta\Vert_0^2
	+
	\Vert \dive \pt^5 \eta\Vert_0^2
	\les	
	C_\delta
	+
	\delta \sup_{t\in [0,T]} E(t)
	+
	C_\delta T P(\sup_{t\in [0,T]} E(t))
	,
	\end{split}
\end{align*}
where the last inequality follows from Lemma~\ref{Lenergy3} and \eqref{EQ999}.
We bound the term $\KK_4$ using the fundamental theorem of calculus as
\begin{align*}
	\begin{split}
	\KK_4
	\les
	1+ P(\sup_{t\in [0,T]} E(t))
	,	
	\end{split}
\end{align*}
since $\Vert \pt^5 v\Vert_0^2 \leq E(t)$.
Consequently, we conclude the estimate
\begin{align}
	\Vert \pt^4 v\Vert_1^2
	\les
	C_\delta
	+
	\delta \sup_{t\in [0,T]} E(t)
	+
	C_\delta T P(\sup_{t\in [0,T]} E(t))
	.
	\label{EQ255}
\end{align}

Next we establish the estimate of $\Vert \pt^4 J\Vert_1^2$.
We rewrite \eqref{EQ200} as
\begin{align}
	v^i_t 
	-2 \rho_0
	J^{-3}
	a^k_i J,_k
	-2
	a^3_i J^{-2}
	= 0
	,
	\label{EQ100}
\end{align}
since $\rho_0,_3=-1$ and $\rho_0,_\alpha = 0$.
Letting $i=3$ in the above equation and applying $\pt^4$ to the resulting identity, we obtain
\begin{align}
	\begin{split}
		&
		2 \rho_0 a^3_3
		\pt^4 J,_3
		=
				4
		a_3^3 \pt^4 J
		+
		J^3
		\big[ 
		\pt^6 \eta^3
		+
		\rho_0 \pt^4
		(a^\alpha_3 J^{-2},_\alpha)
		+
			\sum_{l=0}^3
		c_l                                         
				\pt^{l} 
		J^{-2} 
		\pt^{4-l} a^3_3
						\\&\indeq
		+
		\sum_{l=1}^4
		c_l  \rho_0
		\pt^l a^3_3
		\pt^{4-l} 
		 J^{-2},_3
		+
		\sum_{l=1}^4
		c_l
		\rho_0 	a^3_3 
		\pt^l J^{-3} 
		\pt^{4-l} J,_3
		+
		\sum_{l=1}^3
		c_l
		a^3_3 \pt^l J^{-3} 
		\pt^{4-l} J
		\big]
		.
		\label{EQ43}
	\end{split}
\end{align}
Now we estimate the terms on the right hand side of \eqref{EQ43} in $L^2$.
Using the H\"older and Sobolev inequalities and the fundamental theorem of calculus, we obtain
\begin{align}
	\begin{split}
	\Vert a^3_3 \pt^4 J\Vert_0^2
	\les
	\Vert a^3_3\Vert_{L^\infty}^2
	\Vert \pt^4 J\Vert_0^2
	\les
	1+TP(\sup_{t \in [0,T]} E(t)),
	\end{split}
\end{align}
since $\Vert \pt^5 J\Vert_0^2 \leq E(t)$.
The term $\Vert  J^3 \pt^6 \eta \Vert_0^2$ is treated using the H\"older and Sobolev inequalities and Lemma~\ref{Lenergy3} as
\begin{align*}
	\Vert J^3 \pt^6 \eta \Vert_0^2
	\les
	C_\delta
	+
	\delta \sup_{t\in [0,T]} E(t)
	+
	C_\delta T P(\sup_{t\in [0,T]} E(t)).
\end{align*}
From the Leibniz rule it follows that
\begin{align}
	\begin{split}
	\Vert J^3 \rho_0 \pt^4
	(a^\alpha_3 J^{-2},_\alpha) \Vert_0^2
	&
	\les
	\Vert  \pt^4 \bp J^{-2} \Vert_0^2
	+
	\sum_{l=1}^4
	\Vert \pt^l a \pt^{4-l} \bp J^{-2} \Vert_0^2
	\\&
	\les
		C_\delta
	+
	\delta \sup_{t\in [0,T]} E(t)
	+
	C_\delta T P(\sup_{t\in [0,T]} E(t))
	,
	\label{EQ112}
	\end{split}
\end{align}
where we appealed to the H\"older and Sobolev inequalities and Lemma~\ref{Lenergy3}.
The rest of the terms on the right hand side of \eqref{EQ43} are of lower order which can estimated in a similar fashion as in \eqref{EQ43}--\eqref{EQ112}.
Collecting the above estimates, we infer from \eqref{EQ43} that
\begin{align*}
	\Vert \pt^4 J,_3\Vert_0^2
	\les
			C_\delta
	+
	\delta \sup_{t\in [0,T]} E(t)
	+
	C_\delta T P(\sup_{t\in [0,T]} E(t))
	.
\end{align*}
Combining the already-established estimates of $\Vert \pt^4 \bp J\Vert_0^2$ in Lemma~\ref{Lenergy3}, we conclude the estimate of $\Vert \pt^4 J\Vert_1^2$.
The proof of the lemma is thus completed by combining \eqref{EQ255}.
\end{proof}

\cole
\begin{Lemma}
\label{LnormalJ2}
For $\delta \in (0,1)$, we have
\begin{align}
	\begin{split}
		&
		\sup_{t\in [0,T]}
		\sum_{l=2}^3
		\left(
		\Vert \pt^l v (t)\Vert_{5-l}^2
		+
		\Vert \pt^l J(t) \Vert_{5-l}^2
		\right)
		\les
		C_\delta
		+
		\delta \sup_{t\in [0,T]} E(t)
		+
		C_\delta T P(\sup_{t\in [0,T]} E(t))
		,
		\label{EQ553}
	\end{split}
\end{align}
where $C_\delta>0$ is a constant depending on $\delta$.
\end{Lemma}
\colb

\begin{proof}[Proof of Lemma~\ref{LnormalJ2}]
We start with the term $\Vert \pt^3 v\Vert_2^2$.
Using Lemma~\ref{elliptic}, we have
\begin{align*}
	\begin{split}
		\Vert \pt^3 v\Vert_2^2
		\les
		\underbrace{\Vert \pt^4 \dive \eta\Vert_1^2}_{\KK_1}
		+
		\underbrace{\Vert \pt^4 \curl \eta\Vert_1^2}_{\KK_2}
		+
		\underbrace{	\Vert \pt^4 \bp \eta \cdot N
			\Vert_{H^{0.5}(\partial \Omega)}^2}_{\KK_3}
		+
		\underbrace{	\Vert \pt^4 \eta\Vert_1^2}_{\KK_4}
		.
	\end{split}
\end{align*}
From the Leibniz rule it follows that
$
\dive \pt^4 \eta
=
\pt^4 J
+
(	\delta^s_r 
-
a^s_r)
\pt^4 \eta^r,_s
+
\sum_{l=1}^3
c_l
\pt^l a^s_r \pt^{4-l} \eta^r,_s
$. 
Thus, we have
\begin{align}
	\begin{split}
	\KK_1
	&
	\les
	\underbrace{\Vert \pt^4 J\Vert_1^2}_{\KK_{11}}
	+
	\underbrace{
	\Vert (a-I_3) \pt^4 D\eta\Vert_1^2}_{\KK_{12}}
	+
	\underbrace{		\sum_{l=1}^3
	\Vert \pt^l a \pt^{4-l} D\eta\Vert_1^2}_{\KK_{13}}
	.
	\label{EQ599}
	\end{split}
\end{align}
The term $\KK_{11}$ is already estimated in Lemma~\ref{LnormalJ1}.
For the term $\KK_{12}$, using the Leibniz rule, we obtain
\begin{align}
	\begin{split}
	\KK_{12}
	\les	
	\Vert a-I_3 \Vert_{L^\infty}^2 
	\Vert\pt^4 D\eta\Vert_1^2
	+
	\Vert D^2 \eta \Vert_{L^\infty}^2 
	\Vert\pt^4 D\eta\Vert_0^2
		\les
1+ T P(\sup_{t\in [0,T]} E(t))
,
\label{EQ630}
	\end{split}
\end{align}
since $\Vert a-I_3\Vert_2 \les T$ and $\Vert \pt^3 v\Vert_2^2 + \Vert \pt^4 v\Vert_1^2 \leq E(t)$.
In the last inequality of \eqref{EQ630}, we also used the H\"older and Sobolev inequalities and the fundamental theorem of calculus.
Similarly, the term $\KK_{13}$ is estimated as
\begin{align*}
	\KK_{13}
	\les
	1+
	TP(\sup_{t \in [0,T]} E(t))
	.
\end{align*}
The term $\KK_2$ is already estimated in Lemma~\ref{Lcurl}.
For the term $\KK_3$, we appeal to Lemmas~\ref{trace}--\ref{normaltrace} and~\ref{Lenergy3}, obtaining
\begin{align*}
	\begin{split}
		\KK_3
		&\les
		\Vert \pt^4 \bp^2 \eta \cdot N \Vert_{H^{-0.5} (\partial\Omega)}
		+
		\Vert \pt^4 \bp \eta \cdot N \Vert_{L^2 (\partial \Omega)}
		\\&
		\les
		\Vert \pt^4 \bp^2 \eta\Vert_0^2
		+
		\Vert \dive \pt^4 \bp \eta\Vert_0^2
		+
		\delta \Vert \pt^4 \eta\Vert_2^2
		+
		C_\delta \Vert \pt^4 \eta\Vert_0^2
		\les	
		C_\delta
		+
		\delta \sup_{t\in [0,T]} E(t)
		+
		C_\delta T P(\sup_{t\in [0,T]} E(t))
		,
	\end{split}
\end{align*}
where we also used the Sobolev interpolation in the second inequality and \eqref{EQ599} in the last inequality.
We bound the term $\KK_4$ using the fundamental theorem of calculus as
\begin{align*}
	\begin{split}
		\KK_4
		\les
		1+T P(\sup_{t\in [0,T]} E(t))
		,
	\end{split}
\end{align*}
since $\Vert \pt^4 v\Vert_1^2 \leq E(t)$.
Consequently, we conclude the estimate
\begin{align}
	\Vert \pt^3 v\Vert_2^2
	\les
	C_\delta
	+
	\delta \sup_{t\in [0,T]} E(t)
	+
	C_\delta T P(\sup_{t\in [0,T]} E(t))
	.
	\label{EQ259}
\end{align}

Next we derive the estimate of $\Vert \pt^3 J\Vert_2^2$.
The estimate of $\Vert \pt^3 \bp^2 J\Vert_0^2$ is already established in Lemma~\ref{Lenergy3}.
For the estimate of $\Vert \pt^3 \bp \partial_{3} J\Vert_0^2$, we proceed analogously as in Lemma~\ref{LnormalJ1} by applying $\pt^3 \bp$ to \eqref{EQ100}, obtaining
\begin{align}
	\Vert \pt^3 \bp \partial_{3} J\Vert_0^2
	\les
		C_\delta
	+
	\delta \sup_{t\in [0,T]} E(t)
	+
	C_\delta T P(\sup_{t\in [0,T]} E(t))
	.
	\label{EQ260}
\end{align}
%
It remains to estimate $\Vert \pt^3 \partial_{33} J\Vert_0^2$.
Applying $\pt^3$ to \eqref{EQ100}, we obtain
\begin{align}
	\begin{split}
		&
		2 \rho_0 a^3_3
		\pt^3 J,_3
		=
		4
		a_3^3 \pt^3 J
		+
		J^3
		\big[ 
		\pt^5 \eta^3
		+
		\rho_0 \pt^3
		(a^\alpha_3 J^{-2},_\alpha)
		+
		\sum_{l=0}^2
		c_l                                         
		\pt^{l} 
		J^{-2} 
		\pt^{3-l} a^3_3
		\\&\indeq
		+
		\sum_{l=1}^3
		c_l  \rho_0
		\pt^l a^3_3
		\pt^{3-l} 
		J^{-2},_3
		+
		\sum_{l=1}^3
		c_l
		\rho_0 	a^3_3 
		\pt^l J^{-3} 
		\pt^{3-l} J,_3
		+
		\sum_{l=1}^2
		c_l
		a^3_3 \pt^l J^{-3} 
		\pt^{3-l} J
		\big]
		.
		\llabel{EQ643}
	\end{split}
\end{align}
Applying $\partial_{3}$ to the above equation, while noting that $\rho_0,_3 = -1$, we obtain
\begin{align}
	\begin{split}
		&
		2 \rho_0 a^3_3
		\pt^3 J,_{33}
		=
		6
		a^3_3
		\pt^3 J,_3
		-2 J^3
		\rho_0 (J^{-3} a^3_3),_3
		\pt^3 J,_3
		+
		4 J^3 (J^{-3} a^3_3),_3 \pt^3 J
		\\&\indeq
		+
		J^3
		\big[ 
		\pt^5 \eta^3
		+
		\rho_0 \pt^3
		(a^\alpha_3 J^{-2},_\alpha)
		+
		\sum_{l=0}^2
		c_l                                         
		\pt^{l} 
		J^{-2} 
		\pt^{3-l} a^3_3
		\\&\indeq
		+
		\sum_{l=1}^3
		c_l  \rho_0
		\pt^l a^3_3
		\pt^{3-l} 
		J^{-2},_3
		+
		\sum_{l=1}^3
		c_l
		\rho_0 	a^3_3 
		\pt^l J^{-3} 
		\pt^{3-l} J,_3
		+
		\sum_{l=1}^2
		c_l
		a^3_3 \pt^l J^{-3} 
		\pt^{3-l} J
		\big],_3
		.
		\label{EQ644}
	\end{split}
\end{align}
Now we estimate the terms on the right hand side of \eqref{EQ644} in $L^2$.
Using the H\"older and Sobolev inequalities and the fundamental theorem of calculus, we obtain
\begin{align*}
	\begin{split}
		\Vert a^3_3 \pt^3 J,_3 \Vert_0^2
		\les
		\Vert a^3_3\Vert_{L^\infty}^2
		\Vert \pt^3 J,_3\Vert_0^2
		\les
		1+TP(\sup_{t \in [0,T]} E(t)).
	\end{split}
\end{align*}
Similarly, we have
\begin{align*}
	\begin{split}
	\Vert J^3
	\rho_0 (J^{-3} a^3_3),_3
	\pt^3 J,_3\Vert_0^2
	\les
		1+TP(\sup_{t \in [0,T]} E(t))
	\end{split}
\end{align*}
and
\begin{align*}
	\begin{split}
		\Vert J^3 (J^{-3} a^3_3),_3 \pt^3 J \Vert_0^2
		\les
		1+TP(\sup_{t \in [0,T]} E(t))
	\end{split}
\end{align*}
The term $\Vert  J^3 \pt^5 \eta,_3 \Vert_0^2$ is already estimated in Lemma~\ref{LnormalJ1}. 
The highest order term in $J^3( \rho_0 \pt^3
(a^\alpha_3 J^{-2},_\alpha)),_3$ scales like $\pt^3 \bp J,_3$ which is estimated in \eqref{EQ260}.
The rest of the terms in $J^3( \rho_0 \pt^3
(a^\alpha_3 J^{-2},_\alpha)),_3$, as well as the terms on the right hand side of \eqref{EQ644} are of lower order which can estimated using the H\"older and Sobolev inequalities and the fundamental theorem of calculus.
Collecting the above estimates, we infer from \eqref{EQ644} that
\begin{align}
	\Vert \pt^3 J,_{33} \Vert_0^2
	\les
	C_\delta
	+
	\delta \sup_{t\in [0,T]} E(t)
	+
	C_\delta T P(\sup_{t\in [0,T]} E(t))
	.
	\label{EQ261}
\end{align}
The estimates of $\Vert \pt^2 v\Vert_3^2$ and $\Vert \pt^2 J\Vert_3^2$ follow in a similar fashion using the above arguments. Therefore, the proof of the lemma is completed by combining \eqref{EQ259} and \eqref{EQ261}.
\end{proof}

\cole
\begin{Lemma}
\label{LnormalJ3}
For $\delta \in (0,1)$, we have
\begin{align}
\begin{split}
	&
	\sup_{t\in [0,T]}
	\left(
	\Vert v (t)\Vert_{4}^2
	+
	\Vert \pt v (t)\Vert_{3}^2
	+
	\Vert \pt J (t)\Vert_{4}^2
	\right)
	\les
	C_\delta
	+
	\delta \sup_{t\in [0,T]} E(t)
	+
	C_\delta T P(\sup_{t\in [0,T]} E(t))
	,
	\label{EQ721}
\end{split}
\end{align}
where $C_\delta>0$ is a constant depending on $\delta$.
\end{Lemma}
\colb

\begin{proof}[Proof of Lemma~\ref{LnormalJ3}]
Using the fundamental theorem of calculus, we obtain
\begin{align*}
	\Vert \pt v\Vert_3^2
	\les
	1+TP(\sup_{t\in [0,T]} E(t))
	,
\end{align*}	
since $\Vert \pt^2 v\Vert_3^2 \leq E(t)$.
Next we estimate the term $\Vert v\Vert_4^2$.
From Lemma~\ref{elliptic} it follows that
\begin{align*}
	\begin{split}
		\Vert v\Vert_4^2
		\les
		\underbrace{
		\Vert \pt \dive \eta\Vert_3^2}_{\KK_1}
		+
		\underbrace{
		\Vert \pt \curl \eta\Vert_3^2}_{\KK_2}
		+
		\underbrace{	
		\Vert \pt \bp \eta \cdot N
		\Vert_{H^{2.5}(\partial \Omega)}^2}_{\KK_3}
		+
		\underbrace{
		\Vert \pt \eta\Vert_0^2}_{\KK_4}
		.
	\end{split}
\end{align*}
Note that
$
\dive \pt \eta
=
\pt J
+
(	\delta^s_r 
-
a^s_r)
\pt \eta^r,_s
$.
From the H\"older and Sobolev inequalities and the fundamental theorem of calculus it follows that
\begin{align}
	\begin{split}
	\KK_1
	&
	\les
	\Vert \pt J\Vert_3^2
	+
	\Vert (a-I_3) \pt D\eta\Vert_3^2
	\les
	\Vert \pt J\Vert_3^2
	+
	\Vert a-I_3 \Vert_{L^\infty}^2
	\Vert \pt D\eta\Vert_3^2
	+
	\Vert a-I_3 \Vert_3^2 
	\Vert \pt D\eta\Vert_{L^\infty}^2
	\\&
	\les
	1+T P(\sup_{t\in [0,T]} E(t))
	,
	\label{EQ699}
	\end{split}
\end{align}
since $\Vert a-I_3\Vert_2 \les T$ and $\Vert \pt^2 J\Vert_3^2 + \Vert v\Vert_4^2 \leq E(t)$.
The term $\KK_2$ is already estimated in Lemma~\ref{Lcurl}.
For the term $\KK_3$, we appeal to Lemmas~\ref{trace}--\ref{normaltrace} and \ref{Lenergy1}, we obtain
\begin{align*}
	\begin{split}
		\KK_3
		&\les
		\Vert \pt \bp^4 \eta \cdot N \Vert_{H^{-0.5} (\partial\Omega)}
		+
		\Vert \pt \bp \eta \cdot N \Vert_{L^2 (\partial \Omega)}
		\\&
		\les
		\Vert \pt \bp^4 \eta\Vert_0^2
		+
		\Vert  \pt \bp^3 \dive \eta\Vert_0^2
		+
		\Vert \pt \eta\Vert_2^2
		\les	
			C_\delta
		+
		\delta \sup_{t\in [0,T]} E(t)
		+
		C_\delta T P(\sup_{t\in [0,T]} E(t))
		,
	\end{split}
\end{align*}
where we used \eqref{EQ699} in the last inequality.
We bound the term $\KK_4$ using the fundamental theorem of calculus as
\begin{align*}
	\begin{split}
		\KK_4
		\les
		1+ T P(\sup_{t\in [0,T]} E(t))
		.	
	\end{split}
\end{align*}
Consequently, we conclude the estimate
\begin{align}
	\Vert v\Vert_4^2
	\les
	C_\delta
	+
	\delta \sup_{t\in [0,T]} E(t)
	+
	C_\delta T P(\sup_{t\in [0,T]} E(t))
	.
	\label{EQ677}
\end{align}

Finally, we derive the estimate of $\Vert \pt J\Vert_4^2$.
Applying $\pt D^3$ to \eqref{EQ100}, where $D= (\partial_1, \partial_2, \partial_3)$, we obtain
\begin{align*}
	\begin{split}
		&
		2\rho_0 J^{-3}  a^3_3 \pt D^3 J,_3
		=
		2 \pt D^3 (\rho_0 J^{-3} a^3_3) J,_3
		-
		2 J^{-3} \rho_0 a^\alpha_3 \pt D^3 J,_\alpha
		-
		2\pt D^3 (\rho_0 J^{-3} a^\alpha_3) J,_\alpha
		+
		D^3 \pt^3 \eta^3
		\\&\indeq
		-
		2J^{-2} \pt D^3 a^3_3
		-
		2  a^3_3 \pt D^3 J^{-2}
		+
		\sum_{l=0}^1
		\sum_{\substack{m=0\\1\leq l+m \leq 3}}^3
		c_{l,m} \pt^l D^m
		(\rho_0 J^{-3} a^3_3)
		\pt^{1-l} D^{3-m} J,_3
		\\&\indeq
		+
		\sum_{l=0}^1
		\sum_{\substack{m=0\\1\leq l+m \leq 3}}^3
		c_{l,m} \pt^l D^m
		(\rho_0 J^{-3} a^\alpha_3)
		\pt^{1-l} D^{3-m} J,_\alpha
		+
		\sum_{l=0}^1
		\sum_{\substack{m=0\\1\leq l+m \leq 3}}^3
		c_{l,m} \pt^l D^m a^3_3
		\pt^{1-l} D^{3-m} J^{-2}
		\\&
		=:
		\II_{1}
		+
		\II_2
		+
		\II_3
		+
		\II_4
		+
		\II_5
		+
		\II_6
		+
				\II_7
		+
		\II_8
		+
		\II_9
		.
	\end{split}
\end{align*}
The highest order term in $\II_{1}$ scales like $\rho_0 J,_3 D^3 \bp v$, which is estimated using H\"older and Sobolev inequalities and \eqref{EQ677}.
The rest of the terms in $\II_1$ are of lower order which can be estimated using the H\"older and Sobolev inequalities and the fundamental theorem of calculus. 
Thus, we have
\begin{align*}
	\begin{split}
	\Vert \II_1 \Vert_0^2
	\les
	C_\delta
	+
	\delta \sup_{t\in [0,T]} E(t)
	+
	C_\delta T P(\sup_{t\in [0,T]} E(t))
	.
	\end{split}
\end{align*}
For the term $\II_2$, we have
\begin{align*}
	\begin{split}
	\Vert \II_2 \Vert_0^2
	\les	
	\Vert a^3_\alpha \Vert_{L^\infty}^2
	\Vert \pt J\Vert_4^2
	\les
	1+T P(\sup_{t \in [0,T]} E(t)),
	\end{split}
\end{align*}
since $\Vert a-I_3\Vert_2 \les T$.
The highest order term in $\II_3$ is of the form $\rho_0 a^\alpha_3 J,_\alpha \pt D^3 J$ which can be estimated using the fundamental theorem of calculus, since $\Vert \pt^2 J\Vert_3^2 \leq E(t)$.
The rest of the terms in $\II_3$ are of lower order which can be estimated using the H\"older and Sobolev inequalities and the fundamental theorem of calculus.
Thus, we have
\begin{align*}
	\begin{split}
		\Vert \II_3 \Vert_0^2
		\les
		1+T P(\sup_{t \in [0,T]} E(t)).
	\end{split}
\end{align*}
The term $\II_4$ is already estimated in Lemma~\ref{LnormalJ2}, while the term $\II_5$ is estimated using H\"older and Sobolev inequalities and \eqref{EQ677}.
The term $\II_6$ is estimated using the fundamental theorem of calculus, since $\Vert \pt^2 J\Vert_3^2 \leq E(t)$.
The terms $\II_7$, $\II_8$, and $\II_9$ are of lower order which can be estimated using the H\"older and Sobolev inequalities and the fundamental theorem of calculus as
\begin{align*}
	\begin{split}
	\Vert \II_7 \Vert_0^2
	+
	\Vert \II_8 \Vert_0^2
	+
	\Vert \II_9 \Vert_0^2
	\les
	1+T P(\sup_{t \in [0,T]} E(t)).
	\end{split}
\end{align*}
Consequently, we conclude
\begin{align*}
	\Vert \pt J,_3 \Vert_3^2
	\les
	C_\delta
	+
	\delta \sup_{t\in [0,T]} E(t)
	+
	C_\delta T P(\sup_{t\in [0,T]} E(t))
	.
\end{align*}
The estimate of $\Vert \pt \bp J\Vert_3^2$ follows analogously using the above arguments.
The proof of the lemma is thus completed by combining \eqref{EQ677}.
\end{proof}

\startnewsection{The improved Lagrangian flow map and Jacobian estimates}{sec06}
The next two lemmas provide the estimates of the improved regularity of solutions $J$ and $\eta$.
\cole
\begin{Lemma}
\label{Limproved1}
For $\delta \in (0,1)$, we have
\begin{align}
	\begin{split}
		&
	\sup_{t\in [0,T]}
	\left(
	\Vert \bp^3 J (t)\Vert_{1.5}^2
	+
	\Vert \bp^3 \eta (t) \Vert_{1.5}^2
	\right)
	\les
	C_\delta
	+
	\delta \sup_{t\in [0,T]} E(t)
	+
	C_\delta T P(\sup_{t\in [0,T]} E(t))
	,
	\end{split}
\end{align}
where $C_\delta>0$ is a constant depending on $\delta$.
\end{Lemma}
\colb

\begin{proof}[Proof of Lemma~\ref{Limproved1}]
From Lemma~\ref{elliptic} it follows that
\begin{align*}
	\begin{split}
		\Vert \bp^3 \eta\Vert_{1.5}^2	
		\les
		\underbrace{ 
		\Vert \dive \bp^3 \eta \Vert_{0.5}^2}_{\JJ_1}
		+
		\underbrace{ 
			\Vert \curl \bp^3 \eta \Vert_{0.5}^2}_{\JJ_2}
		+
		\underbrace{
		\vert \bp^4 \eta \cdot N \vert_{0}^2}_{\JJ_3}
		+
		\underbrace{ 
		\Vert  \bp^3 \eta \Vert_{0}^2}_{\JJ_4}
		.
		\llabel{EQ622}
	\end{split}
\end{align*}
Note that from the Leibniz rule it follows that
$
 \bp^3 \dive \eta
 =
\bp^3 J
+
(\delta^s_r - a^s_r) 
\bp^{3}
\eta^r,_s
+
\sum_{l=1}^2
c_l
\bp^l
a^s_r
\bp^{3-l}
\eta^r,_s
$.
Thus, we have
\begin{align*}
	\begin{split}
	\JJ_1
	&
	\les
	\underbrace{
	\Vert \bp^3 J\Vert_{0.5}^2}_{\JJ_{11}}
	+
	\underbrace{ 
	\Vert (a- I_3) \bp^3 D\eta \Vert_{0.5}^2
	}_{\JJ_{12}}
	+
	\underbrace{
	\Vert
	\bp a
	\bp^{2} D \eta \Vert_{0.5}^2}_{\JJ_{13}}
	+
\underbrace{
	\Vert
	\bp^2 a
	\bp D \eta \Vert_{0.5}^2}_{\JJ_{14}}
	.
	\end{split}
\end{align*}
The term $\JJ_{11}$ is estimated using the fundamental theorem of calculus as
\begin{align*}
	\JJ_{11}
	\les
	1+TP(\sup_{t \in [0,T]} E(t))
	,
\end{align*}
since $\Vert \pt J\Vert_4^2 \leq E(t)$.
We bound the term $\JJ_{12}$ using the multiplicative Sobolev inequality \eqref{EQkap} as
\begin{align*}
	\begin{split}
	\JJ_{12}
	&
	\les
	\Vert a-I_3 \Vert_{1.5+\epsilon}^2
	\Vert \bp^3 D\eta\Vert_{0.5}^2
	\les T P(\sup_{t\in [0,T]} E(t))	
	,
	\end{split}
\end{align*}
since $\Vert a-I_3\Vert_2 \les T$.
For the term $\JJ_{13}$, we proceed analogously as $\JJ_{12}$, leading to 
\begin{align*}
	\begin{split}
		\JJ_{13}
		&
		\les
		\Vert \bp a \Vert_{1.5+\epsilon}^2
		\Vert \bp^2 D\eta\Vert_{0.5}^2
		\les
		1+ TP(\sup_{t \in [0,T]} E(t))	
		,
	\end{split}
\end{align*}
since $\Vert v\Vert_4^2 \leq E(t)$.
The term $\JJ_{14}$ is estimated using \eqref{EQkap} as
	\begin{align*}
		\begin{split}
			\JJ_{14}
			&
			\les
			\Vert \bp^2 a \Vert_{0.5}^2
			\Vert \bp D\eta\Vert_{1.5+\epsilon}^2
			\les
			1+ TP(\sup_{t \in [0,T]} E(t))	
			.
		\end{split}
\end{align*}
The term $\JJ_{2}$ is already established in Lemma~\ref{Lcurl}, while the term $\JJ_{3}$ is already estimated in Lemma~\ref{Lenergy1}.
We bound the term $\JJ_{4}$ using the fundamental theorem of calculus as
\begin{align*}
	\JJ_{4}
	\les
	1+T P(\sup_{t \in [0,T]} E(t))
	.
\end{align*}
Consequently, we conclude the estimate
\begin{align}
	\Vert \bp^3 \eta\Vert_{1.5}^2
	\les
		C_\delta
	+
	\delta \sup_{t\in [0,T]} E(t)
	+
	C_\delta T P(\sup_{t\in [0,T]} E(t))
	.
	\label{EQ290}
\end{align}

It remains to estimate the term $\Vert \bp^3 J\Vert_{1.5}^2$.
Applying $\bp^3$ to \eqref{EQ100}, we obtain
\begin{align}
	\begin{split}
	2 
	\bp^3 J,_3
	&
	=
	2 (a^3_3)^{-1}
	a^\alpha_3 \bp^3 J,_\alpha
	+
	\rho_0 ^{-1} (a^3_3)^{-1}
	J^3
	\big[ 
	\pt^2 \bp^3 \eta^3
	+
	\sum_{l=0}^3
	\pt^l a^3_3 \bp^{3-l} J^{-2}
	\\&\indeq
	+
	\sum_{l=1}^3
	c_l \bp^l (\rho_0 J^{-3} a^3_3)
	\bp^{3-l} J,_3
	+
	\sum_{l=1}^{3}
	c_l \bp^l (\rho_0 J^{-3} a^\alpha_3)
	\bp^{3-l} J,_\alpha
	\big]
	.
	\label{EQ549}
	\end{split}
\end{align}
Now we estimate the terms on the right hand side of \eqref{EQ549} in $H^{0.5}$.
Using \eqref{EQkap}, we obtain
\begin{align*}
	\begin{split}
	\Vert (a^3_3)^{-1}
	a^\alpha_3 \bp^3 J,_\alpha
	\Vert_{0.5}^2
	&
	\les
	\Vert  (a^3_3)^{-1} 
	a^\alpha_3\Vert_{2}^2
	\Vert \bp^3 \bp J\Vert_{0.5}^2
	\les
		\Vert a^\alpha_3\Vert_{2}^2
	\Vert \bp^3 J\Vert_{1.5}^2
	\les
	TP(\sup_{t \in [0,T]} E(t)),
	\end{split}
\end{align*}
since $\Vert a- I_3\Vert_2 \les T$, where we used the fact that $H^2$ is an algebra.
Similarly, we have
\begin{align*}
	\begin{split}
		\Vert \rho_0^{-1} (a^3_3)^{-1} 
		J^3 \pt^2 \bp^3 \eta^3 \Vert_{0.5}^2
		&
		\les
		\Vert \rho_0^{-1} (a^3_3)^{-1} 
		J^3 \Vert_{2}^2
		\Vert \pt^2 \bp^3 \eta \Vert_{0.5}^2
		\les
		\Vert \pt^2 \bp^3 \eta \Vert_{0.5}^2
		.
	\end{split}
\end{align*}
From Lemmas~\ref{hardy} and~\ref{Lenergy2} it follows that
\begin{align*}
	\begin{split}
	\Vert  \pt^2 \bp^3 \eta^3
	\Vert_{0.5}^2
	&
	\les
	\int_\Omega
	\dd (|\pt^2 \bp^3 D\eta|^2 
	+
	|\pt^2 \bp^3 \eta|^2 
	)
	\les
	\Vert \rhog \pt^2 \bp^3 D\eta\Vert_0^2
	+
	\Vert \pt v\Vert_3^2
	\\&
	\les
	C_\delta
	+
	\delta \sup_{t\in [0,T]} E(t)
	+
	C_\delta T P(\sup_{t\in [0,T]} E(t))
	,
	\end{split}
\end{align*}
where we have crucially used that
\begin{align*}
	\rhog^2=
	(1+w)^2
	-(1+w)
	=w+w^2 \geq w
	.
\end{align*}
The rest of the terms on the right hand side of \eqref{EQ549} are of lower order which can estimated analogously using the above arguments.
Thus, we infer that
\begin{align*}
	\Vert \bp^3 J,_3\Vert_{0.5}^2
	\les
	C_\delta
	+
	\delta \sup_{t\in [0,T]} E(t)
	+
	C_\delta T P(\sup_{t\in [0,T]} E(t))
	.
\end{align*}
Similar arguments show that
\begin{align*}
	\Vert \bp^3 \bp J\Vert_{0.5}^2
	\les
	C_\delta
	+
	\delta \sup_{t\in [0,T]} E(t)
	+
	C_\delta T P(\sup_{t\in [0,T]} E(t))
	.
\end{align*}
Therefore, we complete the proof of the Lemma by combining \eqref{EQ290}.
\end{proof}

\cole
\begin{Lemma}
\label{Limproved2}
For $\delta \in (0,1)$, we have
\begin{align}
	\begin{split}
	&
	\sup_{t\in [0,T]}
	\sum_{l=2}^4
	\Vert \bp^{4-l} \pt^l \eta  (t)\Vert_{1.5}^2
	\les
	C_\delta
	+
	\delta \sup_{t\in [0,T]} E(t)
	+
	C_\delta T P(\sup_{t\in [0,T]} E(t))
	\end{split}
\end{align}
where $C_\delta>0$ is a constant depending on $\delta$.
\end{Lemma}
\colb

The proof of Lemma~\ref{Limproved2} is analogously to the proof of Lemma~\ref{Limproved1} by using Lemma~\ref{Lenergy3} and thus we omit the details.

\startnewsection{The improved curl estimates}{sec07}
\cole
\begin{Lemma}
\label{improvedcurl}
For $\delta \in (0,1)$, we have
\begin{align}
\begin{split}
	&
	\sup_{t\in [0,T]}
	\Vert \bp^3 \curl_\eta v(t)\Vert_{0.5}^2
	\les
	C_\delta
	+
	\delta \sup_{t\in [0,T]} E(t)
	+
	C_\delta T P(\sup_{t\in [0,T]} E(t))
	,
	\label{EQ607}
\end{split}
\end{align}
where $C_\delta>0$ is a constant depending on $\delta$.
\end{Lemma}
\colb

\begin{proof}[Proof of Lemma~\ref{improvedcurl}]
Applying $\bp^3$ to \eqref{EQ53}, we obtain
\begin{align*}
	\begin{split}
	\Vert \bp^3 \curl_\eta v \Vert_{0.5}^2
	&
	\les
	\underbrace{\Vert \bp^3 \curl u_0\Vert_{0.5}^2}_{\II_1}
	+
	\underbrace{\Vert \int_0^t \bp^3 
	Q_0 (A(t'), Dv(t')) dt'\Vert_{0.5}^2}_{\II_2}
	.
	\end{split}
\end{align*}
The term $\II_1$ is bounded by $M_0$ since $\Vert \bp^3 \curl v (0)\Vert_{0.5} = \Vert \bp^3 \curl u_0\Vert_{0.5}$.
%
%
From \eqref{EQ55}, we infer that the highest order term in $\II_2$ can be written as
\begin{align*}
	- \epsilon_{kji}
	[\bp^3 v^i,_s A^m_j A^s_r v^r,_m 
	+
	\bp^3 v^r,_m A^m_j  A^s_r  v^i,_s 
	]
	.
	\llabel{EQ605}
\end{align*}
Using integration by parts in time, where we drop the indices for simplicity, we obtain
\begin{align*}
	\Vert
	\int_0^t 
	\bp^3 Dv D v AA  dt'
	\Vert_{0.5}^2
	\les
	\underbrace{	\Vert
	\int_0^t 
	\bp^3 D\eta (D v AA)_t  dt' \Vert_{0.5}^2}_{\II_{21}}
	+
	 \underbrace{	\Vert
	 	\bp^3 D\eta Dv AA 
 	\Vert_{0.5}^2 (t)}_{\II_{22}}
	.
\end{align*}
For the term $\II_{21}$, we appeal to the multiplicative Sobolev inequality \eqref{EQkap}, obtaining
\begin{align*}
	\begin{split}
		\II_{21}	
		&
		\les
		\int_0^T
		\Vert \bp^3 D \eta\Vert_{0.5}^2
		\Vert (Dv AA)_t \Vert_{2}^2
		\les
		T P(\sup_{t\in [0,T]} E(t))
		.
		\llabel{EQ604}
	\end{split}
\end{align*}
Similarly, we bound the term $\II_{32}$ as
\begin{align*}
	\begin{split}
		\II_{22}	
		&
		\les
		\Vert \bp^3 D \eta\Vert_{0.5}^2
		\Vert Dv AA \Vert_{2}^2
		\les
		C_\delta
		+
		\delta \sup_{t\in [0,T]} E(t)
		+
		C_\delta T P(\sup_{t\in [0,T]} E(t))
		,
		\llabel{EQ606}
	\end{split}
\end{align*}
where we appealed to Lemma~\ref{Limproved1} in the last inequality.
The rest of the terms in $\II_{2}$ are of lower order which can be estimated using the H\"older and Sobolev inequalities and the fundamental theorem of calculus. 
Consequently, we conclude the proof of the lemma.
\end{proof}

\startnewsection{Concluding the proof of Proposition~\ref{Lapriori}}{sec08}
We sum the estimates in Lemmas~\ref{Lcurl}, ~\ref{Lenergy1}--~\ref{Lenergy3}, ~\ref{LnormalJ1}--\ref{LnormalJ3}, \ref{Limproved1}--\ref{Limproved2}, and~\ref{improvedcurl},
thus completing the proof of Proposition~\ref{Lapriori}.
The proof of Theorem~\ref{T01} is thus concluded.

\startnewsection{The general case $\gamma> 1$}{sec09}
In this section we prove Theorem~\ref{T01} for general $\gamma> 1$.
The proof is similar to those in Sections~\ref{sec03}--\ref{sec08}, thus we only outline the necessary modifications and omit further details.

The Lagrangian curl of \eqref{EQ04} equals zero for any $\gamma>1$. 
Consequently, similar estimates in Section~\ref{sec03} can be carried out.
From \eqref{EQ185}, we may assume that
\begin{align*}
	-J^{-1} 
	(\rhol^\gamma J^{-\gamma}),_3 
	=
	-J^{-1} 
	(\rhol^\gamma,_3 J^{-\gamma} 
	-
	\gamma \rhol^\gamma
	J^{-\gamma-1} J,_3
	 ) 
	\geq 
	\frac{\gamma}{4}
	,
	\llabel{EQ484}
\end{align*}
on $[0,T] \times \Gamma$.
Thus, using the decomposition $\rho_0^\gamma = \rhog^\gamma + \rhol^\gamma$ (see \eqref{EQ184}--\eqref{EQ185}), similar energy estimates in Section~\ref{sec04} hold for $\gamma>1$.
The normal derivative estimates in Section~\ref{sec05} relies on the equation \eqref{EQ200}.
With the additional term $\gamma (\rho_0 J^{-1})^{\gamma-2}$, we get the same normal derivative estimates using the energy estimates.
From \eqref{EQ184} it is readily checked that
\begin{align*}
	\rhog^\gamma (x)
	\geq 
	\frac{\gamma w(x)}{2}
	\comma
	x\in \bar{\Omega}
	,
\end{align*}
from where we obtain
\begin{align*}
	\Vert F\Vert_{0.5}^2
	\les
	\int_\Omega 
	\dd (|DF|^2 +|F|^2)
	\les
	\Vert \rhog^{\frac{\gamma}{2}} DF\Vert_0^2
	+
	\Vert F\Vert_0^2
	.
\end{align*}
Therefore, the improved regularity in Section~\ref{sec06} can also be carried out using the above Hardy-type inequality.
In Section~\ref{sec07}, the improved curl estimates hold for $\gamma >1$.
Consequently, we conclude the proof of Proposition~\ref{Lapriori} and thus completing the proof of Theorem~\ref{T01}.

\section*{Acknowledgments}
The author was supported in part by the NSF grants DMS-2009458 and DMS-2205493.
The author is grateful to M.~Had\v zi\'c, J.~Jang, and I.~Kukavica for fruitful discussions.
\appendix

\section{Hardy-Sobolev embeddings, trace estimates, and Hodge-type bounds}
\label{secA}


In this section we provide some auxiliary results.
We first recall the following Hardy type inequality (cf.~\cite{KMP}).

\cole
\begin{Lemma}[Hardy inequality]
\label{hardy}
For any given $a>0$ and nonnegative integer $b \geq a/2$, we have
\begin{align*}
	\Vert F\Vert_{b-\frac{a}{2}}^2
	\leq
	C\sum_{k=0}^b
	 \int_\Omega
	\dd(x)^a
	|D^k F(x)|^2
	dx
	,
	\llabel{EQ543}
\end{align*}
for some constant $C>0$ depending on $a$, $b$, and $\Omega$.
In particular, let $p=1,2$, then 
\begin{align*}
	\Vert F\Vert_{1-\frac{p}{2}}^2
	\leq
	C \int_\Omega
	\dd(x)^p 
	\left(
	|DF (x)|^2 
	+ 
	|F(x)|^2
	\right)
	dx
	,
	\llabel{EQ211}
\end{align*}
for some constant $C>0$ depending on $\Omega$ and $p$.
\end{Lemma}
\colb

The following lemma provides a Hodge type elliptic estimates (cf.~\cite{Ta}).
\cole
\begin{Lemma}[Hodge type estimate]
\label{elliptic}
Let $F$ be a vector field over $\Omega$ and $s \geq 1$.
Then we have
\begin{align}
	\Vert F\Vert_s
	\leq
	C 
	\left(
	\Vert F\Vert_{0}
	+
	\Vert \dive F\Vert_{s-1}
	+
	\Vert \curl F\Vert_{s-1}
	+
	\Vert \bp F \cdot N \Vert_{H^{s-1.5}(\partial \Omega)}
	\right)
	,
	\label{EQ227}
\end{align}
for some constant $C>0$, where $N$ denotes the outward unit normal to $\partial \Omega$.
\end{Lemma}
\colb

The estimate \eqref{EQ227} follows from the well-known identity $-\Delta F = \curl \curl F - \nabla \dive F$ and the analysis of second order elliptic operators.

We recall from \cite{A} the following Sobolev trace theorem.
\cole
\begin{Lemma}
	\label{trace}
	Let $F\in H^s (\Omega)$ where $s>\frac{1}{2}$. Then we have
	\begin{align*}
		\Vert F \Vert_{H^{s-\frac{1}{2}} (\partial \Omega)}^2
		\leq
		C 
		\Vert F \Vert_s^2
		,
	\end{align*}
	for some constant $C>0$.
\end{Lemma}
\colb

The next lemma provide the normal trace estimate which is needed in the elliptic estimates (cf.~\cite{A,CCS}).

\cole
\begin{Lemma}[Normal trace]
\label{normaltrace}
Let $F \in L^2 (\Omega)$ and $\dive F \in L^2 (\Omega)$.
Then 
\begin{align*}
	\Vert \bp F \cdot N 
	\Vert_{H^{-0.5}(\partial \Omega)}^2
	\leq
	C \left(
	\Vert \bp F\Vert_0^2
	+
	 \Vert \dive F\Vert_{0}^2
	 \right)
	 ,
\end{align*}
for some constant $C>0$, where $N$ denotes the outward unit normal vector to $\partial \Omega$.
\end{Lemma}
\colb
%

We recall from \cite{CS1} the following 
duality inequality.
\cole
\begin{Lemma}[Duality estimate]
\label{duality}
Let $F \in H^{0.5} (\Omega)$.
Then we have
\begin{align*}
	\Vert \bp F\Vert_{-0.5}^2
	\leq
	C \Vert F\Vert_{0.5}^2
	,
\end{align*}
for some constant $C>0$.
\end{Lemma}
\colb


\end{document}